\documentclass[11pt]{amsart}
\usepackage{}
\usepackage{amssymb}
\usepackage{amsfonts}
\usepackage{mathrsfs}
\usepackage{latexsym}
\usepackage{graphicx}
\usepackage{amscd,amssymb,amsmath,amsbsy,amsthm,amsfonts}
\usepackage[all]{xy}
\usepackage[colorlinks,plainpages,backref,urlcolor=blue]{hyperref}
\usepackage{verbatim}
\usepackage{enumerate}

\usepackage [english]{babel}
\usepackage [autostyle, english = american]{csquotes}
\MakeOuterQuote{"}

\newcommand {\Omit}[1]{}

\usepackage{tikz-cd}

\usepackage{tikz}
\usetikzlibrary{automata,positioning}

\usetikzlibrary{arrows,calc}
\tikzset{
>=stealth',
help lines/.style={dashed, thick},
axis/.style={<->},
important line/.style={thick},
connection/.style={thick, dotted},
}
\usetikzlibrary{patterns}
\newlength{\hatchspread}
\newlength{\hatchthickness}
\tikzset{hatchspread/.code={\setlength{\hatchspread}{#1}},
         hatchthickness/.code={\setlength{\hatchthickness}{#1}}}
\tikzset{hatchspread=3pt,
         hatchthickness=0.4pt}
\pgfdeclarepatternformonly[\hatchspread,\hatchthickness]% variables
   {custom north west lines}% name
   {\pgfqpoint{-2\hatchthickness}{-2\hatchthickness}}% lower left corner
   {\pgfqpoint{\dimexpr\hatchspread+2\hatchthickness}{\dimexpr\hatchspread+2\hatchthickness}}% upper right corner
   {\pgfqpoint{\hatchspread}{\hatchspread}}% tile size
   {% shape description
    \pgfsetlinewidth{\hatchthickness}
    \pgfpathmoveto{\pgfqpoint{0pt}{\hatchspread}}
    \pgfpathlineto{\pgfqpoint{\dimexpr\hatchspread+0.15pt}{-0.15pt}}
    \pgfusepath{stroke}
   }
\newcommand{\nc}{\newcommand}
\nc{\rnc}{\renewcommand}
\nc{\bb}[1]{{\mathbb #1}}
\nc{\bbA}{\bb{A}}\nc{\bbB}{\bb{B}}\nc{\bbC}{\bb{C}}\nc{\bbD}{\bb{D}}
\nc{\bbE}{\bb{E}}\nc{\bbF}{\bb{F}}\nc{\bbG}{\bb{G}}\nc{\bbH}{\bb{H}}
\nc{\bbI}{\bb{I}}\nc{\bbJ}{\bb{J}}\nc{\bbK}{\bb{K}}\nc{\bbL}{\bb{L}}
\nc{\bbM}{\bb{M}}\nc{\bbN}{\bb{N}}\nc{\bbO}{\bb{O}}\nc{\bbP}{\bb{P}}
\nc{\bbQ}{\bb{Q}}\nc{\bbR}{\bb{R}}\nc{\bbS}{\bb{S}}\nc{\bbT}{\bb{T}}
\nc{\bbU}{\bb{U}}\nc{\bbV}{\bb{V}}\nc{\bbW}{\bb{W}}\nc{\bbX}{\bb{X}}
\nc{\bbY}{\bb{Y}}\nc{\bbZ}{\bb{Z}}
\nc{\mbf}[1]{{\mathbf #1}}
\nc{\bfA}{\mbf{A}}\nc{\bfB}{\mbf{B}}\nc{\bfC}{\mbf{C}}\nc{\bfD}{\mbf{D}}
\nc{\bfE}{\mbf{E}}\nc{\bfF}{\mbf{F}}\nc{\bfG}{\mbf{G}}\nc{\bfH}{\mbf{H}}
\nc{\bfI}{\mbf{I}}\nc{\bfJ}{\mbf{J}}\nc{\bfK}{\mbf{K}}\nc{\bfL}{\mbf{L}}
\nc{\bfM}{\mbf{M}}\nc{\bfN}{\mbf{N}}\nc{\bfO}{\mbf{O}}\nc{\bfP}{\mbf{P}}
\nc{\bfQ}{\mbf{Q}}\nc{\bfR}{\mbf{R}}\nc{\bfS}{\mbf{S}}\nc{\bfT}{\mbf{T}}
\nc{\bfU}{\mbf{U}}\nc{\bfV}{\mbf{V}}\nc{\bfW}{\mbf{W}}\nc{\bfX}{\mbf{X}}
\nc{\bfY}{\mbf{Y}}\nc{\bfZ}{\mbf{Z}}
\nc{\bfa}{\mbf{a}}\nc{\bfb}{\mbf{b}}\nc{\bfc}{\mbf{c}}\nc{\bfd}{\mbf{d}}
\nc{\bfe}{\mbf{e}}\nc{\bff}{\mbf{f}}\nc{\bfg}{\mbf{g}}\nc{\bfh}{\mbf{h}}
\nc{\bfi}{\mbf{i}}\nc{\bfj}{\mbf{j}}\nc{\bfk}{\mbf{k}}\nc{\bfl}{\mbf{l}}
\nc{\bfm}{\mbf{m}}\nc{\bfn}{\mbf{n}}\nc{\bfo}{\mbf{o}}\nc{\bfp}{\mbf{p}}
\nc{\bfq}{\mbf{q}}\nc{\bfr}{\mbf{r}}\nc{\bfs}{\mbf{s}}\nc{\bft}{\mbf{t}}
\nc{\bfu}{\mbf{u}}\nc{\bfv}{\mbf{v}}\nc{\bfw}{\mbf{w}}\nc{\bfx}{\mbf{x}}
\nc{\bfy}{\mbf{y}}\nc{\bfz}{\mbf{z}}

\newcommand{\G}{\mathbb{G}}
\newcommand{\op}{\text{op}}

\nc{\mcal}[1]{{\mathcal #1}}
\nc{\calA}{\mcal{A}}\nc{\calB}{\mcal{B}}\nc{\calC}{\mcal{C}}\nc{\calD}{\mcal{D}}
\nc{\calE}{\mcal{E}} \nc{\calF}{\mcal{F}}\nc{\calG}{\mcal{G}}\nc{\calH}{\mcal{H}}
\nc{\calI}{\mcal{I}}\nc{\calJ}{\mcal{J}}\nc{\calK}{\mcal{K}}\nc{\calL}{\mcal{L}}
\nc{\calM}{\mcal{M}}\nc{\calN}{\mcal{N}}\nc{\calO}{\mcal{O}}\nc{\calP}{\mcal{P}}
\nc{\calQ}{\mcal{Q}}\nc{\calR}{\mcal{R}}\nc{\calS}{\mcal{S}}\nc{\calT}{\mcal{T}}
\nc{\calU}{\mcal{U}}\nc{\calV}{\mcal{V}}\nc{\calW}{\mcal{W}}\nc{\calX}{\mcal{X}}
\nc{\calY}{\mcal{Y}}\nc{\calZ}{\mcal{Z}}
\nc{\fA}{\frak{A}}\nc{\fB}{\frak{B}}\nc{\fC}{\frak{C}} \nc{\fD}{\frak{D}}
\nc{\fE}{\frak{E}}\nc{\fF}{\frak{F}}\nc{\fG}{\frak{G}}\nc{\fH}{\frak{H}}
\nc{\fI}{\frak{I}}\nc{\fJ}{\frak{J}}\nc{\fK}{\frak{K}}\nc{\fL}{\frak{L}}
\nc{\fM}{\frak{M}}\nc{\fN}{\frak{N}}\nc{\fO}{\frak{O}}\nc{\fP}{\frak{P}}
\nc{\fQ}{\frak{Q}}\nc{\fR}{\frak{R}}\nc{\fS}{\frak{S}}\nc{\fT}{\frak{T}}
\nc{\fU}{\frak{U}}\nc{\fV}{\frak{V}}\nc{\fW}{\frak{W}}\nc{\fX}{\frak{X}}
\nc{\fY}{\frak{Y}}\nc{\fZ}{\frak{Z}}
\nc{\fa}{\frak{a}}\nc{\fb}{\frak{b}}\nc{\fc}{\frak{c}} \nc{\fd}{\frak{d}}
\nc{\fe}{\frak{e}}\nc{\fFf}{\frak{f}}\nc{\fg}{\frak{g}}\nc{\fh}{\frak{h}}
\nc{\fri}{\frak{i}}\nc{\fj}{\frak{j}}\nc{\fk}{\frak{k}}\nc{\fl}{\frak{l}}
\nc{\fm}{\frak{m}}\nc{\fn}{\frak{n}}\nc{\fo}{\frak{o}}\nc{\fp}{\frak{p}}
\nc{\fq}{\frak{q}}\nc{\fr}{\frak{r}}\nc{\fs}{\frak{s}}\nc{\ft}{\frak{t}}
\nc{\fu}{\frak{u}}\nc{\fv}{\frak{v}}\nc{\fw}{\frak{w}}\nc{\fx}{\frak{x}}
\nc{\fy}{\frak{y}}\nc{\fz}{\frak{z}}

\newtheorem{theorem}{Theorem}[section]
\newtheorem{lemma}[theorem]{Lemma}
\newtheorem{corollary}[theorem]{Corollary}
\newtheorem{prop}[theorem]{Proposition}

\newtheorem{assumption}[theorem]{Assumption}

\theoremstyle{definition}
\newtheorem{definition}[theorem]{Definition}
\newtheorem{example}[theorem]{Example}
\newtheorem{remark}[theorem]{Remark}

\newtheorem{thm}{Theorem}

\DeclareMathOperator{\rank}{rank} \DeclareMathOperator{\gr}{gr}
\DeclareMathOperator{\im}{im} 
 \DeclareMathOperator{\id}{id}
\DeclareMathOperator{\Image}{Im} \DeclareMathOperator{\Sym}{Sym}

 \DeclareMathOperator{\GL}{GL}
\DeclareMathOperator{\Hom}{{Hom}} \DeclareMathOperator{\Tor}{{Tor}}

\DeclareMathOperator{\sHom}{{\mathscr{H}om}}

\DeclareMathOperator{\Hilb}{{Hilb}}

 \DeclareMathOperator{\Lie}{Lie}
\DeclareMathOperator{\Spec}{{Spec}} \DeclareMathOperator{\tr}{tr}

\DeclareMathOperator{\Grass}{Grass} \DeclareMathOperator{\End}{End}

\DeclareMathOperator{\Gm}{\bbG_m}
 
  \DeclareMathOperator{\Fl}{Fl}

\newcommand{\sph}{\fs}

  \DeclareMathOperator{\prepr}{prepr}

\DeclareMathOperator{\Sch}{\mathbf{Sch}}

\DeclareMathOperator{\diag}{diag}

\DeclareMathOperator{\Laz}{\bbL az}

\DeclareMathOperator{\Gr}{Gr}

\DeclareMathOperator{\CH}{CH}

\DeclareMathOperator{\MO}{MO}

\DeclareMathOperator{\inc}{in}
\DeclareMathOperator{\out}{out}

\DeclareMathOperator{\Rep}{Rep}
\DeclareMathOperator{\Res}{Res}

\DeclareMathOperator{\Sh}{Sh}
\DeclareMathOperator{\For}{f}

\newcommand{\surj}{\twoheadrightarrow}
\newcommand{\inj}{\hookrightarrow}

\newcommand{\pt}{\text{pt}}
\newcommand{\Aff}{\bbA}

\newcommand{\QQ}{\bbQ}
\newcommand{\Z}{\bbZ}
\newcommand{\C}{\bbC}

\newcommand{\Q}{\bbQ}
\newcommand{\N}{\bbN}

\DeclareMathOperator{\fac}{fac}
\DeclareMathOperator{\pr}{pr}

 \newcommand{\ext}{\fe}

\DeclareMathOperator{\fin}{fin}

\newcount\cols
{\catcode`,=\active\catcode`|=\active
 \gdef\Young(#1){\hbox{$\vcenter
 {\mathcode`,="8000\mathcode`|="8000
  \def,{\global\advance\cols by 1 &}%
  \def|{\cr
        \multispan{\the\cols}\hrulefill\cr
        &\global\cols=2 }%
  \offinterlineskip\everycr{}\tabskip=0pt
  \dimen0=\ht\strutbox \advance\dimen0 by \dp\strutbox
  \halign
   {\vrule height \ht\strutbox depth \dp\strutbox##
    &&\hbox to \dimen0{\hss$##$\hss}\vrule\cr
    \noalign{\hrule}&\global\cols=2 #1\crcr
    \multispan{\the\cols}\hrulefill\cr%
   }
 }$}}
}

\begin{document}
\title[Preprojective Cohomological Hall algebra]
{The cohomological Hall algebra of a preprojective
algebra}
\date{\today}

\author[Y.~Yang]{Yaping~Yang}
\address{School of Mathematics and Statistics, The University of Melbourne, 813 Swanston Street, Parkville VIC 3010, Australia}
\email{yaping.yang1@unimelb.edu.au}

\author[G.~Zhao]{Gufang~Zhao}
\address{Max-Planck-Institut f\"ur Mathematik,
Vivatsgasse 7,
53111 Bonn,
Germany}
\curraddr{Institute of Science and Technology Austria,
Am Campus, 1,
Klosterneuburg 3400,
Austria}
\email{gufang.zhao@ist.ac.at}


\dedicatory{Dedicated to Professor Jerzy Weyman on the occasion of his 60th birthday.}

\begin{abstract}
We introduce for each quiver $Q$ and each algebraic oriented cohomology theory $A$,
the cohomological Hall algebra (CoHA) of  $Q$, as 
the $A$-homology of the moduli of representations of the preprojective algebra of $Q$. This generalizes the $K$-theoretic Hall algebra of commuting varieties defined by Schiffmann-Vasserot \cite{SV2}. 
When $A$ is the Morava $K$-theory, we show evidence that this algebra is a candidate for Lusztig's reformulated conjecture on modular representations of algebraic groups \cite{Lusz}.

We construct an action of the preprojective CoHA on the $A$-homology of Nakajima quiver varieties. 
We compare this with the action of the Borel subalgebra of Yangian when $A$ is the intersection theory. 
We also give a shuffle algebra description of this CoHA in terms of the underlying formal group law of $A$.  
As applications, we obtain a shuffle description of the Yangian.
\end{abstract}
\maketitle
\tableofcontents
\section{Introduction}\label{sec:intro}
Let $Q$ be a quiver, and $\fg_Q$ the corresponding symmetric Kac-Moody Lie algebra. The Nakajima quiver varieties of $Q$ are fine moduli spaces parametrizing stable framed representations of the preprojective algebra of $Q$  \cite{Nak94}. 
They play an important role in constructing representations of various quantum groups associated to $\fg_{Q}$.
 
When $Q$ is a quiver without edge loops, 
 Nakajima constructed an action of the quantum loop algebra  $U_{q}(L\fg_Q)$ on the equivariant $K$-theory of quiver varieties \cite{Nak99}. 
He realized the Drinfeld generators of $U_{q}(L\fg_Q)$ as explicit convolution operators, and showed that they satisfy  the required commutation relations.
Following a similar method,  Varagnolo constructed an action of the Yangian $Y_{\hbar}(\fg_{Q})$ on the equivariant Borel-Moore homology of quiver varieties \cite{Va00}. Based on the pattern, the elliptic quantum group should act on the equivariant elliptic cohomology of quiver varieties, which will be carried out in \cite{Z15} based on the techniques developed in the present paper. 
 
The main goal of the present paper is to show that for any quiver $Q$ and  an arbitrary algebraic oriented cohomology theory  $A$ in the sense of Levine-Morel \cite{LM}, there is an affine quantum group associated to $\fg_Q$ and $A$, which acts on the $A$-homology of quiver varieties. 

There had been quite a few evidence that such an affine quantum group should exist, notably:
\begin{enumerate}
\item As observed by Drinfeld, the dual of any quantum group in the sense of \cite[\S~7]{Dr} is a quantized formal series bialgebra. Hence associate to each quantum group there is a Lie algebra and a formal group law. 
However, there had been no construction the other direction, although for the formal group laws coming from 1-dimensional algebraic groups the corresponding affine quantum groups were known. 
\item As observed in \cite{GKV95}, there is a parallelism among three different mathematical objects: affine quantum groups, oriented cohomology theories, and formal group laws. The correspondence between the last two is well-known, and Nakajima's construction is an evidence of a direct link between the first two. 
\end{enumerate}
Yangian and quantum loop algebra have known Drinfeld-type presentations. In \cite{Nak99,Va00}, Nakajima and Varagnolo constructed operators by which the generators of Yangian and quantum loop algebra act on quiver varieties, and verified their relations.  Consequently, the representations of these algebras are constructed geometrically.
However,  the affine quantum groups associated to other oriented cohomology theories do not have known presentations, hence counterparts of the operators used in {\it loc. cit.} have no prescribed commutation relations, i.e., the method in {\it loc. cit.} does not provide a construction of these algebras themselves.

The present paper gives the first construction, geometrically or algebraically, of quantum groups in the affine setting beyond the Yangians, quantum loop algebras, and elliptic quantum groups. These new affine quantum groups have equality interesting properties, some of which are not available to the previously known cases:
\begin{enumerate}
\item Examples when $A$ is the connective $K$-theory, the Eilenberg-MacLane spectrum $H\bbZ/p$,  or the algebraic cobordism can be found in \S~\ref{subsec:examples}. 
\item In particular, Lusztig proposed  the existence  of a family of quantum groups depending on a prime number $p$ and an integer $n$, with conjectural properties \cite{Lusz}. When $A$ is the Morava $K$-theory, we get such a family with evidence of satisfying the desired properties, including a stability property prescribed by Lusztig (Proposition~\ref{prop:stable}). A complete proof in the case when $n=0$ follows from the present paper. 
\end{enumerate}

We construct an algebra, called the preprojective cohomological Hall algebra (CoHA), as well as an action of it on the $A$-homology of Nakajima quiver varieties associated to $Q$, following an approach similar to  \cite{SV2}. The Nakajima-type raising operators are lifted as elements in this CoHA.
We also construct a suitable extension of  the preprojective CoHA,  which roughly speaking is a quantization of a central extension of $U(\mathfrak{b}_Q[\![u]\!])$, where the quantization depends on the underlying formal group law of $A$. 
We study in detail the test case when the formal group law is additive, and we show in this case the quantization is the Borel subalgebra of the Yangian $Y_\hbar(\fg_Q)$.

To describe the multiplication of the preprojective CoHA algebraically, we also give a shuffle formula. Historically, a shuffle description was not only used as an intrinsic way of defining quantum groups \cite{Ros}, but also provided interesting combinatorial information. For example, PBW property of quantum groups becomes transparent through the shuffle algebra (see, e.g., \cite{Ros2}), and the dual canonical basis has a purely algebraic description in terms of it \cite{Lec}. 
In this paper, we give a geometric description of the shuffle algebra. The commutation relations among Nakajima-type raising operators can be calculated using the shuffle formula. 

The construction in the present paper follows similar idea as \cite{SV2}. 
When $Q$ is the Jordan quiver, the Nakajima quiver variety is isomorphic to $\Hilb^n(\Aff^2)$, the Hilbert scheme of $n$ points on $\Aff^2$. Schiffmann-Vasserot constructed an action of the elliptic Hall algebra on the 
equivariant $K$-theory of $\coprod_{n\in\bbN}\Hilb^n(\Aff^2)$ \cite{SV2}. Feigin-Tsymbaliuk independently constructed an action of the Feigin-Odesskii shuffle algebra on the 
equivariant $K$-theory of $\coprod_{n\in\bbN}\Hilb^n(\Aff^2)$ \cite{FT}. A homomorphism from the elliptic Hall algebra to the shuffle algebra is further given in \cite{SV2}, which is compatible with their actions.

The machinery in the present paper, applied to the special case when $A$ is the Chow group, and $Q$ a quiver without edge-loops,  gives a new construction of the Yangian, in terms of cohomological Hall algebra. This is an affine analogue of the classical theorem that the (extended) Ringel-Hall algebra of $Q$ is isomorphic to $U_q(\mathfrak{b}_Q)$ \cite{Ring}. 
Although this special case might be expected by analogy with \cite{SV2}, no precise statement or proof had been given (see the remarks after Theorem~\ref{ThmIntr:Yangian} for the technicality in the proof).
Note that as a corollary, our construction applied to Morava $K$-theory does satisfy Lusztig's character formula for $n=0$. Another application is a shuffle description of the Yangian, which was not known previously, although a similar description for the quantum loop algebra $U_q^+(L\fg_Q)$ has been known in various special cases (see \cite{En1} for the case when $\fg$ is of finite type which is not $G_2$, \cite{SV2} when $\fg=\widehat{\fg\fl_1}$, and \cite{Neg} when $\fg$ is $\widehat{\fs\fl_n}$).

Last but not least, through the present paper we set up the framework of studying affine quantum groups using cohomological Hall algebras, further applications of which are given in future publications, summarized in \S~\ref{subsec:intro_Kell} and \S~\ref{subsec:intro_crit}.

\subsection{The preprojective CoHA}
Let $Q=(I, H)$  be a quiver, with  set of vertices $I$ and arrows $H$.
Let $\Rep(Q,v)$ be the affine variety parametrizing representations of $Q$ with dimension vector $v=(v^i)_{i\in I}\in \N^{I}$. The vector space $\Rep(Q, v)$ carries a natural action of $G_v=\prod_{i\in I} \GL_{v^i}$. 
Let $\mu_v:T^*\Rep(Q,v)\to \fg \fl_v^*$ be the moment map of the cotangent bundle of $\Rep(Q,v)$.
The torus $T=\Gm^2$ acts on $T^*\Rep(Q,v)$ with the first $\Gm$-factor scaling $\Rep(Q,v)$ and the second one scaling the fibers of $T^*\Rep(Q,v)$ (see Assumption~\ref{Assu:WeghtsGeneral} for the conditions on this action). We set
\[
^A\calP(Q):=\bigoplus_{v\in \bbN^I} \null^A\calP_v(Q)=\bigoplus_{v\in \bbN^I} A_{T\times G_{v}}(\mu^{-1}_{v}(0)).\]
We will use the abbreviations $\calP$ and $\calP_v$ if both $A$ and $Q$ are understood from the context. In \S~\ref{preproj CoHA}, we define maps
\[m_{v_1,v_2}^P:\calP_{v_1}\otimes\calP_{v_2}\to \calP_{v_1+v_2}.\]

\begin{thm}[Theorem \ref{prop:assoc_hall}]\label{thmInt_A}
The $\bbN^I$-graded abelian group $\calP$, endowed with $m_{v_1,v_2}^P$, is an associative, $\bbN^I$-graded algebra over $A_T(\pt)$.
\end{thm}
This algebra will be called the {\it preprojective cohomological Hall algebra} (preprojective CoHA). The name is motivated by the fact that the subvariety $\mu^{-1}_{v}(0)\subset T^*\Rep(Q, v)$  parameterizes the  representations of the preprojective algebra of $Q$. Recall that the latter is the quotient of the path algebra of the doubled quiver $\overline{Q}$ by the relations $\sum_{x\in H}[x,x^*]=0$, where for any arrow $x$ in $Q$ the opposite arrow in $\overline{Q}$ is denoted by  $x^*$.

Our construction of the algebra $\calP$ has two sources. One of them is the construction of a $K$-theoretic Hall algebra of commuting varieties  defined by Schiffmann-Vasserot in \cite{SV2}, the idea of which can be traced back to Grojnowski \cite{Gr2}. Our construction for an arbitrary quiver and  OCT essentially follows from the same techniques. Another source is the construction of  Kontsevich-Soibelman of a Hall multiplication on the critical cohomology (cohomology valued in a vanishing cycle) of representation  spaces of a quiver with potential, which arises from the study of Donaldson-Thomas invariants  (DT for short) \cite{KS}. Our preprojective CoHA can be considered as  CoHA for the 2-Calabi-Yau category of representations of the preprojective algebra.

\subsection{Action on quiver varieties}
Let $\fM(v,w)$ be the Nakajima quiver variety  with dimension vectors $v,w\in\bbN^I$ and  stability condition $\theta^+$ (see \S~\ref{subsec:moment_quiver}). 
We construct a representation of the algebra $^A\calP(Q)$ on the $A$-homology of the Nakajima quiver varieties \[^A\calM(w):=\bigoplus_{v\in\bbN^I}A_{G_w\times T}(\fM(v,w)).\]
\begin{thm}[Theorem \ref{thm:prep action}]\label{ThmIntr:Action}
For any $w\in \bbN^I$, there is a homomorphism of $\bbN^I$-graded $A_T(\pt)$-algebras \[a^{\prepr}: \null^A\calP\to \End\left(\null^A\calM(w)\right).\] 
\end{thm}

For each $k\in I$, let $e_k$ be the dimension vector valued 1 at vertex $k$ and zero otherwise. 
We define the {\it spherical preprojective CoHA} to be the  subalgebra $^A\calP^{\sph}(Q)\subseteq \null^A\calP(Q)$ generated by $\calP_{e_k}=A_{T\times G_{e_k}}(\mu^{-1}_{e_k}(0))$ as  $k$ varies in $ I$.

The map $a^{\prepr}$, when restricted to the spherical subalgebra $\calP^{\sph}$, factors through the action of the convolution algebra of the Steinberg variety $Z=\fM(w)\times_{\fM_0(w)}\fM(w)$ on $\null^A\calM(w)$. In $A_{G_w\times\Gm}(Z)$ there are Nakajima-type operators, which are  cohomology classes on  Hecke correspondences.  In Theorem~\ref{thm: hecke corr},  we write down explicit elements in $\calP^{\sph}$ which map to the Nakajima-type operators under $a^{\prepr}$. 
Hence, for a general quiver, $\calP$ is a bigger symmetry on  $\calM(w)$ than that 
 of \cite{Nak99,Va00}. 
We expect the representations in Theorem~\ref{ThmIntr:Action} to produce highest weight integrable representations of the Drinfeld double of $\calP^{\sph}$. 

\subsection{The shuffle description}
Assume the coefficient ring of $A$ is a $\bbQ$-algebra. Let $^A\calS\calH$ be the shuffle algebra associated to $Q$ and the formal group law of $A$ (see \S~\ref{sec:formalShuf}). 
It is a modified version of the Feigin-Odesskii shuffle algebra \cite{FO, FT}. In \S~\ref{sec:formalShuf}, we give the definition and a geometric interpretation of $^A\calS\calH$, and hence prove the following.
\begin{thm}[Theorem \ref{thm:prepro to shuffle}]\label{ThmIntr:shuffle}
Assume the coefficient ring of $A$ is a $\bbQ$-algebra. There is an algebra homomorphism $\Theta:\null^A\calP\to\null^A\calS\calH$, which becomes an isomorphism after a suitable localization spelled out in Remark~\ref{rmk:torsion}.
\end{thm}
 
The homomorphism $\Theta$ gives an explicit description of the Hall multiplication of $\calP$ using the shuffle formulas. 
%Moreover, the image of $^A\calP^{\sph}(Q)$ in $\calS\calH$ is a deformation of $U(\fn_Q[\![u]\!])\subseteq U(\fg_Q[\![u]\!])$ associated to the formal group law of  $A$. 
In \S~\ref{subsec:Cartan}, we construct a commutative algebra $^A\calP^0$, which acts on  $^A\calP$.  We define the extended spherical CoHA  to be $\calP^{\sph,\ext}:=\calP^{\sph}\rtimes \calP^0$, which, when $Q$ has no edge-loops, is a quantization of central-extended $U(\fb_Q[\![u]\!])$ associated to the formal group law of  $A$.

In  \cite{YZ2}, we define a comultiplication on $\calP^{\sph,\ext}$, making it a bialgebra. We also construct a bialgebra pairing, under which the Drinfeld double of $\calP^{\sph,\ext}$ is a quantization of central-extended $U(\fg_Q[\![u]\!])$ again when $Q$ has no edge-loops.

\subsection{The Yangians $Y_\hbar(\fg_Q)$}
We compare the preprojective CoHA with known quantum groups. In order to do so, we need to twist the multiplication of $\calP$ by a sign. More precisely,  let $\widetilde \calP$ be the twisted preprojective CoHA, whose underlying abelian group is the same as  $\calP$, and the multiplication  $m_{v_1, v_2}^{\widetilde \calP}$ on $\widetilde \calP$ differs from $m_{v_1, v_2}^{\calP}$ by a sign coming from the Euler-Ringel form, spelled out in \S~\ref{subsec:twisting CoHA}.
Similar to the untwisted case, we have the spherical subalgebra $\widetilde{\calP}^{\sph}\subseteq \widetilde\calP$.
Define $\underline{\widetilde{\calP}^{\sph}}$ to be the quotient of $\widetilde{\calP}^{\sph}|_{t_1=t_2=\hbar/2}$ by the torsion part, the precise sense of which is spelled out in Remark~\ref{rmk:torsion}. Let $\underline{\widetilde{\calP}}^{\sph,\ext}$ be $\underline{\widetilde{\calP}^{\sph}}\rtimes \calP^0$.

Now let $A$ be the intersection theory $\CH$ (that is, the Chow group, see \cite{Ful}). 
For any $w\in\bbN^I$, let $\calM(w)=\bigoplus_{v\in\bbN^I}\CH_{G_w\times T}(\mathfrak{M}(v,w))$. 
Following Nakajima \cite{Nak99},  Varagnolo constructed an action of the Yangian $Y_\hbar(\fg_Q)$ on $\calM(w)$ \cite{Va00}.  \footnote{In \cite{Va00} the Borel-Moore homology was used instead of the intersection theory. However, note that in the verification of Yangian action, one only uses the fact that the formal group law of this cohomology theory is the additive group law. As we would like to stay in the algebraic setting in the present paper, we use the Chow group instead of Borel-Morel homology. This modification is not essential. }

\begin{thm}[Theorems \ref{thm:Yangian to sh}, \ref{thm:Yangian}]\label{ThmIntr:Yangian}
Let $Q$ be a quiver without edge loops. Let $Y_{\hbar}^{\geq 0}(\mathfrak{g}_Q)$ be the Borel subalgebra of $Y_{\hbar}(\mathfrak{g}_Q)$ (see \S~\ref{sec:Yangian_action} for a precise definition).
\begin{enumerate}
\item There is a surjective algebra homomorphism  
$\Gamma: Y_{\hbar}^{\geq 0}(\mathfrak{g}_Q) \to \null^{\CH}\underline{\widetilde{\calP}}^{\sph,\ext}$.
\item 
When  $Q$ is of type $ADE$, we have an isomorphism $\Gamma: Y_{\hbar}^{\geq 0}(\mathfrak{g}_Q) \cong \null^{\CH}\underline{\widetilde{\calP}}^{\sph,\ext}$, and the following diagram is commutative
\[
\xymatrix@R=1em {
\null^{\CH}\widetilde{\calP}^{\sph,\ext}
\ar@{->>}[r] \ar[dr]
& Y_{\hbar}^{\geq 0}(\fg_{Q}) \cong \null^{\CH}\underline{\widetilde{\calP}}^{\sph,\ext}  
\ar[d]
\\
&\End(\calM(w)),}
\] 
where the map $Y_\hbar(\fg_Q)\to \End(\calM(w))$ is defined in \cite{Va00}. 
\end{enumerate}
\end{thm}
\begin{remark}
\begin{enumerate}
\item  For each $k\in I$, we have
$\null^{\CH}\calP_{e_k}=\C[z^{(k)}]$. 
By Theorems \ref{thm: hecke corr}, the action of 
$(z^{(k)})^l$ on $\calM(w)$ is given by $(c_1(\calL_k))^l$, where $\calL_k$ is certain tautological line bundle on the Hecke correspondence in quiver varieties. 
This is the same as the action of the standard generators $x^+_{k, l}$ of $Y_{\hbar}^{+}(\fg_{Q})$ from \cite{Va00} when  $Q$ has no edge loops.
\item When $Q$ has no edge loops, we expect that it can be deduced from the recent work of \cite{SV17} that 
the action of $\null^{\CH}\widetilde{\calP}$ on $\calM(w)$ is faithful after passing to $\bbQ(\hbar)$. Therefore, one has an isomorphism 
$Y_{\hbar}^{\geq 0}(\mathfrak{g}_Q) \cong \null^{\CH}{\widetilde{\calP}}^{\sph,\ext}$  after localization.
\end{enumerate}
\end{remark}

A Hall algebra description of the positive part of the affine quantum group is a special case of Lusztig's construction of the composition subalgebra  \cite{L91}. However, a Hall algebra description of the Yangian was not known.

Combining Theorem~\ref{ThmIntr:Yangian} with Theorem~\ref{ThmIntr:shuffle} shows that  the positive part of the Yangian $Y_{\hbar}^+(\mathfrak{g}_Q)$ embeds into the shuffle algebra $\calS\calH$. This can be considered as an affine analogue of \cite{Ros}, that  $U_q^+(\fg_Q)$ is a subalgebra of the quantum shuffle algebra.  
Although the shuffle description for the quantum loop algebra $U_q^+(L\fg)$ has been known in various special cases (see \cite{En1} for the case when $\fg$ is of finite type which is not $G_2$, and \cite{Neg} for the case when $\fg=\widehat{\fs\fl_n}$), the proof of Theorem~\ref{ThmIntr:Yangian} uses an entirely different method.

The proof here pins down the quadratic relation of Nakajima operators as that in the Yangian, which was previously unclear outside of simply-laced case (see  \cite[\S~4, Remark]{Va00}); this is also the first proof of the Serre relation in the Yangian outside of simply-laced case, which we establish using a symmetric polynomial identity proved in Appendix~\ref{app:sym_poly}, which might be interesting on its own right.

The shuffle description of $U_q^+(\widehat{\fg})$ for an affine Lie algebra $\widehat{\fg}$,  was used to provide a combinatorial characterization of the dual canonical basis, and hence the shuffle formula became a basic calculation tool in $q$-characters of the KLR algebra \cite{Lec,KR}. By analogy, we expect the shuffle description of the Yangian to lead to an analogue of the dual canonical basis in the Yangian. Furthermore, according to the recent progresses \cite{BHLW,SVV}, it is expected that the trace of the graded representation category of KLR algebra is isomorphic to the Yangian. Thus, the shuffle description of Yangian should provide a calculation tool in the study of trace decategorification of KLR algebra.

\subsection{$K$-theory and elliptic cohomology}\label{subsec:intro_Kell}
The study of various quantum groups and their representations in the affine setting has been a fruitful subject. The present paper contributes more members to the family of quantum groups in the affine setting.

When $A$ is the $K$-theory, it is expected in \cite{Gr2} that, there is an algebra isomorphism $U_q^+(L\fg_Q)\to \null^K\widetilde{\calP}^\sph(Q)$.
However, the relation between the $^K\widetilde{\calP}^\sph(Q)$ action on equivariant $K$-theory of quiver varieties and the action of $U_q(L\fg_Q)$ studied in \cite{Nak99} is not clear. As the study of this relation would involve a different set of generators  of $U_q^+(L\fg_Q)$ than the one used in  \cite{Nak99}, we do not achieve this in the current paper.

Assuming $A(\pt)$ is a $\bbQ$-algebra, the formal group law associated to $A$ is isomorphic to the additive one, in the sense of \cite[Ch. IV]{Fr68}. However, an isomorphism between two formal group laws does not give an isomorphism between the corresponding CoHA's. Instead, it yields that the multiplications of the two CoHA's are related by a factor, a precise statement of which can be found in   \cite[Remark~1.4]{YZ2}. In other words, the algebra $^A\calP(Q)$   depends on the formal group law itself, instead of the isomorphism class of the formal group law. 
Nevertheless, motivated by \cite{GTL}, it is possible that the extended spherical subalgebra  $^A\calP(Q)^{\sph, \ext}$  is isomorphic to a completion of $^{\CH}\calP(Q)^{\sph,\ext}$.

In \cite{Z15}, we study the preprojective CoHA when $A$ is the equivariant elliptic cohomology of \cite{Gr,GKV95}. We construct a sheafified elliptic quantum group, whose rational sections form the operators studied in \cite{GTL15}, which in turn is the elliptic quantum group defined by Felder and his collaborators (see e.g., \cite{Fed}).

Parallel to this paper,  it is proved  in  \cite{ZZ14} that 
 the formal affine Hecke algebra studied in \cite{HMSZ} acts on the $A$--homology of the Springer fibers. 
We expect  that in type-$A$  the formal affine Hecke algebra  in \cite{ZZ14} and  $\calP$ studied in the current paper are related by a Schur-Weyl duality.  In \cite{ZZ15}, it is proved that the elliptic affine Hecke algebra acts on the equivariant elliptic cohomology of the Springer fibers.

\subsection{Preprojective CoHA and critical CoHA}\label{subsec:intro_crit}
Associated to the category of representations of a quiver with potential, there is a critical CoHA defined by Kontsevich-Soibelman,   in the framework of  Donaldson-Thomas theory. 
For the class of quivers with potential studied by Ginzburg in \cite{Gin06},
we construct an isomorphism of the corresponding preprojective CoHA and the critical CoHA in \cite{YZ16}. In {\it loc. cit. } we use
Theorem~\ref{ThmIntr:Yangian} to show the existence of an algebra homomorphism $Y_\hbar^+(\fg_Q)\to \widetilde{\mathscr{H}}|_{t_1=t_2=\hbar/2}$, where $\widetilde{\mathscr{H}}$ is a version of the critical CoHA of \cite{KS},   twist by a sign coming from the Euler-Ringel form.

After an earlier version of this paper appeared on arXiv, in the course of studying the semicanonical bases and an analogue of the Kac polynomial, a similar notion of  CoHA for preprojective algebra was introduced in \cite{RS}, although the results in the present paper were not included in {\it loc. cit.} except for Theorem~\ref{thmInt_A} in the case when $A$ is the Borel-Moore homology.

Shortly after an earlier version of the current paper, which contains Theorems~\ref{ThmIntr:Yangian} for the Dynkin case and \cite[Theorem~A]{YZ16},  appeared on arXiv,  a conjecture about a relation between Yangian and critical CoHA for more general cases was independently proposed by Davison \cite{D}.
\subsection*{Acknowledgments} Many ideas in this paper are rooted in the work of Schiffmann and Vasserot \cite{SV1,SV2}. The authors are grateful to Marc Levine, Zongzhu Lin, Yan Soibelman, Valerio Toledano Laredo, and Eric Vasserot for helpful discussions and correspondence. We thank Hiraku Nakajima, Olivier Schiffmann, and an anonymous referee for pointing out an error in an earlier version of \S~\ref{sec:Yangian_action}.

Part of this paper was conceived when G.Z. was waiting for the security clearance of his US-Visa in Beijing in 2013. He would therefore like to thank the Morningside Center of Mathematics at the Chinese Academy of Sciences for accommodation, and the U.S. overseas diplomats for the otherwise unavailable opportunity of this stimulating visit. 
During the preparation of this paper, G.Z. was hosted by Max-Planck-Institut f\"ur Mathematik in Bonn. During the revision of this paper, Y.Y. was hosted by MPIM,  G.Z. was supported by Fondation Sciences Math\'ematiques de Paris and CNRS. Part of the work was done when both authors were temporary faculty members at the University of Massachusetts, Amherst.
\subsection*{Summary}
For the convenience of the readers, we summarize the various algebras as follows.
\[
\xymatrix@R=1em @C=1em{
%-----line 2-----
    & ^{\CH}\calP^{\sph, \ext}  \ar[rrdd] |(.465)\hole   \ar@{^{(}->}[d]&& &\\
 %-----line 1-----
^{\CH}\calP \ar[d]\, \ar@{}[r]|-*[@]{\subset} &
 ^{\CH}\calP^{\ext} \ar[d]   \ar[rr]  &&\End\calM(w)\\
 %-----line 3-----
 \calS\calH \, \ar@{}[r]|-*[@]{\subset} & \calS\calH^{\ext} \ar[r]&
\calS\calH^{\ext}|_{t_1=t_2=\frac{\hbar}{2}}
& Y^{\geq 0}_{\hbar} \ar@{^{(}->}[l] \, \ar[u] 
}
\]
Here the multiplications of the Hall algebras and shuffle algebras are understood as twisted  by a sign specified in \S~\ref{subsec:twisting CoHA}.

\section{Algebraic oriented cohomology theory}\label{sec:prelim}

In this section we collect basic notions about equivariant oriented Borel-Moore homology theory.

\subsection{Equivariant oriented cohomology theories}\label{subsec:inf grass}
\label{sec:ordinary BM}
\label{subsec:pushforward}

Fix a base field $k$. For any reductive algebraic group $G$, let $\Sch_k^G$ be the category of schemes over $k$ of finite type with a $G$-action.
We will consider equivariant  Borel-Moore homology theory in the sense of \cite[\S~2]{CZZ3} and \cite[\S~5.1]{ZZ14}.
In particular, it is the following data: 
 \begin{enumerate}
\item For any object  $X$ in $\Sch^G_k$,  $A_G(X)$ is a module over the commutative ring $A_G(\pt)$.
\item (Proper pushforward) For $f:Y\to X$ a proper morphism in $\Sch^G_k$, there is a homomorphism $f_*:A_G(Y)\to A_G(X)$.
\item (Smooth pullback) For a smooth morphism $f:Y\to X$ in $\Sch_k^G$, there is a homomorphism $f^*:A_G(X)\to A_G(Y)$.
\item (Refined Gysin pullback) For any local complete intersection morphism $f: Y\to X$, and an arbitrary morphism $Z\to X$, let $g:Z\times_{X} Y \to Z $ be the base change. Then,  there is a refined pullback map $f^\sharp_{g}: A_G( Z) \to A_G( Z\times_{X} Y).$
We will also write $f^\sharp$ if $g$ is understood from the context. It specializes to the smooth pullback $f^*$ when $f$ is smooth and $Z\to X$ is the identity morphism on $X$.
\item (1st Chern class operators) For each line bundle $L\to X$, $X\in \Sch^G_k$, there is a graded homomorphism $\tilde{c}_1(L):A_G(X)\to A_G(X)$.
\end{enumerate}
These all satisfy a number of compatibilities, detailed in \cite[\S~2.1, \S~2.2]{LM}.
When restricting $A$ to the category of smooth $G$-varieties, it factors through the category of commutative rings with unit.
As we will need the compatibility of push-forward and the refined Gysin pull-backs, we collect some basic facts in \S~\ref{subsec: Lag corr}.

When $G$ is trivial, $A_{G}$ is an oriented Borel-Moore homology theory in the sense of \cite{LM}. Hence, there is a formal group law $(R,F)$ associated to it, where $R=A(\pt)$ and $F(u,v)\in R[\![u,v]\!]$.
We use the short-hand notations $u+_Fv:=F(u,v)$. Denote $-_Fv$ to be the inverse of $v$ of the formal group law, in other words, $F(v,-_Fv)=0$.

Equivariant OCT's of primary interests to us are in the following example.
\begin{example}\label{exam:EOCT}
\begin{enumerate}
\item  In \cite{HM} it has been explained how any non-equivariant Borel-Moore homology theory $A$ extends to an equivariant theory, following an idea due to Totaro in \cite{To}.
More precisely, for any reductive group $G$, the classifying space of $G$ is a system $EG:=\{EG_N\}_{N\in\bbN}$, where each $EG_N$ is a Zariski open subset in a representation of $G$ on which $G$-acts freely, and satisfies the condition of a {\it good system} spelled out in \cite[Definition~10]{HM}.
For simplicity, we call $BG:=\{EG_N/G\}_{N\in\bbN}$ {\it the classifying space} of $G$, and we define $A_G(X)$ to be $\lim_NA(X\times_G EG_N)$.

\item Note that in the construction of \cite{HM}, there is no stabilization in each homological degree in the system $\{A(X\times_G EG_N)\}_N$ for a general $A$. However, it is proved in \cite{To}, that when $A$ is $\CH$, for each fixed homological degree $i$ the system $\{\CH^i(X\times_G EG_N)\}_N$ does stabilize. See also \cite{EG}. Therefore, in this case, we will take $\CH_G(X)$ to be the direct sum $\bigoplus_i\lim_N\CH^i(X\times_G EG_N)$. In particular,  for us $\CH_{\Gm}(\pt)=\bbQ[z]$ while $\lim_N\CH(\pt\times_{\Gm} (E\Gm)_N)=\bbQ[\![z]\!]$.

\item More generally, for any 1-dimensional algebraic group  $\bbG$ over a commutative $\bbQ$-algebra $R$, together with a local uniformizer $\fl$, the expansion of the group structure of $\bbG$ under $\fl$ gives rise to a formal group law.  There is an OCT $A$ associated to this data constructed in \cite[\S~4.1]{LM}.
For $A$ of this type, and for any compact Lie group $G$,  Lurie \cite[Theorem~3.2]{Lur} constructed an equivariant OCT $A_G$ for $G$-finite CW-complexes. In particular, $A_G(\pt)$ is the coordinate ring of $\bbG^{\rank G}/W$, where $W$ is the Weyl group of $G$. 
We expect similar equivariant OCT exists in the algebraic setting, i.e., for a reductive group $G$ acting on an algebraic variety $X$. This is known to be true when  $\bbG$ is the additive group (see (2) above) and the multiplicative group (see (4) below).
\item \label{rmk:K_Chow}
When $A$ is the $K$-theory with $\bbQ$-coefficients, we will use the equivariant $K$-theory of \cite{Th}. Hence, $K_{\Gm}(\pt)=\bbQ[z^\pm]$.
\end{enumerate}
\end{example}

Let $V\to X$ be a vector bundle of $X$ with Chern roots $\lambda_1, \dots, \lambda_n$. It is known that the cohomology of Grassmannian is generated by the Chern classes of the tautological vector bundle. In other words, let $v_1,\cdots,v_r$ be the Chern roots of the rank-$r$ tautological bundle $\calR(r)$ on $\Grass(r, V)$.
For any pair $(p,q)$ of positive integers, let $\hbox{Sh}(p,q)$ be the subset of $\fS_{p+q}$ consisting of $(p,q)$-shuffles (permutations of $\{1,\cdots,n\}$ that preserve the relative order of $\{1,\cdots,p\}$ and $\{p+1,\cdots,p+q\}$). 
\begin{prop}\label{prop:grass_push}
Let $V\to X$ be a rank $n$ vector bundle and let $p: \Grass(r, V)\to X$ be the associated Grassmannian bundle.
For any  the oriented cohomology theory $A$ with formal group law $(R, F)$ where $\bbQ\subseteq R$, let $f(v_1, \dots, v_r) \in A(\Grass(r, V))$.
Then, 
$$p_*^{A}(f(v_1,\dots,v_r))=\sum_{\{\sigma\in \hbox{Sh}(r, n-r)\}}
\sigma\frac{ f(\lambda_1, \dots, \lambda_r)}{\prod_{1\leq j\leq r, r+1 \leq i\leq n}(\lambda_i-_{F}\lambda_j)},$$
where $\lambda_1 \dots \lambda_n$ are Chern roots of $V$ in $A$.
\end{prop}

Let $I$ be a finite set and $v=(v^i)_{i\in I}\in \bbN^I$ be a vector with entries non-negative integers. Let $G_{v}$ be $\prod_{i\in I} \GL_{v^i}$, and let $T$ be the maximal torus of $G_v$.
The Weyl group $\fS_v:=\prod_{i\in I}\fS_{v^i}$ acts on $A_T(\pt)$ by permutation. 
When $A_G$ is as in Example \ref{exam:EOCT}(2), $A_G(\pt)$ is the coordinate ring of $(\prod_{i\in I}\bbG^{v^i})/\fS_v$. When $A_G$ is as in Example \ref{exam:EOCT}(1), $A_G(\pt)\cong A(\pt)[\![\lambda_j^i]\!]^{\fS_v}_{i\in I,j=1,\dots,v^i}$. 
For any dimension vector $v\in \bbN^I$, with $v=v_1+v_2$, we denote $\Sh(v_1,v_2)\subset \fS_v$ to be the product $\prod_{i\in I}\Sh(v_1^i,v_2^i)$.

\subsection{Lagrangian correspondence formalism}
\label{subsec: Lag corr}
Now we recall the Lagrangian correspondence formalism following the exposition in \cite{SV2}.

Let $X$ be a smooth quasi-projective variety endowed with an action of a reductive algebraic group $G$. The cotangent bundle $T^*X$ is a symplectic variety. The induced action of $G$ on $T^*X$ is Hamiltonian. Let $\mu:T^*X\to (\Lie G)^*$ be the moment map. Following \cite{SV2}, we denote $\mu^{-1}(0)\subseteq T^*X$ by $T^*_GX$.

Let $P\subset G$ be a parabolic subgroup and $L\subset P$ be a Levi subgroup. Let $Y$ be a smooth quasi-projective variety equipped an action of $L$, and $X'$ smooth quasi-projective with a $G$-action. Let $\calV\subseteq Y\times X'$ be a smooth subvariety. Let $\pr_1, \pr_2$ be the two projections restricted on $\calV$
\[\xymatrix{
Y& \calV \ar[l]_{\pr_1}\ar[r]^{\pr_2} &X'}.\]
Assume the first projection $\pr_1$ is a vector bundle, and the second projection $\pr_2$ is a closed embedding. 

Let $X:=G\times_PY$ be the twisted product. Set $W:=G\times_P\calV$ and consider the following maps
\[\xymatrix{
X & W \ar[l]_{f}\ar[r]^{g} &X'}\]
\[f: [(g, v)] \mapsto [(g, \pr_1(v))], \,\ g: [(g, v)] \mapsto g\pr_2(v), \] 
where $[(g, v)]$ is the pair $(g, v)\mod P$. Note that the natural map $T^*X\to G\times_PT^*Y$ is a vector bundle. 
\begin{lemma}[\cite{SV2}, Lemma~7.1]
\label{lem:iso of G}
There is an isomorphism $G\times_PT^*_LY\cong T^*_GX$ such that the following diagram commutes
\[\xymatrix@R=1.5em{
G\times_PT^*_LY\ar[r]^(0.6){\cong}\ar@{^{(}->}[d]&T^*_GX\ar@{^{(}->}[d]\\
G\times_PT^*Y\ar@{^{(}->}[r]&T^*X
}\]where $G\times_PT^*Y\inj T^*X$ is the zero-section of the vector bundle $T^*X\to G\times_PT^*Y$.
\end{lemma}

Let  $Z:=T^*_W(X\times X')$ be the conormal bundle of $W$ in $X\times X'$. Let $Z_G\subseteq T^*_GX\times T^*_GX'$ be the intersection $Z\cap (T^*_GX\times T^*_GX')$. 
Then we have the following diagram.
\[\xymatrix@R=1.5em{
G\times_PT^*_LY\ar[r]^(0.6){\cong}\ar@{^{(}->}[d]&T^*_GX\ar@{^{(}->}[d]&Z_G\ar[l]_{\overline\phi}\ar[r]^{\overline\psi}\ar@{^{(}->}[d]&T^*_GX'\ar@{^{(}->}[d]\\
G\times_PT^*Y\ar@{^{(}->}[r]&T^*X&Z\ar[l]_{\phi}\ar[r]^{\psi}&T^*X'
}\]where $\phi:Z\to T^*X$ and $\psi:Z\to T^*X'$ are respectively the first and second projections of $T^*X\times T^*X'$ restricted to $Z$.
%Following the same proof as in \cite[Lemma~7.3]{SV2}, we have the following lemma.
\begin{lemma} \cite[Lemma~7.3]{SV2}
\label{lem:fiber_lag}
The morphism $\psi:Z\to T^*X'$ is proper. We have $\psi^{-1}(T^*_GX')=Z_G$ and $\phi^{-1}(T^*_GX)=Z_G$.
\end{lemma}

Let $A$ be an oriented Borel-Moore homology theory. Existence of refined Gysin pull-back and Lemma~\ref{lem:fiber_lag} ensure the existence of the map $\overline\psi_*\circ\phi^\sharp:A_G( T^*_GX)\to A_G( T^*_GX')$. 

\begin{lemma}\label{lem:gysin+pullback_comm}
The following diagram commutes
\[\xymatrix@R=1.5em{
A_G( T^*_GX)\ar[r]^{\overline\psi_*\circ\phi^\sharp}\ar[d] & A_G( T^*_GX')\ar[d]\\
A_G(T^*X)\ar[r]_{\psi_*\circ\phi^*}&A_G(T^*X')
}\]
where the vertical maps are push-forwards induced by natural embeddings.
\end{lemma}
\begin{proof}
It follows directly from Lemma~\ref{Lem:Gysin}(\ref{Lem:gysin_base}) below.
\end{proof}

\subsection{Base-change for Lagrangian correspondences}
We collect some basic facts about compatibility of push-forward and Gysin pull-back in oriented Borel-Moore homology. We will apply these facts to the setting of Lagrangian correspondences.

Recall that two morphisms $f:Y\to X$ and $q:X'\to X$ are said to be {\it transversal} if
\[
\Tor_k^{\calO_X}(\calO_{X'}, \calO_Y)=0,\,\  \hbox{for any $k>0$.}
\] 

\begin{lemma}[\cite{LM}, Theorem~6.6.6(2), and Lemma~6.6.2]\label{Lem:Gysin}
Consider the following diagram in $\Sch_k$
\Omit{\[
\xymatrix{
H\ar[d]^{g'}\ar[r]^{f''}&W\ar[d]^{g}\\
Y'\ar[r]^{f'} \ar[d]&X'\ar[d]^q\\
Y\ar[r]^{f} &X}
\]}
\[
\xymatrix@R=1.5em{
H \ar[r]^{g'} \ar[d]^{f''}&Y'\ar[r] \ar[d]^{f'}&Y\ar[d]^{f}\\
W \ar[r]^{g} &X'\ar[r]^{q}&X}
\]with all squares Cartesian. 
Assume $f$ is a locally complete intersection morphism.
\begin{enumerate}
\item\label{Lem:gysin_base} If $g$ is proper, then $f^\sharp_{f'}g_*=g'_*f^\sharp_{f''}$.
\item\label{Lem:gysin_two} If $f$ and $q$ are transversal,
then $f^\sharp_{f''}=f'^\sharp_{f''}$.
\end{enumerate}
\end{lemma}

As a consequence, we have the following.
\begin{lemma}\label{Lem:base_change}
Consider the following commutative diagram in which every square is Cartesian.
\[\xymatrix@R=1em {
%----line 1-----
W' \ar[rr]^{f''}\ar[dd]^{g''} 
\ar[rd]
&&Y'\ar@{-->}[dd]^(0.3){g'} \ar[rd]&\\
%----line 2-----
&W \ar[rr]\ar[dd]_(0.3){l} &&Y \ar[dd]^{g} \\
%----line 3-----
Z' \ar[rr]^(0.3){f'} \ar[rd]& &
 X' \ar[rd]^{i}& \\
%----line 4-----
&Z \ar[rr]_{f} & & X 
}\]
Assume $g$ and $i$ are proper, $f$ and $g$ are transversal, and $f$ is a locally complete intersection morphism. 
Then we have the equality
\[
 f''_* \circ l_{g''}^\sharp=g^\sharp_{g'}\circ f'_*
\] as homomorphisms from $A(Z')$ to $A(Y')$.
\end{lemma}
\begin{proof}
By Lemma~\ref{Lem:Gysin}(\ref{Lem:gysin_base}), we have $f''_*\circ g_{g''}^\sharp =g^\sharp_{g'}\circ f'_*$.
By Lemma~\ref{Lem:Gysin}(\ref{Lem:gysin_two}), $g^\sharp_{g''}=l^\sharp_{g''}$. We are done.
\end{proof}

The following is a sufficient condition for two morphisms to be transversal.
\begin{lemma}[\cite{SV2}, Proposition C.1]\label{lem:dim_trans}
Consider the following Cartesian diagram
\[
\xymatrix@R=1.5em{
Y'\ar[r]^{f'} \ar[d]^{g'}&X'\ar[d]^{g}\\
Y\ar[r]^{f} &X.}
\]
Assume $g$ is proper, the map $f'\times g': Y' \to X'\times Y$ is a closed embedding, and assume $\dim(X)+\dim(Y')=\dim(X')+\dim(Y)$. 
Then, $g'$ is proper, and  $f$ and $g$ are transversal. 
\end{lemma}

One example is given as below.
Let $W_1\subset X_3\times X_2$, $W_2\subset X_3\times X_1$, and $W_3\subset X_2\times X_1$
be subvarieties.
We assume $W_2=W_1\times_{X_2} W_3$. 
We consider the Lagrangian subvarieties 
\[
Z_1=T^*_{W_1}(X_3\times X_2), \,\
Z_2=T^*_{W_2}(X_3\times X_1), \,\
Z_3=T^*_{W_3}(X_2\times X_1).
\]
Assume the intersection $(W_1\times X_1)\cap (X_3\times W_2)$ is transversal in $X_3\times X_2\times X_1$. Thus, by \cite[Theorem 2.7.26]{CG} we have an isomorphism $Z_1\times_{T^*X_2} Z_{3} \cong Z_2$. 
In particular, the following commutative diagram is Cartesian
\[\xymatrix@R=1.5em{
Z_{1}\ar[r]^{\phi_1}&T^*X_2\\
Z_{2}\ar[u]\ar[r]&Z_{3}\ar[u]_{\psi_3}.
}\]

\begin{lemma}\label{lem:dim_lag}
With notations as above,  we have $\dim(Z_1)+\dim(Z_3)=\dim(Z_2)+\dim(T^* X_2)$. In particular, $\phi_1$ and $\psi_3$ are transversal. The map $Z_2\to Z_1$ is proper.
\end{lemma}
\begin{proof}
The dimension counting follows from the fact that $Z_1$, $Z_2$, and $Z_3$ are all Lagrangian subvarieties. 

In Lemma~\ref{lem:dim_trans}, taking $Y$ to be $Z_1$, $Y'$ to be $Z_2$, $X$ to be $T^*X_2$, and $Z_3$ to be $X'$, the properness of $Z_2\to Z_1$ follows.
\end{proof}

\section{The formal cohomological Hall algebras}\label{sec:formalCoHA}
In this section, we review in the algebraic setting the cohomological Hall algebra defined by Kontsevich and Soibelman in \cite{KS}.
The idea of studying CoHA from arbitrary oriented Borel-Moore homology theory goes back to \cite[\S~3.7]{KS}.
We spell out in an explicit fashion the shuffle formula for this CoHA, with  emphasis  on the dependence  on the formal group law.
\subsection{The formal cohomological Hall algebras}\label{subsec:2.1}
Let $k$ be an arbitrary field.
Let $Q$ be a quiver with vertex set $I$ and arrow set $H$. 
We assume in this paper that $I$ and $H$ are finite sets. 
For $h\in H$, we denote by $\inc(h)$ (resp. $\out(h)$) the incoming (resp. outgoing) vertex of $h$.
For any dimension vector $v=(v^i)_{i\in I}\in \bbN^I$, the representation space of $Q$ with dimension vector $v$ is denoted by $\Rep(Q,v)$. That is, let $V=\{V^i\}_{i\in I}$ be an $I$-tuple of $k$-vector spaces with dimension vector $\dim(V^i)=v^i$.  Then, 
\[
\Rep(Q, v):=\bigoplus_{h\in H}\Hom_{k}(V^{\out(h)},V^{\inc(h)}).
\]
\Omit{
Let  $a_{ij}\in \Z_{\ge 0}$ be the number of arrows from vertex $i$ to vertex $j$. Choosing a basis for each $V^i$, for each $i\in I$, 
then the representation space of $Q$ is isomorphic to
\[
\Rep(Q, v)\cong \prod_{i, j\in I}k^{a_{ij}v^iv^j}.
\]
}
The algebraic group $G_{v}:=\prod_{i\in I}\GL(v^i)$ acts on $\Rep(Q,v)$ by conjugation.
For any $v_1, v_2\in \bbN^I$ such that $v:=v_1+v_2$, let $V_1\subset V$ be an $I$-tuple of vector subspaces of $V$ with dimension vector $v_1$. The parabolic subgroup of $G_v$ that preserves $V_1$ will be denoted by 
$P\subset G_v$. Let $L=G_{v_1}\times G_{v_2}$ be the standard Levi-subgroup in $P$.

Let $A$ be an oriented Borel-Moore homology theory as in \S~\ref{sec:prelim} and $(R,F)$ be the formal group law associated to $A$. 
We consider the $\bbN^I$-graded $R$-module $\calH^{\For}_A(Q):=\bigoplus_{v\in \bbN^I}A_{G_v}(\Rep(Q,v))$ (or simply $\calH^{\For}$ when $A$ and $Q$ are understood from the context).
For each pair of dimension vectors $v_1,v_2\in\bbN^I$, we define the Hall multiplication 
\begin{equation}\label{equ: mulp}
m_{v_1, v_2}: A_{G_{v_1}}^*(\Rep(Q,v_1))\otimes A_{G_{v_2}}^*(\Rep(Q,v_2))
\to A_{G_{v}}^*(\Rep(Q,v))
\end{equation}
 as in \cite{KS}.
Start with the K{\"u}nneth morphism (refereed to as the external product in \cite[Ch 2]{LM}, and \cite[Proposition 5.4]{ZZ14} in the equivariant setting)
\[
\otimes: A_{G_{v_1}}^*(\Rep(Q,v_1))\otimes A_{G_{v_2}}^*(\Rep(Q,v_2))
\to A_{G_{v_1}\times G_{v_2}}^* (\Rep(Q,v_1)\times \Rep(Q,v_2)). \]
Define
\[\Rep(Q)_{v_1, v_2}:=\{x\in \Rep(Q,v)\mid x(V_1)\subset V_1\}\subset \Rep(Q,v).\]
We have the following correspondence of $G_v$-equivariant varieties: 
\[
\begin{xymatrix}{
G_{v}\times_{P}\big(\Rep(Q,v_1)\times \Rep(Q,v_2)\big) &
G_{v}\times_{P}(\Rep(Q)_{v_1, v_2})\ar[r]^(0.55){\eta} \ar[l]_(0.37){p} &
\Rep(Q, v_1+v_2)},
\end{xymatrix}\]
where 
$p$ is the projection, and $\eta: (g, x)\mapsto gxg^{-1}$ is the action by conjugation.
Consider the following 3 morphisms: 
\begin{enumerate}
\item The isomorphism 
\[
A_{G_{v_1}\times G_{v_2}}^* (\Rep(Q,v_1)\times \Rep(Q,v_2)) \cong
A_{G_v}\left(G_{v}\times_{P}\big(\Rep(Q,v_1)\times \Rep(Q,v_2)\big)\right).
\]
\item The pullback $p^*$:
\[
p^*: A_{G_v}\left(G_{v}\times_{P}\big(\Rep(Q,v_1)\times \Rep(Q,v_2)\big)\right)
\to A_{G_v}(G_{v}\times_{P}(\Rep(Q)_{v_1, v_2})).
\]
\item The pushforward $\eta_*$:
\[
\eta_*: A_{G_v}(G_{v}\times_{P}(\Rep(Q)_{v_1, v_2}))
\to A_{G_v}(\Rep(Q,v_1+v_2)).
\]\end{enumerate}
The map $m_{v_1, v_2}$  \eqref{equ: mulp} is defined as the composition of the K{\"u}nneth morphism with the above 3 morphisms.
\begin{prop}
The maps $m_{v_1, v_2}$ for  $v_1,v_2\in\bbN^I$ fit together, defining an $\bbN^I$-graded associative $R$-algebra structure on $\calH^{\For}_A(Q)$.
\end{prop}
This is essentially Theorem~1 of \cite{KS}, replacing the usual cohomology by oriented Borel-Moore homology theory $A$.

\begin{definition}
The $\bbN^I$-graded $A(\pt)$-module $\calH^{\For}_A(Q)$ endowed with multiplication $m_{v_1,v_2}$ is the formal cohomological Hall algebra (formal CoHA) associated to $A$ and $Q$.
\end{definition}

\subsection{Formula of the formal Hall multiplication}
Now let $A$ be an oriented Borel-Moore homology theory, with  $\bbQ\subseteq A(\pt)$, extended to an equivariant theory by Totaro's construction. 
In this subsection, we use the pushforward formula in  \S~\ref{subsec:pushforward} to give an explicit formula of the multiplication $m_{v_1, v_2}$. The space $\Rep(Q, v)$ is contractible. 
Thus, we have the isomorphism 
\[ 
A_{G_v}(\Rep(Q,v))\cong A_{G_v}(\pt)\cong R
[\![\lambda^i_j]\!]_{i\in I, j=1,\dots, v^i}^{\fS_{v}},\] where 
$R:=A(\pt)$, and $\{\lambda_{j}^i\}_{j=1, \dots, v^i}$ are the Chern roots of the tautological bundle $\calR(v^i)$ on $\Gr(v^i ,\infty)$.
We now describe the  multiplication map
\[m_{v_1,v_2}:\calH^{\For}_{v_1}\otimes\calH^{\For}_{v_2}\to\calH^{\For}_{v_1+v_2}.\]
It is convenient to write $\calH^{\For}_{v_1}$ as $R
[\![\lambda'^i_j]\!]_{i\in I, j=1,\dots, v_1^i}^{\fS_{v_1}}$, and  $\calH^{\For}_{v_2}$ as $R[\![\lambda''^i_s]\!]_{i\in I, s=1,\dots, v_2^i}^{\fS_{v_2}}$. We view $\calH^{\For}_{v_1}\otimes \calH^{\For}_{v_2}$ as a subalgebra of $R
[\![\lambda^i_j]\!]_{i\in I, j=1,\dots, (v_1+v_2)^i}$, by sending $\lambda'^i_s$ to $\lambda^i_s$, and $\lambda''^i_t$ to $\lambda^i_{t+v_1^i}$.
Then, the multiplication map $m_{v_1, v_2}$ can be considered as a map 
\[
m_{v_1, v_2}: 
R[\![\lambda'^i_j]\!]_{i\in I, j=1,\dots, v_1^i}^{\fS_{v_1}}\otimes R
[\![\lambda''^i_t]\!]_{i\in I, t=1,\dots, v_2^i}^{\fS_{v_2}} \to R
[\![\lambda^i_j]\!]_{i\in I, j=1,\dots, v^i}^{\fS_{v}}.
\]
The following formula of $m_{v_1, v_2}$ is essentially Theorem 2 in \cite{KS}. 
\begin{prop}
For $f_i\in \mathcal{H}_{v_i}^{\For}$, $i=1, 2$, the product $m_{v_1, v_2}(f_1, f_2)$, as a symmetric function in $R
[\![\lambda^i_j]\!]_{i\in I, j=1,\dots, (v_1+v_2)^i}^{\fS_{v_1+v_2}}$, is given by the following formula:
\[
\sum_{\sigma\in \Sh(v_1,v_2)}\sigma \Big( f_1(\lambda\rq{}^i_ s)f_2(\lambda\rq{}\rq{}^j_t )
\frac{ \prod_{i, j\in I}\prod_{s=1}^{v_1^i}
\prod_{t=1}^{v_2^j}(\lambda\rq{}\rq{}^j_t-_{F}\lambda\rq{}^i_s)^{a_{ij}} }{\prod_{i\in I}\prod_{s=1}^{v_1^i}
\prod_{t=1}^{v_2^i}(\lambda\rq{}\rq{}^i_t-_{F}\lambda\rq{}^i_s)}\Big),
\] where $a_{ij}$ is the number of arrows from vertex $i$ to vertex $j$.
\end{prop}
\begin{proof}
Let $i: \Rep(Q)_{v_1,v_2}\inj \Rep(Q,v)$ be the embedding. 
The pushforward $\eta_*$ is the composition of the following two morphisms:
\begin{align*}
&i_*: A_{P}( \Rep(Q, v)_{v_1, v_2}) \to A_{P}( \Rep(Q, v))\\
&\pi_*: A_{P}( \Rep(Q, v)) \cong A_{G_{v}}( G_v\times_{P}\Rep(Q, v))
\to A_{G_{v}}( \Rep(Q, v)).
\end{align*}

The pushforward $i_*$ in the equivariant oriented Borel-Moore theory is given by $i_*(f)=f\cdot e_{v_1, v_2}$, where $e_{v_1, v_2}$ is the equivariant Euler class of the normal bundle of $i$. The embedding $i$ induces the following embedding of vector bundles on $BL:=\Grass(v_1, \infty) \times \Grass(v_2, \infty)$:
\[\Rep(Q)_{v_1, v_2}\times_{(G_{v_1}\times G_{v_2})}(EG_{v_1}\times EG_{v_2})\hookrightarrow \Rep(Q,v)\times_{(G_{v_1}\times G_{v_2})}(EG_{v_1}\times EG_{v_2}).\]
By definition, $e_{v_1, v_2}$ is the Euler class of the quotient bundle. 
We identify the quotient bundle with  \[\bigoplus_{h\in H}\sHom_{\calO}(\calR(v_1^{\out(h)}),\calR(v_2^{\inc(h)}))\cong\bigoplus_{i,j\in I}(\calR(v_1^i)^*\otimes \calR(v_2^j))^{a_{ij}}, \] where $\calR(r)$ is the tautological bundle of $\Grass(r, \infty)$. Thus, the equivariant Euler class is \[
e_{v_1, v_2}=\prod_{i,j\in I}\prod_{s=1}^{v_1^{i}}\prod_{t=1}^{v_2^{j}}(\lambda^j_t-_{F}\lambda^i_s)^{a_{ij}}.
\]

As a consequence, the multiplication map $m_{v_1, v_2}$ sends $f_1\in\calH^{\For}_{v_1}$ and $f_2\in\calH^{\For}_{v_2}$ to $\pr_*(f_1\cdot f_2\cdot e_{{v_1, v_2}})$, where $\pr$ is the projection $\Grass(v_1, \calR(v)) \to\Gr(v,\infty)$. Applying Proposition~\ref{prop:grass_push} to  $\pr$, we get the  multiplication formula. 
\end{proof}

\begin{example}
Let $Q$ be a quiver with one single vertex and $a$ loops. The formal cohomological Hall algebra is $\calH^{\For}=\bigoplus_n \calH^{\For}_n$ with $ \calH^{\For}_n= \QQ[\![\lambda_1,\dots,\lambda_n]\!]^{\fS_n}$. 
The Hall multiplication of $f_1\in\calH^{\For}_r$ and $f_2\in\calH^{\For}_{n-r}$ becomes \begin{align*}
m(f_1, f_2)=&\sum_{\{J\subset [1\dots, n], |J|=r\}}
\frac{ f_1(\lambda_j)f_2(\lambda_i)_{j\in J, i\notin J}
}{\prod_{j\in J, i\notin J}(\lambda_i-_{F}\lambda_j)) }\cdot \prod_{j\in J, i\notin J}(\lambda_i-_{F}\lambda_j)^{a} \\
%-------------------------------
=&\sum_{\{\sigma\in \hbox{Sh}(r, n-r)\}}
\sigma\cdot \Big(\frac{ f_1(\lambda_1, \dots, \lambda_r)f_2(\lambda_{r+1}, \dots, \lambda_n)}{\prod_{1\leq j\leq r, (r+1) \leq i\leq n}(\lambda_i-_{F}\lambda_j)}\prod_{1\leq j \leq r,  (r+1) \leq i\leq n}(\lambda_i-_{F}\lambda_j)^{a} \Big).
\end{align*}
\end{example}
\Omit{
\begin{remark}
Recall we have the isomorphism
\[A_{G_{v_1}\times G_{v_2}} (\Rep(Q,v_1)\times \Rep(Q,v_2))\cong
A_{G_{v_1,v_2}}(\Rep(Q)_{v_1, v_2}).\]
By definition, the multiplication map $m_{v_1, v_2}$ is the pushforward of the following composition:
\[
p\circ i: EG_{v_1,v_2}\times_{G_{v_1,v_2}}(\Rep(Q)_{v_1, v_2}) \to E{G_v}\times_{G_v}\Rep(Q, v). 
\]
In Section \S\ref{subsec:inf grass}, we identify the space $EG_v$ with 
$EL\cong E G_{v_1,v_2}$. Thus, 
\[
EG_{v_1,v_2}\times_{G_{v_1,v_2}}(\Rep(Q)_{v_1, v_2}) \cong EG_v\times_{G_v} ({G_v}\times_{G_{v_1,v_2}}\Rep(Q)_{v_1, v_2}). 
\]
The multiplication map is now the pushforward of the map
\[
\psi: EG_v \times_{G_v} ({G_v}\times_{G_{v_1,v_2}}\Rep(Q)_{v_1, v_2})\to 
E{G_v}\times_{G_v}\Rep(Q, v). 
\]
We will use this description $m_{v_1, v_2}=\psi_*$  later. 
\end{remark}
\textcolor{blue}{The setup is different now. Maybe delete $EG$!!!}
 }

\section{The generalized shuffle algebras}\label{sec:formalShuf}
Let  $(R, F)$ be any formal group law. 
In this section, we define the generalized formal shuffle algebra $\calS\calH$ associated to the formal group law $(R,F)$ and the quiver $Q$. 
In the shuffle algebra considered in this section,
there are two quantization parameters $t_1, t_2$. Geometrically these two quantization parameters come from the two dimensional torus $T=\G_m^2$ action on the cotangent bundle of representation space of the quiver $Q$.

\subsection{The formal shuffle algebra}\label{subsec:formal shuffle}
The formal shuffle algebra $\calS\calH$ is an $\bbN^I$-graded $R[\![t_1,t_2]\!]$-algebra. As an $R[\![t_1,t_2]\!]$-module, we have $\calS\calH=\bigoplus_{v\in\bbN^I}\calS\calH_v$. The degree $v$ piece is
\[\calS\calH_v:=R[\![t_1,t_2]\!]
[\![\lambda^i_s]\!]_{i\in I, s=1,\dots, v^i}^{\fS_v}.\] 
For any $v_1$ and $v_2\in \bbN^I$, we consider $\calS\calH_{v_1}\otimes \calS\calH_{v_2}$ as a subalgebra of \[R[\![t_1,t_2]\!][\![\lambda^i_j]\!]_{i\in I, j=1,\dots, (v_1+v_2)^i}\] by sending $\lambda'^i_s$ to $\lambda^i_s$, and $\lambda''^i_t$ to $\lambda^i_{t+v_1^i}$.
%We now define the multiplication 
%\[\calS\calH_{v_1}\otimes_{R[\![t_1,t_2]\!]}\calS\calH_{v_2}\to \calS\calH_{v_1+v_2}.\]
Set:
\begin{equation}\label{equ:fac1}
\fac_1:=\prod_{i\in I}\prod_{s=1}^{v_1^i}
\prod_{t=1}^{v_2^i}\frac{\lambda\rq{}^i_s-_{F}\lambda\rq{}\rq{}^i_t+_Ft_1+_Ft_2}{\lambda\rq{}\rq{}^i_t-_{F}\lambda\rq{}^i_s}. 
\end{equation}
To each arrow $h\in H$, we associate two integers, $m_{h}$ and $m_{h^*}$. Define
\begin{small}\begin{equation}
\label{equ:fac2}
\fac_2:=\prod_{h\in H}\Big(
\prod_{s=1}^{v_1^{\out(h)}}
\prod_{t=1}^{v_2^{\inc(h)}}
(\lambda_t^{'' \inc(h)}-_{F}\lambda_s^{'\out(h)}+_{F} m_h \cdot t_1)
\prod_{s=1}^{v_1^{\inc(h)}}
\prod_{t=1}^{v_2^{\out(h)}}
(\lambda_t^{''\out(h)}-_{F}\lambda_s^{'\inc(h)}+_{F}m_{h^*}\cdot t_2)
\Big),
\end{equation}\end{small}
where $m \cdot t=t+_{F} t +_{F} \cdots +_{F} t$ is the summation of $m$ terms, for $m\in \N$.
When $m_h=m_{h^*}=1$, the formula \eqref{equ:fac2} can be simplified as
%\footnote{The simplified formula has to be interpreted correctly.}
\begin{equation*}
\fac_2=\prod_{i, j\in I}
\prod_{s=1}^{v_1^i}
\prod_{t=1}^{v_2^j}(\lambda\rq{}\rq{}^j_t-_{F}\lambda\rq{}^i_s+_F t_1)^{a_{ij}}
(\lambda\rq{}\rq{}^j_t-_{F}\lambda\rq{}^i_s+_F t_2)^{a_{ji}}, 
\end{equation*} where $a_{ij}$ is the number of arrows from $i$ to $j$ of quiver $Q$.
\Omit{
An equivalent formula for $\fac_2$ is: 
\begin{scriptsize}
\begin{align*}
\fac_2=\prod_{i, j\in I}
&\Bigg(\prod_{\begin{smallmatrix}h\in H, \\
\inc{h}=j, \out(h)=i\end{smallmatrix}}
\prod_{s=1}^{v_1^i}
\prod_{t=1}^{v_2^j}(\lambda\rq{}\rq{}^j_t-_{F}\lambda\rq{}^i_s+_Fm_h \cdot t_1)\Bigg)
\cdot\Bigg(
\prod_{\begin{smallmatrix}h\in H, \\
\inc{h}=i, \out(h)=j\end{smallmatrix}}
\prod_{s=1}^{v_1^i}
\prod_{t=1}^{v_2^j}
(\lambda\rq{}\rq{}^j_t-_{F}\lambda\rq{}^i_s+_Fm_{h^*}\cdot t_2)\Bigg).
\end{align*}
\end{scriptsize}}

The multiplication of $f_1(\lambda')\in \calS\calH_{v_1}$ and $f_2(\lambda'')\in \calS\calH_{v_2}$ is defined to be
\begin{equation}\label{shuffle formula}
\sum_{\sigma\in\Sh(v_1,v_2)}\sigma(f_1\cdot f_2\cdot \fac_1\cdot \fac_2)\in R[\![t_1,t_2]\!]
[\![\lambda^i_j]\!]_{i\in I, j=1,\dots, (v_1+v_2)^i}^{\fS_{v_1+v_2}}.
\end{equation}

\subsection{The geometric construction of the generalized shuffle algebra}
\label{subsec: shuffle geometric}

With notations as before, let $Q=(I, H)$ be a quiver. Let $\overline{Q}=Q\sqcup Q^{\op}$ be \textit{the double} of $Q$. That is, $\overline{Q}$ has the same vertex set as $Q$ and whose set of arrows is a disjoint union of the sets of arrows of $Q$ and $Q^{\op}$, the opposite quiver. To be more precise, the set of arrows of $\overline{Q}$ is $H\sqcup H^{\op}$. There is a bijection $H\to H^{\op}$, such that, for each $h\in H$, there is a reverse arrow $h^*\in H^{\op}$, with 
$\out(h^*)=\inc(h)$ and $\inc(h^*)=\out(h)$. We have the following isomorphisms
\[
\Rep(\overline{Q},v)\cong \Rep(Q, v)\times \Rep(Q^{\op}, v)\cong T^*\Rep(Q,v).
\]
The algebraic group $G_v$ acts on $T^*\Rep(Q,v)$ by conjugation. 
The torus $T=\Gm^2$ acts on $T^*\Rep(Q,v)$ the same way as in \cite[(2.7.1) and (2.7.2)]{Nak99}. Let $a$ be the number of arrows in $Q$ from vertex $i$ to $j$. We fix a numbering $h_1, \dots, h_a$ of the arrows in $Q$. The corresponding  reversed arrows in $Q^{\op}$ are labelled by $h_1^*, \cdots, h_a^*$. 
For $h_p \in H$, and $B=(B_p) \in \Hom(V^i, V^j)$, $B^*=(B^*_p) \in \Hom(V^j, V^i)$, define the $T=\G_m^2$--action by
 \[
 t_1\cdot B_p:=t_1^{m_{h_p}} B_p, \,\  t_2\cdot B_p^*:=t_2^{m_{h_p^*}} B_p^*.
 \] 
We assume the $T$-action on $T^*\Rep(Q,v)$ satisfies the following.
\begin{assumption} \label{Assu:WeghtsGeneral}
For any $h\in H$, $t_1^{m_h}t_2^{m_{h^*}}$ is a constant, i.e., does not depend on $h\in H$.
\end{assumption}

\begin{remark}\label{rmk:weights}
One example of  $T$ action satisfying Assumption~\ref{Assu:WeghtsGeneral} is the following. 
For any pair of vertices $i$ and $j$ with arrows  $h_1, \dots, h_a$ from $i$ to $j$, let the pairs of integers be $m_{h_p}=a+2-2p$ and $m_{h_p^*}=-a+2p$. 
Let the $T$ action on $T^*\Rep(Q,v)$ to factor through $\Gm$, i.e.,  via $t_1=t_2=\hbar/2$ where $\hbar$ is the weight of $\Gm$.
This choice is essential in \S~\ref{sec:Yangian} and \ref{sec:Yangian_action}.
 \end{remark}
For any pair of dimension vectors $v_1,v_2\in\bbN^I$, we consider a map \[m^S_{v_1, v_2}:A_{G_{v_1}\times T}(\Rep(\overline{Q}, v_1))\otimes_{R[\![t_1,t_2]\!]}A_{G_{v_2}\times T}(\Rep(\overline{Q}, v_2))\to A_{G_{v_1+v_2}\times T}(\Rep(\overline{Q}, v_1+v_2)),\] 
defined as follows. Let $v=v_1+v_2$.
In the Lagrangian correspondence formalism in \S~\ref{subsec: Lag corr}, we take $Y$ to be $\Rep(Q,v_1)\times \Rep(Q,v_2)$, $X'$ to be $\Rep(Q,v_1+v_2)$, and $\calV$ to be $\Rep(Q)_{v_1,v_2}$. Recall that
\[\Rep({Q})_{v_1, v_2}:=\{x\in \Rep({Q},v)\mid x(V_1)\subset V_1\}\subset\Rep(Q,v).\]
We write $G:=G_{v}$, and $P \subset G_v$, the parabolic subgroup preserving the subspace $V_1$.
Let $L:=G_{v_1}\times G_{v_2}$ be the Levi subgroup of $P$. 

As in  \S~\ref{subsec: Lag corr}, we have the following Lagrangian correspondence of $G\times T$-varieties:
\begin{equation*}
\xymatrix{
G\times_PT^*Y\ar@{^{(}->}[r]^{\iota}&T^*(G\times_PY)&Z \ar[l]_(0.3){\phi} \ar[r]^(0.3){\psi} & \Rep(\overline{Q}, v_1+v_2).
}
\end{equation*}
We now define the multiplication map $m_{v_1, v_2}^{S}$. We first have the K\"unneth isomorphism.
\[
\otimes: A_{G_{v_1}\times T}(\Rep(\overline{Q}, v_1))\otimes_{R[\![t_1,t_2]\!]}A_{G_{v_2}\times T}(\Rep(\overline{Q}, v_2))\cong A_{G_{v_1}\times G_{v_2}\times T}(\Rep(\overline{Q}, v_1)\times \Rep(\overline{Q}, v_2)). 
\] 
Consider the following sequence of morphisms:
\begin{enumerate}
\item The isomorphism:
$
A_{G_{v_1}\times G_{v_2}\times T}(\Rep(\overline{Q}, v_1)\times \Rep(\overline{Q}, v_2))
\cong A_{G\times T}(G\times_{P} T^*Y).
$
\item The pushforward map:
$
\iota_*: A_{G\times T}(G\times_P T^*Y) \to A_{G\times T}(T^*(G\times_PY)).
$
\item Following the notations in the Lagrangian correspondence diagram, we have
\[
\xymatrix{
A_{G\times T}(T^*(G\times_PY)) \ar[r]^(0.6){\phi^*} &A_{G\times T}(Z) \ar[r]^(0.3){\psi_*} &A_{G\times T}(\Rep(\overline{Q}, v)).
}\]
Note that $\psi$ is a proper morphism, hence the push-forward $\psi_*$ is well-defined.
\end{enumerate}
We define map $m^S_{v_1, v_2}$ to be the composition of the 
K\"unneth isomorphism with the above sequence of 3 morphisms.

\begin{prop}
\label{prop:assoc_shuffle}
The multiplication maps $m^S_{v_1, v_2}$ are associative.
\end{prop}
\begin{proof}
The proof follows the same idea as in \cite[Proposition 7.5]{SV2}. For the convenience of the readers, we include a proof here. 
Fix a flag $V_{1}\subset V_{2} \subset V$, where $V_i$ is an $I$-tuple subspaces of $V$ of dimension vector $v_1+\cdots+v_i$.  Let $P_1$, $P_{12}$ be the parabolic subgroups $P_1:=\{g \in G_v | g(V_1)\subset V_1\}$, and 
$
P_{12}:=\{g\in G_v \mid g(V_1)\subset V_1, g(V_2)\subset V_2\}.$

We first define the following varieties.
\begin{itemize}
\item Let $X_1$ be the set of quadruples $(F_1, F_2, a)$, where $F_1\subset F_2\subset V$ is a flag such that $F_1\cong V_1, F_2\cong V_2$, and $a\in \Rep(Q, v)$ is an endomorphism of the vector space 
$F_1\oplus (F_2/F_1)\oplus (V/F_2)$.
\Omit{$a$ is of the form 
$\left( \begin{array}{ccc}
* & 0 & 0\\
0 & * & 0\\
0&0& *
\end{array}\right)$.
In other words, \[X_1:= G_{v}\times_{P_{12}} (\Rep(Q, v_1)\times \Rep(Q, v_2-v_1))\times \Rep(Q, v-v_2).\]}
\item Let $X_2$ be the set of pairs $(F_1, a)$, where $F_1\subset V$, such that $F_1\cong V_1$ and $a \in \Rep(Q, v)$ is an endomorphism of the vector space $F_1\oplus (V/F_1)$. \Omit{$a$ is of the form 
$\left( \begin{array}{ccc}
* & 0 & 0\\
0 & * & *\\
0&*& *
\end{array}\right)$
In other words, $X_2:= G_{v}\times_{P_{1}} (\Rep(Q, v_1)\times \Rep(Q, v-v_1))$.}
\item $X_3=\Rep(Q, v)$.
\end{itemize}
We then define the following correspondences. For $i=1, 2, 3$, let $W_i$ be the following
\begin{align*}
W_1=\{(F_1&, a)\mid  F_1\subset V, \hbox{such that $F_1\cong V_1$, and } a(F_1) \subset F_1, \hbox{for $a\in \Rep(Q, v)$}\}.\\
\Omit{
&\hbox{a is of the form $\left( \begin{array}{ccc}
* & * & *\\
0 & * & *\\
0&*& *
\end{array}\right)$}\\}
W_2=\{(F_1&, F_2, a)\mid  F_1\subset F_2 \subset V, a \in \Rep(Q, v), \hbox{such that $F_j\cong V_j$, and }
a(F_j) \subset F_j,  \hbox{for $j=1, 2$}\}.
\\
\Omit{
&\hbox{a is of the form $\left( \begin{array}{ccc}
* & * & *\\
0 & * & *\\
0&0& *
\end{array}\right)$}\\}
W_3=\{(F_1&, F_2, a)\mid  F_1\subset F_2 \subset V, a\in \Rep(Q, v), \hbox{such that $F_j\cong V_j$, for $j=1, 2$, and }\\
& a\in \End(F_1\oplus V/F_1), \hbox{$a$ preserves the subspace $\{0\}\oplus F_2/F_1$}\}.
\Omit{
&\hbox{a is of the form $\left( \begin{array}{ccc}
* & 0 & 0\\
0 & * & *\\
0&0& *
\end{array}\right)$}}
\end{align*}

Let $Z_i$ be the conormal bundle of $W_i$. Consider the following commutative diagram with the square being Cartesian.
\[\xymatrix@R=1.5em{
T^*X_3&Z_{1}\ar[l]_{\psi_1}\ar[r]^{\phi_1}&T^*X_2\\
&Z_{2}\ar[u]\ar[lu]^{\psi_2}\ar[dr]_{\phi_2}\ar[r]&Z_{3}\ar[d]^{\phi_3}\ar[u]_{\psi_3}\\
&&T^*X_1.
}\]
By Lemma~\ref{Lem:base_change} and Lemma~\ref{lem:dim_lag}, we have $I_2=I_1\circ I_3$, where
\begin{align*}
I_1&= \psi_{1*}\circ\phi_1^*: A_{G\times T}(T^* X_2) \to A_{G\times T}(T^* X_3).\\
I_2&= \psi_{2*}\circ\phi_2^*: A_{G\times T}(T^* X_1) \to A_{G\times T}(T^* X_3).\\
I_3&= \psi_{3*}\circ\phi_3^*: A_{G\times T}(T^* X_1) \to A_{G\times T}(T^* X_2).
\end{align*}

An argument similar to \cite[Lemma 3.4]{L91} implies the associativity.
\end{proof}

Now assume $A$ is an arbitrary oriented Borel-Moore homology theory, with $\bbQ\subseteq A(\pt)$, extended to an equivariant Borel-Moore homology by Totaro's construction (see Example \ref{exam:EOCT}). For any $v\in\bbN^I$, we identify \[\calS \calH_v:=R[\![t_1,t_2]\!]
[\![\lambda^i_s]\!]_{i\in I, s=1,\dots, v^i}^{\fS_v}\] with the $A_T(\pt)$-module $A_{G_v\times T}(\Rep(\overline{Q}, v))$. Such identification comes from the extended homotopy equivalence property of $A$, i.e.,  
\[
A_{G_v\times T}(\Rep(\overline{Q}, v))\cong A_T(\pt)\otimes A(BG_v)\cong R[\![t_1,t_2]\!]
[\![\lambda^i_s]\!]_{i\in I, s=1,\dots, v^i}^{\fS_v}.
\]

We  identify the algebraically defined formal shuffle algebra multiplication in \S~\ref{subsec:formal shuffle} with the geometrically defined $m^S_{v_1,v_2}$ as follows. 
\begin{prop}\label{thm: shuffle formula}
Assume $A$ is an arbitrary oriented Borel-Moore homology theory, with $\bbQ\subseteq A(\pt)$, extended to an equivariant Borel-Moore homology by Totaro's construction. Under the identification 
\[\calS \calH_v:=R[\![t_1,t_2]\!]
[\![\lambda^i_s]\!]_{i\in I, s=1,\dots, v^i}^{\fS_v} \cong A_{G_v\times T}(\Rep(\overline{Q}, v)),\] 
the map $m^S_{v_1,v_2}$ is equal to the multiplication map  \eqref{shuffle formula} of the shuffle algebra.
\end{prop}
\begin{proof}
Let $e_{v_1,v_2}^{\iota}$ be the equivariant Euler class of the normal bundle of the embedding
\[\iota: G_{v}\times_{G_{v_1,v_2}}T^*Y\inj T^*(G_{v}\times_{G_{v_1,v_2}}Y). \] 
\Omit{We have the isomorphism
 \begin{align*}
&A_{G\times T}(G\times_PT^*Y) \cong A_{G\times T}(G/P)\\
&A_{G\times T}(T^*(G \times_{P}Y))\cong A_{G\times T}(T^*G/P).
\end{align*}}
The normal bundle of $\iota$ is isomorphic to $T^*G/P$ as a bundle over the Grassmannian $G/P$. Also $T^*G/P$ is in turn isomorphic to $\bigoplus_{i\in I}(\calR(v_1^i)\otimes \calR(v_2^i)^*)$. 
Thus we have
 \[e_{v_1,v_2}^{\iota}=\prod_{i\in I}\prod_{s=1}^{v_1^i}
\prod_{t=1}^{v_2^i}(\lambda\rq{}^i_s-_{F}\lambda\rq{}\rq{}^i_t+_Ft_1+_Ft_2).\]
Therefore, $\iota_*$ is multiplication by $e_{v_1,v_2}^{\iota}$. 

%By definition, \[A_{G_{v_1}\times G_{v_2}\times T}(\Rep(\overline{Q}_{v_1, v_2})):=A_T(E(G_{v_1}\times G_{v_2})\times_{G_{v_1}\times G_{v_2}} \Rep(\overline{Q}_{v_1, v_2}))\] and \[A_{G_{v_1+v_2}\times T}(\Rep(\overline{Q}, v_1+v_2)):=A_T(E(G_{v_1+v_2})\times_G\Rep(\overline{Q}, v_1+v_2)).\]

The composition $Z:=T^*_W(X\times X')\to W\to G/P$ is a vector bundle, where the second map is the natural projection.
It induces a morphism $EG\times_GZ\to EG\times_G(G/P)\cong BL$.
Recall that \[BL\cong\Gr(v_1,\infty)\times\Gr(v_2,\infty)\cong EG/P\xrightarrow{p}BG.\]
Note that $EG\times_GT^*\Rep(Q,v)$ is a vector bundle over $BG$. Let $p^*T^*\Rep(Q,v)$ be the pull-back vector bundle on $BL$ via the projection $p:BL\to BG$. 
The natural map $\psi:EG\times_GZ\to EG\times_GT^*\Rep(Q,v)$ factors through $\psi_1:EG\times_GZ\to p^*T^*\Rep(Q,v)$ by the universality of the pullback. 
We summarize these notations in the following diagram
\begin{equation*}
\xymatrix@C=1em @R=1.5em{
EG\times_G Z \ar@/^1.5pc/[drr]^{\psi}\ar@{_{(}-->}[dr]^{\psi_1}\ar@/_1pc/[ddr]\\
&p^*T^*\Rep({Q}, v)\ar[d]^{\pi} \ar[r]^{p'}& EG\times_GT^*\Rep({Q}, v)\ar[d]^{\pi'}\\
&BL%=\Grass(v_1, \infty)\times \Grass(v_2, \infty)
\ar[r]^(0.3)p & BG=\Grass(v, \infty).
}
\end{equation*}The pushforward $ \psi_*$ is the composition $p'_* \circ  \psi_{1*} $.  
%Consequently, there is a natural isomorphism  \[A_{G\times T}( G\times_P(\Rep(Q)_{v_1, v_2}\times \Rep(Q^{\op})_{v_1, v_2}))\cong A_{P\times T}(\Rep(\overline{Q})_{v_1, v_2}).\]21

The  map $\psi_1: EG\times_GZ\hookrightarrow p^* T^*\Rep(Q, v)$ is an embedding of vector bundles on $BG$.
The pushforward $\psi_{1*}$ is the multiplication by the equivariant Euler class $e_{v_1, v_2}^{\psi_1}$ of the normal bundle of the embedding $\psi_1$. 
The normal bundle to $\psi_1$ can be identified with 

\begin{align*}
&\sHom_{\overline{Q}}(\calR(v_1), \calR(v_2))\\
=&\bigoplus_{h\in H}\sHom_{\calO}\big(\calR(v_1^{\out(h)}),\calR(v_2^{\inc(h)})\big)\bigoplus\bigoplus_{h\in H^{op}}\sHom_{\calO}\big(\calR(v_1^{\out(h)}),\calR(v_2^{\inc(h)})\big)
\end{align*}
over
$BL=\Grass(v_1, \infty)\times \Grass(v_2, \infty)$. 
%Here $\calR(r)$ is the tautological bundle of $\Grass(r, \infty)$.
Here, the first copy $\bigoplus_{h\in H}\sHom_{\calO}(\calR(v_1^{\out(h)}),\calR(v_2^{\inc(h)}))$ is considered as a subspace of $\Rep(Q,v)$, and the second copy $\bigoplus_{h\in H^{op}}\sHom_{\calO}(\calR(v_1^{\out(h)}),\calR(v_2^{\inc(h)}))$ as a subspace of $\Rep(Q^{op},v)$.
Thus, the equivariant Euler class of the normal bundle to $\psi_1$ is
\[
e_{v_1, v_2}^{\psi_1}=
\prod_{h\in H}\Big(
\prod_{s=1}^{v_1^{\out(h)}}
\prod_{t=1}^{v_2^{\inc(h)}}
(\lambda_t^{'' \inc(h)}-_{F}\lambda_s^{'\out(h)}+_{F} m_h\cdot t_1)
\prod_{s=1}^{v_1^{\inc(h)}}
\prod_{t=1}^{v_2^{\out(h)}}
(\lambda_t^{''\out(h)}-_{F}\lambda_s^{'\inc(h)}+_{F} m_{h^*} \cdot t_2)
\Big).\]

So far, we obtained that $\iota_*$ is multiplication by $e_{v_1, v_2}^{\iota}$, and $\psi_{1*}$ is multiplication by $e_{v_1, v_2}^{\psi_1}$. 
The map $p'$ is the pull-back of $p$ via the projection $\pi'$. Therefore, $p'$ is also a Grassmannian bundle,
and consequently $p'_*$ is given by Proposition~\ref{prop:grass_push}.
Putting all the above together, the map $m_{v_1, v_2}^S$ is given by exactly the same formula as  \eqref{shuffle formula}.
\end{proof}

\begin{remark}
When $A$ is the equivariant Chow group with $\bbQ$-coefficients, under the identification, 
\[\calS \calH_v:=\bbQ[t_1,t_2]
[\lambda^i_s]_{i\in I, s=1,\dots, v^i}^{\fS_v} \cong A_{G_v\times T}(\Rep(\overline{Q}, v)),\] 
Proposition~\ref{thm: shuffle formula} still holds after replacing the formal power series by power series.
Similar for the equivariant $K$-theory, which will be spelled out in detail below. 
\end{remark}

\subsection{An example: $K$-theoretical shuffle algebra}
As an example, we take $A$ to be the $K$-theory with rational coefficients.  
We relate the shuffle algebra $\calS\calH$ with the Feigin-Odesskii shuffle algebra ( see \cite{FO}) in the Jordan quiver case.

For any line bundle $p: L\to X$ on a smooth quasi-projective variety $X$. Let $s: X\to L$ be the zero-section. The resolution of $s_* \calO_X$:
\[
1\to p^*L^{\vee} \to \calO_L \to s_* \calO_X \to 0
\] implies that the first Chern class of $L$ in the K-theory is 
\begin{equation}\label{eqn_K_Chern}
c_1(L):=s^*s_*(\calO_X)=1-L^\vee.
\end{equation}
As a consequence, the formal group law $(K(\pt), F_m)$  is  
\begin{equation}\label{eqn:gm1}
F_m(u, v)=u+_{F_m}v=u+v-uv.
\end{equation}
In particular, \begin{equation}\label{eqn:gm2}
u-_{F_m}v=\frac{u-v}{1-v}.
\end{equation}

For $r\in\N$, let $\calR(r)$ be the tautological vector bundle of $\Grass(r,\infty)$ and let $\Fl(\calR(r))$ be the associate full flag bundle. We identify $K(\Grass(r,\infty))$ with $\QQ[z_1^\pm,\dots,z_r^\pm]^{\fS_r}$ with $z_i$ being the $i$-th tautological line bundle (rather than its 1st Chern class) of $\Fl(\calR(r))$. 
Let $s_i$ be the one-dimensional natural representation of the $i$-th copy of $\Gm$ in $T$. We have the following.
\begin{corollary} Let $Q$ be any quiver, and $A$ be the $K$-theory with rational coefficients.
Assume $m_h=m_{h^*}=1$ for any $h\in H$. 
For any $v\in\bbN^I$, identify $\calS\calH_v$ with $\QQ[(z_1^{i})^{\pm},\dots,(z_{v^i}^{i})^{\pm}]_{i\in I}^{\fS_v}$ as above. 
For any pair of dimension vectors $v_1,v_2\in\bbN^I$, and $f_i\in \calS \calH_{v_i}$ for $i=1, 2$,
$m_{v_1,v_2}(f_1\otimes f_2)$ is equal to
\[
%------f_1----f_2----
\sum_{\sigma\in\Sh(v^1,v-v^2)}\sigma\left(
f_1(z\rq{}^i_s)f_2(z\rq{}\rq{}^i_t)
%------ fac_1------
\prod_{i\in I}\prod_{s=1}^{v_1^i}\prod_{t=1}^{v_2^i}
\frac{1- \frac{z\rq{}\rq{}^i_t}{z\rq{}^i_s s_1s_2}}
{1-\frac{z\rq{}^i_s}{z\rq{}\rq{}^i_t}}
%-------fac_2------
\cdot\prod_{i, j\in I}\prod_{s=1}^{v_1^i}\prod_{t=1}^{v_2^j}
(1-\frac{z\rq{}^i_s}{z\rq{}\rq{}^j_t s_1})^{a_{ij}}
(1-\frac{z\rq{}^i_s}{z\rq{}\rq{}^j_t s_2})^{a_{ji}}\right).\]
\end{corollary}
\begin{proof}
Recall that $t_i$ is the first Chern class of $s_i$, hence by \eqref{eqn_K_Chern}, we have $t_i=1-\frac{1}{s_i}$, for $i=1,2$. 
Similarly, $\lambda_s^i=1-\frac{1}{z_s^i}$.
Plugging-in to Proposition~\ref{thm: shuffle formula} while using \eqref{eqn:gm1} and \eqref{eqn:gm2}, we get this corollary.
\end{proof}

\begin{example}
Let $Q$ be the Jordan quiver and $A$ be the $K$-theory with rational coefficients. Then as a vector space, $\calS\calH\cong \bigoplus_n \QQ[s_1^\pm,s_2^\pm][z_1^\pm,\dots,z_n^\pm]^{\fS_n}$. The multiplication $\calS\calH_r\otimes\calS\calH_{n-r}\to \calS\calH_n$ sends $f_1(z_1,\dots,z_r)\otimes f_2(z_{r+1},\dots,z_n)$ to 
\[
\sum_{\{\sigma\in \hbox{Sh}(r, n-r)\}}
\sigma\cdot \Big(f_1 \cdot f_2 \cdot \prod_{1\leq j \leq r,  (r+1) \leq i\leq n}\frac{(1-\frac{z_j}{z_is_1s_2})(1-\frac{z_i}{z_js_1})(1-\frac{z_i}{z_js_2})}{1-\frac{z_i}{z_j}} \Big).
\]

Let $\mathcal{SH}'$ be the shuffle algebra defined by Feigin-Odesskii in \cite{FO} (see also \cite{FT}). By definition, $\mathcal{SH}'\cong \bigoplus\QQ[q_1^\pm,q_2^\pm][z_1^\pm,\dots,z_n^\pm]^{\fS_n}$, with multiplication $f_1(z_1,\dots,z_r)\otimes f_2(z_{r+1},\dots,z_n)$ defined by
\begin{align*}
\sum_{\{\sigma\in \hbox{Sh}(r, n-r)\}}
\sigma\cdot \Big(f_1 \cdot f_2 \cdot \prod_{1\leq j \leq r,  (r+1) \leq i\leq n}\frac{(1-q_1z_i/z_j)(1-q_2z_i/z_j)}{(1-z_i/z_j)(1-q_1q_2z_i/z_j)} \Big).
\end{align*}
Define an algebra homomorphism $\mathcal{SH}'\to \mathcal{SH}$ by
\begin{align*}
& q_i\mapsto s_i^{-1}\hbox{, for }i=1, 2;\\
& f(z_1, \dots, z_n) \mapsto f(z_1, \dots, z_n) Y_n\hbox{, for }n\in\bbN, 
\end{align*}
where
\[
Y_n=\prod_{1\leq j<i\leq n}\big((1-q_{1}q_2\frac{z_j}{z_i})(1-q_{1}q_2\frac{z_i}{z_j})\big).
\] 
This map is well-defined, since the factor $Y_n$ is invariant under the action of $\hbox{Sh}(r, n-r)$.
A straight-forward calculation shows that this is an algebra homomorphism.
\Omit{
The action of $\mathcal{SH}'$ on the $K$-theory of the Hilbert scheme of points on $\Aff^2$ has been constructed by Feigin and Tsymbaliuk in \cite{FT} and independently studied by Schiffmann and Vasserot in \cite{SV2}.}
\end{example}

When $Q$ is the cyclic quiver, and $A$ is the $K$-theory, it is shown in \cite{Neg} that the shuffle algebra is isomorphic to the positive half of the quantum toroidal algebra of type $A$.

\section{The preprojective cohomological Hall algebras}
\label{preproj CoHA}
In this section, we introduce the main object to be studied in this paper, the preprojective CoHA.  The representations of this algebra  are realized as $A$-homology of Nakajima quiver varieties. Now we describe the multiplication of  this algebra.

\subsection{Hall multiplication}\label{subsec:HallMulti}
Notations are as before. Let $Q=(I, H)$ be a quiver, (not necessarily edge-loop free) and let  $v\in\bbN^I$ be a dimension vector. The group $G_{v}:=\prod_{i\in I}\GL_{v^i}$ acts on the cotangent space $T^*\Rep(Q, v)$ via conjugation. Let $\fg_{v}$ be the Lie algebra of $G_v$. Let 
\[\mu_{v}: T^*\Rep(Q,v)\to \fg^*_{v}, \,\ (x, x^{*})\mapsto [x, x^{*}]\] be the moment map. 
Note that the closed subvariety $\mu_v^{-1}(0)\subset T^*\Rep(Q,v)$ could be singular in general.

We consider the $\bbN^I$-graded $R[\![t_1, t_2]\!]$-module $^A\calP(Q):=\bigoplus_{v\in \N^I}\null^A\calP_v(Q)$ with $^A\calP_v(Q)= A_{G_v\times T}(\mu_v^{-1}(0))$.
When both $A$ and $Q$ are understood from the context, we will use the abbreviations $\calP$ and $\calP_v$.
For each pair $v_1,v_2\in\bbN^I$, we define the multiplication map $m_{v_1, v_2}^{P}:\calP_{v_1}\otimes_{R[\![t_1t_2]\!]}\calP_{v_2}\to \calP_{v_1+v_2}$. 

We write $v=v_1+v_2$.
We consider the Lagrangian correspondence formalism in \S~\ref{subsec: Lag corr}, with the following specializations: Take $Y$ to be $\Rep(Q,v_1)\times \Rep(Q,v_2)$, $X'$ to be $\Rep(Q,v)$ and  $\calV$ to be $\Rep(Q)_{v_1,v_2}$. As in \S~\ref{subsec: shuffle geometric}, we write $G:=G_{v}$ for short. Let $P\subset G_v$ be the parabolic subgroup and $L:=G_{v_1}\times G_{v_2}$ be the Levi subgroup of  $P$. 
Recall in \S~\ref{subsec: Lag corr}, we have the following correspondence of $G\times T$-varieties:
\begin{equation*}
\xymatrix@R=1.5em{
G\times_P\left(\mu_{v_1}^{-1}(0)\times\mu_{v_2}^{-1}(0)\right)\ar@{=}[r]\ar@{^{(}->}[d]&T^*_GX\ar@{^{(}->}[d]&Z_G\ar[l]_{\overline\phi}\ar[r]^{\overline\psi}\ar@{^{(}->}[d]&\mu_{v}^{-1}(0)\ar@{^{(}->}[d]\\
G\times_PT^*Y\ar@{^{(}->}[r]^{\iota}&T^*(G\times_PY)&Z \ar[l]_(0.3){\phi} \ar[r]^(0.3){\psi} & \Rep(\overline{Q}, v).
}
\end{equation*}

We have the K\"unneth morphism (which may or may not be an isomorphism).
\[
\otimes: \calP_{v_1}\otimes_{R[\![t_1,t_2]\!]}\calP_{v_2}\to 
A_{G_{v_1}\times G_{v_2}\times T}(\mu_{v_1}^{-1}(0)\times \mu_{v_2}^{-1}(0)). 
\] 
Consider the following sequence of morphisms:
\begin{enumerate}
\item The natural projection $G_{v_1}\times G_{v_2} \twoheadleftarrow P$ is a homotopy equivalence.
It induces the following isomorphism
$
A_{G_{v_1}\times G_{v_2}\times T}(\mu_{v_1}^{-1}(0)\times \mu_{v_2}^{-1}(0))
\cong A_{P \times T}(\mu_{v_1}^{-1}(0)\times \mu_{v_2}^{-1}(0))$.
We have the following isomorphism
\[
A_{P \times T}(\mu_{v_1}^{-1}(0)\times \mu_{v_2}^{-1}(0))\cong A_{G \times T}\big(G\times_P(\mu_{v_1}^{-1}(0)\times \mu_{v_2}^{-1}(0))\big).\]
\item Following the Lagrangian correspondence diagram, we have
\[
\xymatrix{
A_{G\times T}(T^*_GX) \ar[r]^{\phi^\sharp} &A_{G\times T}(Z_G) \ar[r]^(0.4){\overline\psi_*} &A_{G\times T}(\mu_v^{-1}(0)) \cong \calP_{v},
} \]
where $\phi^\sharp$ is the Gysin pullback of $\phi$. 
\end{enumerate}
The  map $m_{v_1, v_2}^P$ is defined to be the composition of the above morphisms.

\begin{theorem}\label{prop:assoc_hall}
The maps $m^{\calP}_{v_1,v_2}$ fit together to define an associative algebra structure on $\calP$.
\end{theorem}
\begin{proof}
We keep the same notations as in the proof of Proposition~\ref{prop:assoc_shuffle}.
By definition, $T^{*}_{G} X_3=\mu_{v}^{-1}(0)$. By Lemma \ref{lem:iso of G}, $T^*_{G} X_2=G_{v}\times_{P_{1}}(\mu^{-1}_{v_1}(0)\times \mu_{v_2+v_3}^{-1}(0))$ and
\[
T^*_{G} X_1=G_{v}\times_{P_{12}}
(\mu^{-1}_{v_1}(0)\times \mu_{v_2}^{-1}(0) \times \mu_{v_3}^{-1}(0)).
\] 
By Lemma~\ref{Lem:base_change} and Lemma~\ref{lem:dim_lag}, we have $\overline{I_2}=\overline{I_1}\circ \overline{I_3}$, where
\begin{align*}
\overline{I_1}&= \overline{\psi_{1}}_*\circ\phi_1^\sharp: A_{G\times T}(T^*_G X_2) \to A_{G\times T}(T^*_G X_3).\\
\overline{I_2}&= \overline{\psi_{2}}_*\circ\phi_2^\sharp: A_{G\times T}(T^*_G X_1) \to A_{G\times T}(T^*_G X_3).\\
\overline{I_3}&= \overline{\psi_{3}}_*\circ\phi_3^\sharp: A_{G\times T}(T^*_G X_1) \to A_{G\times T}(T^*_G X_2).
\end{align*}
An argument similar to \cite[Lemma 3.4]{L91} implies the associativity of the multiplication $m^{\calP}$.
\end{proof}
\begin{definition}
For any $v\in \bbN^I$, consider the $\bbN^I$-graded $R[\![t_1,t_2]\!]$-module $^A\calP(Q):=\bigoplus_{v\in \N^I} \null^A\calP_v$, where $^A\calP_v:=A_{G_v\times T}(\mu_v^{-1}(0))$. 
The preprojective cohomological Hall algebra (CoHA) of the quiver $Q$ is the associative algebra 
$^A\calP(Q)$ endowed with the Hall multiplication $m^{\calP}_{v_1,v_2}$.
\end{definition}
The name preprojective CoHA is motivated by the fact that the subvariety $\mu_{v}^{-1}(0) \subset \Rep(\overline Q, v)$ parametrizes representations of the preprojective algebra.

\begin{theorem}\label{thm:prepro to shuffle}
Under conditions of Proposition~\ref{thm: shuffle formula}, there is a well-defined morphism of $R[\![t_1,t_2]\!]$-algebras \[^A\calP\to\calS\calH \] induced from the embedding
$i_v: \mu^{-1}_{v}(0) \inj \Rep(\overline{Q}, v)$.
\end{theorem}

\begin{proof}
The pushforward $i_{v*}$ induces a well-defined morphism 
\[
i_{v*}:  A_{G\times T}(\mu^{-1}_{v}(0)) \to A_{G\times T}(\Rep(\overline{Q}, v))\cong \calS\calH_v.\]
Recall that the Hall multiplication of $^A\calP$ (resp. $\calS\calH$) is defined by the composition of the K\"unneth morphism and the morphism in the first row (resp. second row) in the commutative diagram in Lemma~\ref{lem:gysin+pullback_comm}. Taking into account the compatibility of the K\"unneth morphism  with the proper pushforward (\cite[\S~2.1]{LM}), together with 
Lemma~\ref{lem:gysin+pullback_comm}, we get the conclusion.
\end{proof}
\begin{remark}\label{rmk:torsion}
The algebra homomorphism in Theorem \ref{thm:prepro to shuffle} becomes an isomorphism after  suitable localization. Indeed,  $\mu_v^{-1}(0)$ has only one $T$-fixed point. 
It follows from the Thomason localization theorem \cite[Theorem~6.2]{GKM}, which in the present setting can be found in \cite{Kr}, that $i_{v*}$ is an isomorphism when passing to a  localization. This localization of $R[\![t_1,t_2,\lambda^i_s]\!]^{\fS_v}$ is at the prime ideal  generated by all the symmetric functions in $\lambda^i_s$ without constant terms. For the power series ring, this is the same as passing to  $R(\!(t_1,t_2)\!)$. However, for OCT's described in Example\ref{exam:EOCT}(3), e.g., equivariant Chow groups and $K$-theory in the sense of Thomason, these two localizations are different.\footnote{We thank a referee for pointing this out to us.}

Let $\underline{\calP}$ be the quotient of $\calP$ by the torsion part, where torsion means elements which vanishes when passing to the localization at the above ideal. Then, $i_{v*}$ induces an isomorphism between $\underline{\calP}^{\sph}$ and $\calS\calH^{\sph}$.
\end{remark}

\subsection{Spherical subalgebras}
In general, the algebra $\calP$ defined above and the shuffle algebra $\calS\calH$ in \S~\ref{sec:formalShuf} are different, but closely related. Each one has a spherical subalgebra. Conjecturally, their spherical subalgebras are  isomorphic, but at present, this is not known.

For any vertex $k\in I$ of the quiver $Q$, let $e_k$ be the dimension vector such that $e_k^i=\delta_{ki}$. In other words, $e_k$ has value $1$ on vertex $k$, and $0$ otherwise. 
\begin{definition}
The spherical subalgebra $^A\calP^{\sph}(Q) \subset \null^A\calP(Q)$ is the subalgebra  generated by $\calP_{e_k}$ for $k$ varies in $I$. We will abbreviate as $^A\calP^{\sph}$ or $\calP^{\sph}$ if understood.
Similarly, define $\calS\calH^{\sph} \subset \calS\calH$ to be the subalgebra of $\calS\calH$ generated by the set $\{\calS\calH_{e_k}\mid  k\in I\}$.
\end{definition}

\begin{prop}
The morphism in Theorem~\ref{thm:prepro to shuffle} restricts to a surjective morphism on the spherical subalgebras \[\calP^{\sph}\surj \calS\calH^{\sph}.\]
\end{prop}
\begin{proof}
The surjectivity of the restriction 
$\calP^{\sph}\to \calS\calH^{\sph}$ follows from  the isomorphism $\calP_{e_k}\cong \calS\calH_{e_k}$, for $k\in I$. Indeed, this is a consequence of the isomorphisms $\calP_{e_k}\cong A_{G_{e_k}\times T}(\pt)=A_{T}(\pt)[\![z_k]\!]$ and $\calS\calH_{e_k}\cong A_{T}(\pt)[\![\lambda^k]\!]$.
\end{proof}

\section{Representations of the preprojective CoHA}
We construct representations of  $^A\calP(Q)$ in this section. We show that $^A\calP(Q)$ acts on the equivariant $A$-homology of Nakajima quiver varieties.
\subsection{Preliminaries on Nakajima quiver variety}
\label{subsec:moment_quiver}
In this subsection, we recall the definition of Nakajima quiver varieties in \cite{Nak94}. 

For a quiver $Q$, (not necessarily edge-loop free), we introduce the \textit{ framed quiver} $Q^\heartsuit$, whose set of vertices is $I \sqcup I'$, where $I'$ is another copy of the set $I$, equipped with the bijection $I\to I'$, $i\mapsto i'$. 
The set of arrows of $Q^\heartsuit$ is, by definition, the disjoint union of $H$ and a set of additional edges $j_i : i \to i'$, one for each vertex $i\in I$.
We follow the tradition that $v\in \bbN^I$ is the notation for the dimension vector at $I$, and $w\in \bbN^I$ is the dimension vector at $I'$. We denote $\Rep(Q^\heartsuit,(v,w))$ simply by $\Rep(Q,v,w)$.

Let $\overline{Q^{\heartsuit}}=Q^{\heartsuit}\sqcup Q^{\heartsuit, \op}$ be the double of $Q^{\heartsuit}$. We have the isomorphism
\[
\Rep(\overline{Q^\heartsuit},(v,w))=T^*\Rep(Q^\heartsuit,v,w)=\Rep(Q, v)\times \Rep(Q^{op}, v)\times \Hom_{k I}(W, V)\times \Hom_{k I}(V, W).
\]
Let $\mu_{v, w}:T^*\Rep(Q^\heartsuit,v,w)\to \fg\fl_v^*\cong \fg\fl_v$ be the moment map
 \[
\mu_{v, w}: (x, x^*, i, j)\mapsto 
\sum[x, x^*]+i\circ j \in \fg\fl_v.
\]
For any $\theta=(\theta_i)_{i\in I} \in \Z^{I}$, 
let $\chi_\theta: G_v\to \Gm$ be the character $g=(g_i)_{i\in I} \mapsto \prod_{i\in I} \det (g_i)^{-\theta_i}.$ 
The set of $\chi_\theta$-semistable points in $T^*\Rep(Q^\heartsuit, v, w)$ is denoted by $\Rep(\overline{Q^\heartsuit},v,w)^{ss}$.
The Nakajima quiver variety is defined to be the Hamiltonian reduction 
\[\fM_{\theta}(v, w):=\mu_{v, w}^{-1}(0)/\!/_\theta G_v.\] 

The following description of stability condition can be found in \cite[Corollary 5.1.9]{G}.
When $\theta=\theta^+=(1,\cdots,1)$, the point $(x, x^*, i, j)\in \mu^{-1}(0)$ is $\theta^+$--semistable, if and only if, the following holds:
For any collection of vector subspaces $S = (S_i)_{i\in k}\subset
V = (V_i)_{i\in k}$,  which is stable under the maps
$x$ and $x^*$, if
$S_k\subset \ker(j_k)$ for every $k\in I$, then $S=0$.

In this paper, we use the stability condition $\theta^+$ unless otherwise specified.
\subsection{Representations from quiver varieties}
\label{subsec:PreproRepn}
Let $v_1,v_2\in\bbN^I$ be two dimension vectors. Let $v=v_1+v_2$. 
We fix an $I$-tuple of vector spaces $V$ of dimension vector $v$.
We fix $V_1\subset V$ an $I$-tuple of subspaces of $V$ with dimension vector $v_1$. 
Let $V_2:=V/V_1$, with the projection map $\pr_2: V\to V_2$. 
 As in \S~\ref{subsec: Lag corr}, we set $G=G_{v}$, and $P=\{g\in G\mid g(V_1)\subset V_1\}$. 
We consider the Lagrangian correspondence formalism in \S~\ref{subsec: Lag corr},
specialized as follows:
We take $X'$ to be $\Rep(Q,v, w)$ and $Y$ to be $\Rep(Q,v_1, w)\times\Rep(Q,v_2)$, and take $\calV$ to be
\[
\Rep(Q)_{v_1, v_2, w}:=\{(x, j)\in \Rep(Q^{\heartsuit}, v_1+v_2, w) \mid x(V_1) \subset V_1\} \subset X'. 
\]
As in \S~\ref{subsec: Lag corr}, set $X:=G\times_{P} Y$, $W:=G\times_P \Rep(Q)_{v_1, v_2, w}$, and $Z:=T^*_W(X\times X')$ the conormal bundle of $W$. 
We then have the correspondence 
\[\xymatrix{
T^*X & Z\ar[l]_{\phi} \ar[r]^{\psi} &T^*X'.}\]
\begin{lemma}\label{lem:descrip of Z_G}
Notations are as above.
\begin{enumerate}[(a)]
\item We have the following canonical isomorphisms of $G$-varieties.
\begin{align*}
&T^*X'=T^*\Rep(Q^{\heartsuit}, v_1+v_2, w)=\Rep(Q^{\heartsuit}, v_1+v_2, w)\times \Rep((Q^{\heartsuit})^{\op}, v_1+v_2, w)=\Rep(\overline{Q^{\heartsuit}}, v_1+v_2, w).\\
&T^*X=G\times_P
\{
(c, x, x^*, i, j)\mid c \in \fp_v , x \in \Rep( Q, v_1)\times \Rep( Q, v_2),
x^*\in \Rep( Q^{\op}, v_1)\times \Rep( Q^{\op}, v_2),
\\ & \phantom{1234567890} j\in \Hom(V_1, W), i\in \Hom(W, V_1), [x, x^*]+i\circ j=\pr(c)
\},\\
&\phantom{1234}\text{ where $\pr(c)$ is the projection of $c$ from $\fp_v$ in $\fg_{v_1}\oplus \fg_{v_2}$}.\\
&Z=G\times_P\{ (x, x^*, i, j )\in T^*\Rep(Q^{\heartsuit}, v_1+v_2, w)\mid (x, x^*)(V_1)\subset V_1,  \text{Im}(i)\subset V_1\}.
\end{align*}
\item
For $(g, x, x^*, i, j)\in Z$, the maps $\phi$, $\psi$ are given by
\begin{align*}
&\phi\big((g, x, x^*, i, j) \mod P\big)=\big(g, [x, x^*]+i\circ j, \pr(x), \pr(x^*), i^{V_1}, j_{V_1}\big) \mod P, 
\\
&\psi\big((g, x, x^*, i, j) \mod P\big)=\big( gxg^{-1}, gx^*g^{-1}, j g^{-1}, gi\big) . 
\end{align*}
\item We have the following canonical isomorphisms of $G$-varieties.
\begin{align*}
&T_{G}^* X'=\mu_{v, w}^{-1}(0).\\
& T_{G}^* X=G\times_P(\mu_{v_1, w}^{-1}(0)\times \mu_{v_2}^{-1}(0)).\\
& Z_{G}=G\times_P\{ (x, x^*, i, j )\in \mu^{-1}_{v, w}(0)\mid (x, x^*)(V_1)\subset V_1,   \text{Im}(i)\subset V_1\}.
\end{align*}
The maps $\overline{\phi}: Z_G \to T_G^*X$ and $\overline{\psi}: Z_G\to T^*_GX'$ are induced from $\phi, \psi$ in (b).
\end{enumerate}
\end{lemma}
\begin{proof}
The proof goes the same way as \cite[Lemma 7.4]{SV2}. 
We only explain how to get the formula of $T^*X$ in (a) here. The rest are similar.
By  \cite[Lemma 7.1]{SV2}, we have $T^*X=T^*_{P}(G\times Y)/P$.
Thus, 
\begin{align}
T^*X=G\times_{P}\{(f, a)\in &\big(\mathfrak{g}\times \Rep(Q^\heartsuit, v_1, w) \times \Rep(Q, v_2)\big)^*\times \big(\Rep(Q^\heartsuit, v_1, w) \times  \Rep(Q, v_2)\big) \notag\\
&\mid f(-b, [\pr(b), a])=0, \forall b\in \fp\}.\label{equ: T*X}
\end{align}
%\textcolor{blue}{This follows from the definition of the conformal %bundle.}
%It remains to check this description of $T^*X$ coincides with the description in the Lemma. 
For $(f, a)$ as in \eqref{equ: T*X}, we write $f=f_1\times f_2$, where 
\[
f_1\in \mathfrak{g}^*, \,\ \text{and $f_2\in \big(\Rep(Q, v_1, w) \times \Rep(Q, v_2)\big)^*$}.
\]
Write $\fh:=\mathfrak{gl}_{v_1}\times \mathfrak{gl}_{v_2}$ for short. 
Starting with an element $(f, a)$ in \eqref{equ: T*X}, we define an element $\widetilde{f}\in (\mathfrak{g}\times \fh)^*$ by $\widetilde f(g, h):=f(g, [h, a])$. Let $\delta': \fp\to \fg\times \fh$ be the linear function $b\mapsto (b, -\pr(b))$. Therefore, we have
$\widetilde{f}(\delta'(\fp))=0$.

Identify  $(\mathfrak{g}\times \fh)^*$ with $\mathfrak{g}\times \fh$ via the non-degenerate pairing $(g_1, g_2):=\tr(g_1 g_2)$.
Let $\delta: \fp\to \fg\times \fh$ be the linear function $b\mapsto (b, \pr(b))$. Then, $\delta' \fp^{\perp}$ is identified with $\delta\fp \cong \fp$.
Therefore, $\widetilde{f}\in (\mathfrak{g}\times \fh)^*$ corresponds to $(c, \pr(c)) \in \delta\fp$ for some $c\in \fp$, and its first component $f_1$ corresponds to $c$ under the identification $\fg^*\cong \fg$. The second component $f_2$ of $f$ satisfies
\begin{equation}\label{eqn:f_2}f_{2}([h, a])=\tr (\pr(c)\cdot h), \,\ \text{for any $h\in \fh$}.\end{equation}

For vector spaces $E, F$, the bilinear function
\[
\Hom(E, F)\times \Hom(F, E)\to k, \,\ (f, f^*)\mapsto \tr(f\circ f^*). 
\]
gives an isomorphism $\Hom(E, F)^* \cong \Hom(F, E)$. We identify the
second component $f_2$ with 
some element $b \in \Rep(Q^{\heartsuit,op}, v_1, w) \times  \Rep(Q^{op}, v_2)$. 
The equality \eqref{eqn:f_2}
yields $[a, b]=\pr(c)$. This proves (a).
\end{proof}
We use the following abbreviations: 
$T^*Y^s=\Rep(\overline{Q^\heartsuit},v_1, w)^{ss}\times \Rep(\overline{Q},v_2)\subseteq T^*Y$, and 
$T^*X'^{s}=\Rep(\overline{Q^\heartsuit},v, w)^{ss}\subseteq T^*X'$.
There is a bundle projection $T^*X\to G\times_PT^*Y$. We define $T^*X^s$ to be the preimage of $G\times_PT^*Y^s$ under this bundle projection.
In particular, we have 
\[T^*_{G}X^{s}:=G\times_{P}(T^*_LY\cap T^*Y^s)=G\times_{P}(\mu_{v_1, w}^{-1}(0)^{ss}\times \mu^{-1}_{v_2}(0)),\] for $L=G_{v_1}\times G_{v_2}$.
We define $Z^s:=\psi^{-1}(T^*X'^{s})$ and $Z^s_G:= Z^{s}\cap Z_{G}$. 

\begin{lemma}\label{lem:stable_phi}
We have $\phi(Z^s)\subset T^*X^{s}$. 
\end{lemma}
\begin{proof}
This follows from the description of stability condition $\theta^+$ in \S~\ref{subsec:moment_quiver}.
Indeed, assuming there is an element $(g,x,x^*,i,j)\in Z$ so that $\phi((g,x,x^*,i,j))\notin T^*X^{s}$, then by definition of $T^*X^{s}$,  in $V_1$ there is a sub-vector space contained in $\ker(j)$ which is fixed by $(\pr(x), \pr(x^*))$. This sub-vector space is in turn a sub-vector space of $V$ via the inclusion $V_1\inj V$, contained in $\ker(j)$, which is still fixed by $x,x^*$. (This is because that $(g,x,x^*,i,j)\in Z$ implies that $V_1$ is a sub-representation and hence for any vector $a\in V_1$ the action of $(x,x^*)$ coincide with that of  $(\pr(x), \pr(x^*))$). This implies that $(g,x,x^*,i,j)$ is not in $Z^s$. Therefore $\phi(Z^s)\subset T^*X^{s}$.
\end{proof}

Thus, we have the following diagram of correspondences of $G_v\times T\times G_w$-varieties.
\begin{equation} \label{corr for action}
\xymatrix@R=1.5em{
T^*_G X^s\ar@{^{(}->}[d]
&&Z_G^s\ar@{^{(}->}[d]\ar[ll]_{\overline\phi}\ar[r]^(0.4){\overline\psi}&
T^*_GX'\cap T^*X'^s\ar@{^{(}->}[d]\\
G\times_PT^*Y^s\ar@{^{(}->}[r]^{\iota} & T^*X^s &
Z^s \ar[l]_(.4)\phi \ar[r]^(.4)\psi & T^*X'^s.
}\end{equation}
\begin{lemma}
The left square of diagram \eqref{corr for action} is a pullback diagram. 
\end{lemma}
\begin{proof}
By Lemma \ref{lem:descrip of Z_G}, we have:
\begin{align*}
&Z_{G}^s=G\times_P\{ (x, x^*, i, j )\in \mu^{-1}_{v, w}(0)^{ss} \mid (x, x^*)(V_1)\subset V_1,   \text{Im}(i)\subset V_1\}.\\
&T^*_{G}X^{s}=G\times_{P}(\mu_{v_1, w}^{-1}(0)^{ss}\times \mu^{-1}_{v_2}(0)).\end{align*}
To prove this lemma, we need to compute the fiber product $T^*_GX^s\times_{T^*X^s} Z^s$ and identify it with $Z_{G}^s$. Indeed,
\begin{align*}
&T^*_GX^s\times_{T^*X^s} Z^s\\
=&G\times_P
\{
(x_1, x_1^*, i_1, j_1, x_2, x_2^*), (x, x^*, i, j) \in T^*_GX^s\times Z^s \mid  \pr(x)=(x_1, x_2), \pr(x^*)=(x_1^*, x_2^*), \\ & \,\  i^{V_1}=i_1,  j_{V_1}=j_1,  [x, x^*]+i\circ j=([x_1, x_1^*]+i_1\circ j_1)\times [x_2, x_2^*]=0
\}\\
=&G\times_P\{ (x, x^*, i, j )\in \mu^{-1}_{v, w}(0)^{ss} \mid (x, x^*)(V_1)\subset V_1,   \text{Im}(i)\subset V_1\}\\
=& Z_{G}^s.
\end{align*}
This completes the proof.
\end{proof}

Let $v$, $w\in \bbN^I$ be the dimension vectors. %We denote by $\fM(v, w)=\fM_{\theta^+}(v, w)$ the Nakajima quiver variety with stability condition $\theta^+$. The group $G_w \times T:=\prod_{i\in I} GL_{w^i}\times T$ acts naturally on $\fM(v, w)$. Let $\calM_{v,w}:=A_{G_w\times T}(\fM(v, w))$ be the equivariant Borel-Moore homology of $\fM(v, w)$. 
As the action of $G_v$ on $\mu_{v, w}^{-1}(0)^{ss}:=\mu_{v, w}^{-1}(0)\cap \Rep(\overline{Q}, v, w)^{ss}$ is free, we have 
\[\calM(v,w):=A_{T\times G_w}(\fM(v, w))\cong A_{G_v\times T\times G_w}(\mu_{v, w}^{-1}(0)^{ss}).\]

For each  $w\in \bbN^I$, and each pair $v_1,v_2\in\bbN^I$, we define maps
\[
a_{v_1,v_2}:  \calM(v_1,w) \otimes \mathcal{P}_{v_2} \to
\calM(v_1+v_2,w)
\]
as follows.

We start with the K\"unneth morphism.
\begin{align}\label{equ: action iso}
\calM(v_1,w) \times \mathcal{P}_{v_2}
= &A_{G_{v_1}\times G_w \times T}(\mu_{v_1, w}^{-1}(0)^{ss}) 
\otimes A_{G_{v_2}\times T}(\mu_{v_2}^{-1}(0)) \notag \\
\to &A_{G_{v_1}\times G_{v_2}\times T\times G_w}(\mu_{v_1, w}^{-1}(0)^{ss}\times \mu^{-1}_{v_2}(0)) \notag\\
\cong&A_{G\times T\times G_w}\left(G\times_P \big(\mu_{v_1, w}^{-1}(0)^{ss}\times \mu^{-1}_{v_2}(0)\big)\right) .
\end{align}
We define the map $a_{v_1,v_2}$ to be the composition of the morphism in \eqref{equ: action iso} with the following morphism
\[
\overline\psi_*\circ\phi^\sharp:A_{G_{v}\times T\times G_w}\left(
T_G^*X^s\right)\to A_{G_{v}\times T\times G_w}(
T^*_GX'\cap T^*X'^s)=\calM(v, w),
\]
where the pullback $\phi^\sharp$ is the Gysin pullback of $\phi$ in diagram \eqref{corr for action}. Note that $\overline{\psi}$ is a proper map, since it is a restriction of a proper map. In particular, $\overline{\psi}_*$ is well-defined. Similarly, $\phi:Z^s\to T^*X^s$ is a local complete intersection morphism (the condition for being so is local, hence if $\phi: Z\to T^*X$ is a local complete intersection morphism, then so is $\phi:Z^s\to T^*X^s$). By \ref{subsec:pushforward}(4), $\phi^\sharp$, being the Gysin pull-back via  a complete intersection morphism is well-defined. 

\begin{theorem}\label{thm:prep action}
For each $w\in \bbN^I$, the maps 
\[
a_{v_1,v_2}:\calM(v_1,w)\otimes \calP_{v_2} \to\calM(v_1+v_2,w)\] fit together to define an action of $^A\calP$ on 
$\calM(w):=\bigoplus_{v}\calM(v,w)$.
In other words, $a_{v_1, v_2}$ induces an 
$R[\![t_1,t_2]\!]$-algebra homomorphism
\[\Phi: \null^A\calP\to \End\Big(\bigoplus_{v\in \N^I} A_{T\times G_w}(\fM(v,w))\Big).\]
\end{theorem}
\begin{proof}
The proof follows from the same idea as the proof of Theorem~\ref{prop:assoc_hall}.
More precisely, we fix a flag $V_1\subset V_2\subset V$, with $\dim V_i=v_1+\cdots+v_i$. Fix an $I$-tuple of vector spaces $W$ with dimension vector $w$.
We define the following varieties:
\begin{itemize}
\item Let $X_1$ be the set of quadruples $(F_1, F_2, a, j)$, where $F_1\subset F_2\subset V$ is a flag such that $F_1\cong V_1, F_2\cong V_2$, and $a\in \Rep(Q, v)$ is an endomorphism of the vector space 
$F_1\oplus (F_2/F_1)\oplus (V/F_2)$.
$j$ is an element of $\Hom(F_1, W)$.
\item Let $X_2$ be the set of triples $(F_1, a, j)$, where $F_1\subset V$, such that $F_1\cong V_1$ and $a \in \Rep(Q, v)$ is an endomorphism of the vector space $F_1\oplus (V/F_1)$. $j$ is an element of $\Hom(F_1, W)$.
\item $X_3=\Rep(Q, v)\oplus \Hom(V, W)$.
\end{itemize}
We then define the following correspondences. For $i=1, 2, 3$, let $W_i$ be the following. 
\begin{align*}
W_1=\{(F_1&, a, j)\mid F_1\subset V, \hbox{such that $F_1\cong V_1$, and } a(F_1) \subset F_1, \hbox{for $a\in \Rep(Q, v), j\in \Hom(V, W)$}\}.\\
W_2=\{(F_1&, F_2, a, j)\mid F_1\subset F_2 \subset V, a \in \Rep(Q, v), \hbox{such that $F_i\cong V_i$, and }
a(F_i) \subset F_i,  \\
&\hbox{for $i=1, 2$, and $j\in \Hom(V, W)$}\}.
\\
W_3=\{(F_1&, F_2, a, v)\mid F_1\subset F_2 \subset V, a\in \Rep(Q, v), \hbox{such that $F_i\cong V_i$, for $i=1, 2$, and }\\
& a\in \End(F_1\oplus (V/F_1)), \hbox{$a$ preserves the subspace $\{0\}\oplus (F_2/F_1)$, and $j\in \Hom(F_1, W)$}\}.
\end{align*}
We have the inclusions $W_1\subset X_3\times X_2$, $W_2\subset X_3\times X_1$, and $W_3\subset X_2\times X_1$.
It is clear that those inclusions give an isomorphism $W_2=W_1\times_{X_2} W_3$. 
We consider
\[
Z_1=T^*_{W_1}(X_3\times X_2), \,\
Z_2=T^*_{W_2}(X_3\times X_1), \,\
Z_3=T^*_{W_3}(X_2\times X_1).
\]
The intersection $(W_1\times X_1)\cap (X_3\times W_2)$ is transversal in $X_3\times X_2\times X_1$. Thus, by \cite[Theorem 2.7.26]{CG} we have an isomorphism $Z_1\times_{T^*X_2} Z_{3} \cong Z_2$. As in Theorem~\ref{prop:assoc_hall}, we have $\dim(Z_1)+\dim(Z_3)=\dim(Z_2)+\dim(T^* X_2)$ by Lemma \ref{lem:dim_lag}. 

Let $P_1:=\{g\in G=G_{v_1+v_2+v_3} \mid g(V_1)\subset V_1\}$, with Lie algebra $\fp_1$,
and $P:=\{g \in G \mid g(V_i)\subset V_i, i=1, 2\}$ with Lie algebra $\fp$.
By Lemma \ref{lem:descrip of Z_G}, we have
\begin{align*}
&T^*X_{2}\subset G\times_{P_1}\big(\fp_1\times \Rep(\overline {Q^{\heartsuit}}, v_1, w)\times \Rep(\overline{Q}, v_2+v_3)\big),
\\
&T^*X_{1}\subset G\times_{P}\big(\fp\times \Rep(\overline {Q^{\heartsuit}}, v_1, w)\times \Rep(\overline Q, v_2)\times\Rep(\overline{Q}, v_3)\big).
\end{align*}
Define
\begin{align*}
T^*X_{3}^s&:=\Rep(\overline{Q^{\heartsuit}}, v_1+v_2+v_3, w)^{ss},\\
T^*X_{2}^s&:=T^*X_{2}\cap G\times_{P_1}\big(\fp_1\times \Rep(\overline {Q^{\heartsuit}}, v_1, w)^{ss}\times \Rep(\overline{Q}, v_2+v_3)\big),\\
T^*X_{1}^s&:=T^*X_{1}\cap G\times_{P}\big(\fp\times \Rep(\overline {Q^{\heartsuit}}, v_1, w)^{ss}\times \Rep(\overline Q, v_2)\times\Rep(\overline{Q}, v_3)\big).
\end{align*}
Define\[
Z_{3}^{s}:=\psi_{3}^{-1}(T^* X_2^s), \,\
Z_{2}^{s}:=\psi_{2}^{-1}(T^* X_3^s), \,\
Z_{1}^{s}:=\psi_{1}^{-1}(T^* X_3^s).
\]
Then we have the following diagram with the square being Cartesian.
\[\xymatrix@R=1.5em{
T^*X_3^s&Z_{1}^s\ar[l]_{\psi_1}\ar[r]^{\phi_1}&T^*X_2^s,\\
&Z_{2}^s\ar[u]\ar[lu]^{\psi_2}\ar[dr]_{\phi_2}\ar[r]&Z_{3}^s\ar[d]^{\phi_3}\ar[u]_{\psi_3},\\
&&T^*X_1^s.
}\]
We define the maps $I_1, I_2, I_3$ as in Theorem~\ref{prop:assoc_hall}. The same argument shows $I_2=I_1\circ I_3$. This implies $a_{v_1, v_2}$ is an action map.
\end{proof}
By convention, we consider the right action of $\calP$ on $\calM(w)$. Therefore, the multiplication of $\End(\calM(w))$ is understood as composition of right action operators. 
\begin{remark}
If one uses the stability condition $\theta^{-}=(-1,\cdots,-1)$ in the definition of Nakajima quiver variety, the Lagrangian correspondence  $Z_{G}^{(\chi_{\theta^-})-ss}$ should be adjusted to
\[
Z_{G}^{(\chi_{\theta^-})-ss}=
G\times_P\{ (x, x^*, i, j )\in \mu^{-1}_{v, w}(0)^{(\chi_{\theta^-})-ss} \mid (x, x^*)(V_1)\subset V_1,   \ker(j)\supset V_1\}.
\]
The Lagrangian correspondence formalism 
will give us a left action of  $\calP^{op}$ on $A_{G_w\times T}(\mathfrak{M}_{\theta^-}( w))$.
This left module of $\calP^{op}$ coincides with the natural left action of $\calP^{op}$ on $\calM( w)$, under 
the identification of $\mathfrak{M}_{\theta^+}(v, w)$ and 
$\mathfrak{M}_{\theta^-}(v, w)$, sending any representation $V$ to its dual $V^\vee$.
\end{remark}

\subsection{Nakajima's raising operators}
In this section, we take $Q$ to be an arbitrary quiver. We interpret the action of the CoHA $^A\calP(Q)$ 
constructed in \S\ref{subsec:PreproRepn} in terms of  Nakajima's raising operators. This interpretation allows us to compare $^A\calP(Q)$ with the quantum groups.

We start by recalling the raising operators constructed by Nakajima in \cite{Nak98, Nak99}. Recall in \S\ref{subsec:moment_quiver}, we denote by $\fM(v,w)$ the 
Nakajima quiver variety with the fixed stability condition $\theta^+$. Let $\fM_{0}(v, w)$ be the affine quotient of $\mu_{v, w}^{-1}(0)$. That is, 
\[
\fM_{0}(v, w):=\Spec (k[\mu_{v, w}^{-1}(0)]^{G_v}),
\]
where $k[\mu_{v, w}^{-1}(0)]$ is the coordinate ring of $\mu_{v, w}^{-1}(0)$.
We have the resolution of singularities  $\pi: \fM(v,w)\to \fM_{0}(v, w)$.
For two dimension vectors $v_1$ and $v_2$,
the composition $\fM(v_i,w)\to \fM_{0}(v_i, w) \subset \fM_{0}(v_1+v_2, w)$ are denoted by $\pi_i$.
Let 
\[
Z(v_1, v_2, w):=\{ (x_1, x_2)\in \fM(v_1,w)\times \fM(v_2,w) \mid \pi_1(x_1)=\pi_2(x_2)\}
\] be the Steinberg variety. 
By the construction of $\fM(v, w)$, we have the tautological vector bundle 
\[
\mu_v^{-1}(0)^{ss}\times _{G_v} V \to \fM(v, w)
\]
associated to the principal $G_v$-bundle $\mu_v^{-1}(0)^{ss} \to \fM(v, w)$. Here 
$V$ is the $G_v$ representation with dimension vector $v$. We denote the vector bundle by $\calV(v, w)$.

We now consider the case when $v_1=v_2-e_k$, where $e_k$ is the dimension vector whose entry $k$ is $1$, and other entries are $0$. The Hecke correspondence $C^+_k(v_2,w)$ (see \cite{Nak98, Nak99}) is an irreducible component of  $Z(v_1,v_2,w)$, defined as the set of quintuples $\{ (x,x^*, i, j, S)\}$  up to $G_v$-conjugation, where $(x,x^*, i, j)\in \mu_{v, w}^{-1}(0)^{ss}$ and $S\subset V$ is a $x,x^*$-invariant subspace containing the image of $i$ with $\dim(S)=v_2-e_k$. We consider $C^+_k(v_2,w)$ as a closed subvariety of 
$\fM(v_2-e_k, w)\times \fM(v_2, w)$ by setting
\[
(x^1, x^{*1}, i^1, j^1):=\text{the restriction of $(x, x^*, i, j)$ to $S$},\,\ 
(x^2, x^{*2}, i^2, j^2):=(x, x^*, i, j).
\]
The component $C^+_k(v_2,w)$ is smooth, and it is a Lagrangian subvariety of $\fM(v_2-e_k, w)\times \fM(v_2, w)$ as shown by Nakajima. 
%The dimension of the quiver variety is
%\[
%\dim \fM(v, w)=2w\cdot v-C_{Q} v\cdot v,
%\] where $C_{Q}$ is the Cartan matrix of $Q$.
In particular, 
\begin{align*}
\dim C^+_k(v_2,w)=&\frac{\dim \fM(v_2-e_k, w)+\dim \fM(v_2, w)}{2}.
%=&\frac{2w\cdot (v_2-e_k)-C_{Q} (v_2-e_k)\cdot (v_2-e_k)+2w\cdot v_2-C_{Q} v_2\cdot v_2}{2}
\end{align*}
The tautological line bundle $\calL_k$ of $C^+_k(v_2,w)$ is defined to be the quotient 
\[\calL_k:=\calV(v_2,w)/\calV(v_1,w).\]
 
Nakajima defined the following raising operators. 
Let $f(t)\in A_T(\pt)[\![t]\!]$ be a power series. 
Then $f(c_1(\calL_k))$ is a well-defined element in $A_{G_w\times T}(C^+_k(v_2,w))$.  
We have the following diagram:
\[
\xymatrix@R=1.5em @C=1em{
C^+_k(v_2,w) \ar@{^{(}->}[r]& \fM(v_2-e_k, w)\times \fM(v_2, w) \ar[dl]\ar[dr] &\\
\fM(v_2-e_k, w)&&\fM(v_2, w).
}
\]
Denote by $p_i: C^+_k(v_2,w)\to \fM(v_i, w)$ the composition of the inclusion with the $i$--th projections, for $i=1, 2$, and $v_1=v_2-e_k$.
Let $\Psi(f(c_1(\calL_k)))\in \End_{R[\![t_1,t_2]\!]}(A_{G_w\times T}\left(\fM(w))\right)$ be the raising operation given by convolution with $f(c_1(\calL_k))$. In other words, let $\alpha\in A_{G_w\times T}(\fM(v_1, w))$, 
\[
\Psi(f(c_1(\calL_k)))(\alpha):=p_{2*}\big(p_1^*(\alpha)\cap f(c_1(\calL_k))\big).
\]

For the dimension vector $e_k$, the condition $[x, x^*]=0$ in the moment map \[
\mu_{e_k}: \Rep(Q, e_k)\times \Rep(Q^{\op}, e_k) \to k, \,\ [x,x^*]=0
\] is automatic. Therefore, $\mu_{e_k}^{-1}(0)=T^*\Rep(Q, e_k)$. In particular, $\mu_{e_k}^{-1}(0)$ is a vector space with  $G_{e_k}=\Gm$-action, and we have the isomorphism
\[
\calP_{e_k}:=A_{G_v\times T}(\mu_{e_k}^{-1}(0))
\cong A_{\Gm\times T}(\pt)\cong A_{T}(\pt)[\![ z^{(k)}]\!]. \]

Let $\xi_k$ be the natural one dimensional representation of $G_{e_k}=\Gm$. Then,
$z^{(k)}$ can be viewed as $c_1(\xi_k) \in A_{G_v\times T}(\mu_{e_k}^{-1}(0))$.

In the case of $v_1+e_k=v_2$, we write $v=v_2$ for short, for the Lagrangian correspondence, we have
$Y=\Rep(Q, v-e_k, w) \times \Rep(Q, e_k),  \,\ X'=\Rep(Q, v, w).$ Therefore, the correspondence in this case becomes
\begin{align}
\xymatrix{
G_v\times_P(\mu_{v-e_k, w}^{-1}(0)^{ss}\times \mu_{e_k}^{-1}(0))  & Z^s_{G} \ar[l]_(0.2){\overline{\phi}}\ar[r]^(0.4){\overline{\psi}} & \mu_{v, w}^{-1}(0)^{ss}}. \label{eq:lag for e_k}
\end{align}
\Omit{
\begin{align}
&Y=\Rep(Q, v-e_k, w) \times \pt,  \,\ X'=\Rep(Q, v, w). \notag\\
&\calV:=\{(x, j)\in \Rep(Q, v, w) \mid x(V_1) \subset V_1\} \subset X'. \notag \\
&X:=G\times_{P} Y=G\times_{P} \Rep(Q, v-e_k, w) , \,\ W=G\times_{P} \calV. \notag\\
&\xymatrix{
T^*(G\times_P\Rep(Q, v-e_k, w)) & Z \ar[l]\ar[r] & T^*\Rep(Q, v, w)
} \notag\\
&\xymatrix{
G_v\times_P\mu_{v-e_k, w}^{-1}(0)^{ss} & Z^s_{G} \ar[l]_(0.3){\overline{\phi}}\ar[r]^(0.4){\overline{\psi}} & \mu_{v, w}^{-1}(0)^{ss}}. \label{eq:lag for e_k}
\end{align}}
\begin{theorem}
\label{thm: hecke corr}
For any $f(t)\in  A_T(\pt)[\![t]\!]$, view $f(z^{(k)})\in \calP_{e_k} \cong A_{T}(\pt)[\![ z^{(k)}]\!]$, 
we have the equality 
\[
\Psi(f(c_1(\calL_k)))=\Phi(f(z^{(k)}))\] in $\End_{R[\![t_1,t_2]\!]}(A_{G_w\times T}\left(\fM(w))\right)$, where $\Phi$ is the action of the preprojective CoHA $^A\calP$.
\end{theorem}
\begin{proof}
Taking the quotient by $G_v$ of the Lagrangian correspondence \eqref{eq:lag for e_k}, we get the following commutative diagram.
\[
\xymatrix@R=1.5em{
T^*X^{s} & Z^s \ar[l]_{\phi}\ar[r]^{\psi} & T^*X'^{s} \\
% line 1---
G_v\times_P(\mu_{v-e_k}^{-1}(0)^{ss}\times \mu_{e_k}^{-1}(0)) \ar@{^{(}->}[u]_{g}& Z^s_{G} \ar@{^{(}->}[u]_{g'}\ar[l]_(0.2){\overline{\phi}}\ar[r]^{\overline{\psi}} \ar[d]& \mu_{v}^{-1}(0)^{ss} \ar@{^{(}->}[u]\ar[d]\\
%line 2---
\mathfrak{M}(v-e_k, w) & C_k^+( v, w) \ar[l]_{p_1} \ar[r]^{p_2}& \mathfrak{M}(v, w)} \]
The vertical maps $g, g'$ are closed embeddings. 
In the above diagram, the map $\overline{\phi}$ is a smooth morphism of smooth varieties. The usual pullback $\overline{\phi}^*$ is well-defined. 
We first show the Gysin pullback $\phi^\sharp$ is the same as the usual pullback $\overline{\phi}^*$.  By Lemma \ref{lem:descrip of Z_G}, the variety $Z_{G}^s$ is a principal $G_{v}$--bundle of $C_k^+( v, w)$. 
 Thus, 
\begin{align*}
& \dim Z_G^s=
 \dim G_{v}+\dim  C_k^+( v, w)=\dim G_v+\frac{\dim \fM(v-e_k, w)+\dim \fM(v, w)}{2}\\
 %-----------
 &\phantom{1234567}=\dim G_v+\frac{2\dim \Rep(Q,v-e_k, w)-2\dim G_{v-e_k}+2\dim \Rep(Q, v, w)-2\dim G_v}{2}\\
  &\phantom{1234567}=\dim \Rep(Q,v-e_k, w)-\dim G_{v-e_k}+\dim \Rep(Q, v, w)\\
  &\phantom{1234567}=w\cdot (v-e_k)+A_Q(v-e_k)\cdot (v-e_k)-(v-e_k)\cdot (v-e_k)+w\cdot v+A_Q v\cdot v;\\
 %----
&\dim T^*X
=2\dim (G\times_{P} Y)=
2(\dim (G_v/P)+\dim\Rep(Q, v-e_k, w)+\dim \Rep(Q, e_k))\\
%----
 &\phantom{1234567}=2(\dim (G_v/P)+\dim\Rep(Q, v-e_k, w))+2\dim \Rep(Q, e_k);
 \end{align*}
 \begin{align*}
&\dim (G\times_P(\mu_{v-e_k, w}^{-1}(0)^{ ss}\times \mu_{e_k}^{-1}(0))
= \dim (G_v /P) +\dim \mu_{v-e_k, w}^{-1}(0)^{ ss}+\dim \mu_{e_k}^{-1}(0)\\
 &\phantom{1234567}=  \dim (G_v/P) +2\dim \Rep(Q, v-e_k, w)-\dim G_{v-e_k}+2\dim \Rep(Q, e_k);\\
%----
&\dim Z
=\dim G_v/P+2(\Rep(Q, v-e_k)+ A_Q e_k\cdot v)+ w\cdot (v-e_k)+w\cdot v\\
 &\phantom{1234567}=\dim G_v/P+2(A_Q(v-e_k)\cdot (v-e_k)+A_Q e_k \cdot v)+ w\cdot (v-e_k)+w\cdot v.
\end{align*}
Therefore,
\begin{align*}
\dim Z-\dim Z_{G}
=\dim (G_v/P)+\dim G_{v-e_k}=\dim T^*X-\dim \big(G\times_P(\mu_{v-e_k, w}^{-1}(0)^{ ss}\times \mu_{e_k}^{-1}(0))\big).
\end{align*}
Hence we have
\[
\dim(T^*X^s)+\dim (Z_G^s)=\dim\big(G\times_P(\mu_{v-e_k, w}^{-1}(0)^{ ss}\times \mu_{e_k}^{-1}(0))\big)+\dim Z^s.
\]
Thus, Lemma \ref{Lem:Gysin}(\ref{Lem:gysin_two}) yields $ \phi^*\circ g_*=g'_*\circ \overline{\phi}^*$. Therefore, for $\alpha\in A_{G_w\times T}(\mathfrak{M}(v-e_k, w))$, 
\[\Phi((z^{(k)})^l)(\alpha)=\overline{\psi}_*\overline{\phi}^*((z^{(k)})^l\otimes \alpha),\] here $\overline{\phi}^*$ is the usual pullback.
Here to distinguish the vertex $k\in I$ and the power $l\in \bbN$, we write $(k)$ for the label $k\in I$.

The isomorphism $Z_G^s/G_v \cong C_k^+( v, w)$ follows from Lemma \ref{lem:descrip of Z_G}.
It induces an isomorphism \[
A_{G\times T\times G_w}(Z^s\cap Z_G)\cong 
A_{T\times G_w}(C_k^+( v, w)).
\] 
The isomorphism maps $\overline{\phi}^*( (z^{(k)})^l \otimes \alpha)$ to $(c_1(\calL_k))^{l}\otimes p_1^*(\alpha)$, for any $l$. The pullback of the line bundle $\calL_k$ on $Z_{G}^s$ is the trivial bundle with fiber $V(v, w)/V(v-e_k, w)$. It carries a natural $G_{e_k}=\Gm$ action. The element $z^{(k)}$ can be interpreted as $z^{(k)}=c_{1}(V(v, w)/V(v-e_k, w)) \in A_{G_v\times T}(\mu_{e_k}^{-1}(0))$. Thus, $\overline{\phi}^*( z^{(k)})\mapsto c_1(\calL_k)$ under the isomorphism. 

The claim follows now from the definitions of the two actions $\Psi$ and $\Phi$.
\end{proof}

\subsection{Extended CoHA}
\label{subsec:Cartan}
For later use, we will consider some modifications of $\calP$.

\begin{assumption}\label{Assump:Weights}
Assume the action of $T=\Gm^2$ on $T^*\Rep(Q,v)$ satisfies one of the following:
\item[Case 1:] $m(h)=m(h^*)=1$, for any $h\in H$;
\item[Case 2:] The $T$-action is the one defined in Remark~\ref{rmk:weights}. 
\end{assumption}
Note that Assumption~\ref{Assu:WeghtsGeneral} is satisfied in either case above.

Let $(R,F)$ be the formal group law of $A$. Define 
$^A \calP^0:=\Sym_{R}(\oplus_{i\in I}A_{G_{e_i}}(\pt))$ be the symmetric algebra of $\oplus_{i\in I}A_{G_{e_i}}(\pt)$.
Let 
\[
\Phi_{k}(z)=\sum_{r \geq 0} (u^{(k)})^r z^{-r}\in\null ^A \calP^0[\![ z^{-1}]\!]\]
 be the generating series of generators $(u^{(k)})^r \in\null  ^A \calP^0$. 
 
Recall that $a_{ik}$ is the number of arrows in $H$ from $i$ to $k$. 
Let $c_{ik}=-a_{ik}-a_{ki}$ if $k\neq i$, $2$ if $k=i$. 
We define a $^A \calP^0$ action on  $^A\calP$ as follows.
For any $g \in ^A \calP_{v}$, 
 \[
 \Phi_{k}(z) g  \Phi_{k}(z)^{-1}:=g \widehat{\Phi_k}(z, v),
 \]
 where 
  \[
\widehat{\Phi_k}(z, v):= \prod_{i\in I\backslash \{k\}} \prod_{j=1}^{v^{i}}\frac{(z-_F \lambda_{j}^{(i)}-_Ft_1)^{a_{ik}} (z-_F \lambda_{j}^{(i)}-_Ft_2)^{a_{ki}}}{(z-_F \lambda_{j}^{(i)}+_Ft_2)^{a_{ik}} (z-_F \lambda_{j}^{(i)}+_Ft_1)^{a_{ki}}}\cdot 
 \prod_{j=1}^{v^{k}}\frac{(z-_F \lambda_{j}^{(k)}+_Ft_1+_F t_2) }{(z-_F \lambda_{j}^{(k)}-_Ft_1-_F t_2)}
 \]
when the  $T$-action satisfies Assumption~\ref{Assump:Weights}(1);
 and  
 \[
 \widehat{\Phi_k}(z, v):= \prod_{i\in I} \prod_{j=1}^{v^{i}}
\frac{(z-_F \lambda_{j}^{(i)}+_F(c_{ki})_F\frac{\hbar}{2})}{(z-_F \lambda_{j}^{(i)}-_F(c_{ki})_{F} \frac{\hbar}{2})}
 \]
 when the  $T$-action satisfies Assumption~\ref{Assump:Weights}(2).
Note that the element $\widehat{\Phi_k}(z, v)$ lies in $A_{\GL_v\times T}(\pt)[\![z^{-1}]\!]$ in both cases.  \begin{lemma}
\begin{enumerate}
\item The action of $^A \calP^0$  on $^A\calP$ is well-defined.
\item Furthermore, for any $k\in I$,  $\Phi_k$ acts on $\calP$ by an algebra homomorphism.
\item The subalgebra $\calP^{\sph}\subseteq \calP$ is a $\calP^0$-submodule.
\end{enumerate}
\end{lemma} 
\begin{proof}
(1) follows from the fact that $\widehat{\Phi_i}(z, v)$ is symmetric. (2) follows from the projection formula and the equality $\widehat{\Phi_i}(z, v_1)\cdot \widehat{\Phi_i}(z, v_2)=\widehat{\Phi_i}(z, v_1+v_2)$.  
(3) is clear.
\end{proof} 

\begin{definition}
Define the extended preprojective CoHA associated to $A$ and $Q$ to be the algebra $^A\calP^{\ext}=\null^A\calP^{0}\ltimes  \null^A\calP$. The spherical subalgebra in the extended CoHA is $^A\calP^{\sph,\ext}=\null^A\calP^{0}\ltimes  \null^A\calP^{\sph}$.
\end{definition}

Similarly, there is an action of $\null ^A \calP^0$ on $\null ^A \calS\calH$ by algebra homomorphism.
We define the extension of shuffle algebra
$\null ^A \calS\calH^{\ext}:=\null ^A \calP^0\ltimes \null ^A \calS\calH$.
As a direct consequence of Proposition~\ref{thm:prepro to shuffle}, we have the following.
\begin{prop}
There is an algebra homomorphism  $^{A}\calP^{\ext}\to \calS\calH^{\ext}$, which becomes an isomorphism after localization. 
\end{prop}

\subsection{Action on quiver varieties}
\label{subsec:f(vw)}We are still under the condition that $Q$ has no edge-loops. 
For any $G$-variety $X$ and a $G$-equivariant rank-$n$ vector bundle $V$, let $\lambda_{-1/z}(V)\in A_G(X)[\![z]\!]$ be the equivariant Chern polynomial $\prod_{i=1}^n(z-_Fx_i)$ where $\{x_1,\cdots,x_n\}$ are the Chern roots of $V$. Note that for any element $\xi$ in $K_G(X)$, the Grothendieck ring of $G$-equivariant vector bundles on $X$, $\lambda_{-1/z}(\xi)$ is well-defined.

Recall that $
\fM(v, w):=\mu_{v, w}^{-1}(0)^{ss} /G_v$ is the quiver variety. Let 
$\calV_k:=\mu_{v, w}^{-1}(0)^{ss} \times_{G_v} V_k$ (resp. $\calW_k$) be the 
$k$-th tautological bundle (resp. the trivial $W_k$-bundle) on $\fM(v, w)$ respectively. Consider the following tautological element in $K^{G_w\times T}( \fM(v, w))$.
\begin{displaymath} 
\calF_{k}(v, w)= 
\left\{
     \begin{array}{lr}
     {\begin{matrix}
        q_1^{-1}\calW_k-(1+(q_1 q_2)^{-1}\calV_k)&+q_1^{-1}\sum_{\{h\in H\mid \inc(h)=k\}}V_{\out(h)}      \\
        &+q_2^{-1}\sum_{\{h\in H\mid \out(h)=k\}} V_{\inc(h)} \end{matrix}}
        & \hbox{ under Assumption~\ref{Assump:Weights}(1)}; \\
        
       q^{-1}\calW_k-(1+q^{-2}) \calV_k+q^{-1}\sum_{l: k\neq l}[-c_{kl}]_q\calV_{l} 
       &  \hbox{ under Assumption~\ref{Assump:Weights}(2)}.
     \end{array}
   \right.
   \end{displaymath}
   where $[n]_q=\frac{q^n-q^{-n}}{q-q^{-1}}$. 
This is a modification of \cite[(2.9.1)]{Nak99} according to $T$-action. When $t_1\neq t_2$, the $T$ action is given by $(t_1, t_2)(B_h, B_{h^*}, i, j)=(t_1B_h, t_2B_{h^*}, t_1i, t_2j ).$

The following lemma can be proved by straightforward calculation. In particular, (2) is proved in \cite[\S~10.1]{Nak99}.

\begin{lemma}\label{Lem:F_L}
\begin{enumerate}
\item  Under Assumption~\ref{Assump:Weights}(1), we have 
\begin{align*}
\calF_k(v_1+e_i, w)-q_1 q_2\calF_k(v_1+e_i, w)-
\calF_k(v_1, w)+q_1 q_2\calF_k(v_1, w)=\\
\left\{
     \begin{array}{lr}
      (q_1q_2-(q_1q_2)^{-1})\calL_k, & \hbox{ if } k=i, \\  (q_1^{-1}-q_2) a_{ik}\calL_i
  +(q_2^{-1}-q_1) a_{ki} \calL_i, & \hbox{ if } k\neq i 
  \end{array} \right.
\end{align*} 
\item   Under Assumption~\ref{Assump:Weights}(2), we have
\begin{align*}
\calF_k(v_1+e_i, w)-q^2\calF_k(v_1+e_i, w)-
\calF_k(v_1, w)+q^2\calF_k(v_1, w)=(q^{c_{ki}}-q^{-c_{ki}}) \calL_i.
\end{align*} 
\end{enumerate}
\end{lemma}

We define a $\calP^0$-action on $\calM(w)$ by  
\begin{equation*}
\Phi_{k}(z)\cdot m = \frac{\lambda_{-1/z}(\calF_{k}(v, w))}
{\lambda_{-1/z}(q_1q_2\calF_{k}(v, w))}m, 
\end{equation*}
for any $m\in\calM(v,w)$.

 \begin{prop}
 For any $w\in\bbN^I$, the action of $\calP$ and $\calP^0$ on $\calM(w)$ can be extended to the action of  $\calP^{\ext}=\calP^0\ltimes\calP$ on $\calM(w)$.
 \end{prop}
 \begin{proof}
 By the same argument as in \cite[\S 10.4]{Nak99}, it suffices to show 
 \[\Phi_k(z)g\Phi_k(z)^{-1}(m)=g \widehat{\Phi_k}(z, v)(m),\] 
 for any $m\in \calM(v_1, w)$, and the simple roots $v=e_i$, $i\in I$. In this case, $g=g(z^{(i)})\in A_{T}(\pt)[\![ z^{(i)}]\!]$. By Theorem \ref{thm: hecke corr}, the action of $g(z^{(i)})$ on $\calM(w)$ is given by convolution with $g(c_{1}(\calL_i))$, where $\calL_i$ is the tautological line bundle on Hecke correspondence $C^+_i(v_1+e_i,w)$. 
The action of $\Phi_k(z) z^{(i)} \Phi_k(z)^{-1}(m)$ is then given by 
 \begin{align*}
&\Big(\Delta_* \frac{\lambda_{-1/z}(\calF_{k}(v_1+e_i, w))}
{\lambda_{-1/z}(q_1q_2\calF_{k}(v_1+e_i, w))}\Big) \star c_{1}(\calL_i) \Big(\Delta_* \frac{\lambda_{-1/z}(\calF_{k}(v_1, w))}
{\lambda_{-1/z}(q_1q_2\calF_{k}(v_1, w))}\Big)^{-1}\\
=&\star c_{1}(\calL_i)\cdot \Big(\Delta_* \frac{\lambda_{-1/z}(\calF_{k}(v_1+e_i, w))}
{\lambda_{-1/z}(q_1q_2\calF_{k}(v_1+e_i, w))} \frac{\lambda_{-1/z}(q_1q_2\calF_{k}(v_1, w))}{\lambda_{-1/z}(\calF_{k}(v_1, w))}
\Big)
 \end{align*}
 where the equality is obtained by projection formula. 

Under Assumption~\ref{Assump:Weights}(1), by Lemma~\ref{Lem:F_L}(1),   the action of $\Phi_k(z) z^{(i)} \Phi_k(z)^{-1}(m)$ coincides with the action of $z^{(i)} \widehat{\Phi_k}(z, e_i)$,  where 
 \[
 \widehat{\Phi_k}(z, v)= \prod_{i\in I\backslash \{k\}} \prod_{j=1}^{v^{i}}\frac{(z-_F \lambda_{j}^{(i)}-_Ft_1)^{a_{ik}} (z-_F \lambda_{j}^{(i)}-_Ft_2)^{a_{ki}}}{(z-_F \lambda_{j}^{(i)}+_Ft_2)^{a_{ik}} (z-_F \lambda_{j}^{(i)}+_Ft_1)^{a_{ki}}}\cdot 
 \prod_{j=1}^{v^{k}}\frac{(z-_F \lambda_{j}^{(k)}+_Ft_1+_F t_2) }{(z-_F \lambda_{j}^{(k)}-_Ft_1-_F t_2)}.\]
Similarly, the claim holds in the case under Assumption~\ref{Assump:Weights}(2), by Lemma~\ref{Lem:F_L}(1).
 \end{proof}

\subsection{Twisting the preprojective CoHA}\label{subsec:twisting CoHA}
The spherical subalgebra of the preprojective CoHA per se is not directly related to affine quantum groups, albeit a modified version is. The modification involves a twisting the multiplication of the preprojective CoHA by a sign which involves the Euler-Ringel form of $Q$. The same sign twist is also present in many places when constructing quantum groups via quiver representations, e. g., \cite{Ring,Va00,Nak99}.
Here again we assume $Q$ has no edge-loops.

As in \cite{Nak99}, we define the adjacency matrices $AD$ and $\overline{AD}$ as
\begin{align*}
&(AD)_{kl}:=\#\{h\in H \mid \inc(h)=k, \out(h)=l\},\\
&(\overline{AD})_{kl}:=\#\{h\in H^{\op} \mid \inc(h)=k, \out(h)=l\}.
\end{align*}
Thus, $(AD)^{t}=\overline{AD}$. 
We define the matrices $C, \overline C$ as
\begin{equation}\label{equ:C Omega}
C:=I-AD, \overline C:=I-\overline{AD}.
\end{equation}

Let $m_{v_1, v_2}: \calP_{v_1}\otimes \calP_{v_2}\to \calP_{v_1+v_2}$ be the multiplication defined in \S~\ref{preproj CoHA}.
\begin{definition}
The twisted preprojective CoHA, denoted by $\widetilde{\calP}$,  is $\widetilde{\calP}:=\bigoplus_{v}\calP_{v}$ as $\bbN^I$-graded $R[\![ t_1, t_2 ]\!]$-module, endowed with the multiplication $\widetilde m_{v_1, v_2}$
\[
\widetilde m_{v_1, v_2}:=(-1)^{(v_2, \overline{C} v_1)+1}m_{v_1, v_2},
\] where $(\cdot, \cdot)$ is the standard inner product on $k^I$. 
\end{definition}
\begin{lemma}
The multiplication $\widetilde m_{v_1, v_2}$ is associative.
\end{lemma}
\begin{proof}
This follows from the associativity of $m^{\calP}$:
\begin{align*}
\widetilde m_{v_1+v_2, v_3}(\widetilde m_{v_1, v_2}(x_1, x_2), x_3)
%----
=&(-1)^{(v_3, \overline{C}(v_1+v_2))+1}(-1)^{(v_2, \overline{C} v_1)+1} m^{\calP}_{v_1+v_2, v_3}(m^{\calP}_{v_1, v_2}(x_1, x_2), x_3)\\
=&(-1)^{(v_2+v_3, \overline{C} v_1)+1}(-1)^{(v_3, \overline{C} v_2)+1}  m^{\calP}_{v_1, v_2+v_3}(x_1, m^{\calP}_{v_2, v_3}(x_2, x_3))\\
=&\widetilde  m_{v_1, v_2+v_3}(x_1, \widetilde m_{v_2, v_3}(x_2, x_3)).
\end{align*}
\end{proof}

Similarly, we define $\widetilde{\calS\calH}$ to be $\calS\calH$ as $\bbN^I$-graded $R$-module, with multiplication  given by
\[
\widetilde m_{v_1, v_2}:=(-1)^{(v_2, \overline C v_1)+1}m_{v_1, v_2},
\] where $m_{v_1, v_2}$ is the multiplication of $\calS\calH$. 

For any $w\in\bbN^I$
we define the map, for each $v_1,v_2\in\bbN^I$,
\[
\widetilde a_{v_1, v_2}:=(-1)^{(v_2, \overline{C} v_1)+1}a_{v_1, v_2}:
\calM(v_1,w) \otimes \widetilde{\calP}_{v_2}\to \calM(v_1+v_2,w).
\]

\begin{lemma}\label{lem:twistedCoHA}
Notations are as above. 
\begin{enumerate}
\item There is a well-defined algebra homomorphism $\widetilde\calP\to \widetilde{\calS\calH}$.
\item The maps $\widetilde a_{v_1, v_2}$ define an action of $\widetilde{\calP}$ on $\calM(w)$.
\end{enumerate}
\end{lemma}
As in the untwisted case, we write $\widetilde{a}_v:\calP_v\to \bigoplus_{v_1\in\bbN^I}\Hom(\calM(v_1,w),\calM(v_1+v,w))$.

Recall that $
\calP_{e_k}:=A_{G_{e_k}\times T}(\mu_{e_k}^{-1}(0))
=A_{\Gm\times T}(\pt)\cong A_{T}(\pt)[\![ z^{(k)}]\!]. $
By Theorem \ref{thm: hecke corr}, the action of $(z^{(k)})^l \in \calP_{e_k}\subseteq \widetilde{\calP}$ on the Nakajima quiver varieties $\calM(w):=\bigoplus_{v}\calM(v, w)$ is by
\begin{equation}\label{eq:action a_ek}
\widetilde{a}_{e_k}((z^{(k)})^l)\mapsto \sum_{v} (-1)^{(e_k, \overline{C} (v))}\big(c_1(\calL_k)\big)^{l}\star,
\end{equation} 
where $\calL_k$ is the tautological line bundle on the Hecke correspondence $C_k^+(v, w)$, and $\star$ is convolution action.

We define the {\it spherical subalgebra} of $\widetilde{\calP}$, denoted by $\widetilde{\calP}^{\sph}$, to be the subalgebra 
generated by $\calP_{e_k}$, for $k\in I$.  
\begin{remark}
The image of $\widetilde{\calP^{\sph, \ext}}$ in $\widetilde{\calS\calH^{\ext}}$ is a deformation of $U(\fb_Q[\![u]\!])\subseteq U(\fg_Q[\![u]\!])$ associated to the formal group law of the OCT $A$.
In the case when $A=\CH$, Theorem~\ref{thm:Yangian} below shows that this deformation is the same as the Borel of the Yangian. When $A=K$, according to Grojnowski's work in progress \cite{Gr2}, this deformation is expected to be the Borel of the quantum loop algebra.
 \end{remark}
 
\section{Examples related to Lusztig's conjecture} 
\label{subsec:examples}
To convince the readers that affine quantum groups associated to formal group laws other than additive and multiplicative ones are equally interesting, we provide examples, one of which is related Lusztig's reformulated conjecture of modular representations of Lie algebras.
Although in most of the present paper we work with representations of $Q$ over an arbitrary field (see \S~\ref{subsec:pushforward} and \S~\ref{subsec:2.1}), in this section we assume the base field $k$ of quiver representations has characteristic zero, due the reference to many properties of the algebraic cobordism theory of \cite{LM} that require the existence of resolution of singularities.
\begin{example}
When $A$ comes from the Eilenberg-MacLane spectrum $H\bbZ/p$, this construction yields Yangians over $\bbF_p$, the representation theory of which is  richer than the modular representations of Lie algebras. 
\end{example}

\begin{example}
When $A$ is the connective $K$-theory, this gives a geometric interpretation of Drinfeld's degeneration of quantum loop algebra to the Yangian (see also, \cite[Example~3.9]{YZ2}).
\end{example}

\begin{example}
 Recall that the formal group law associated to the algebraic cobordism  of Levine-Morel \cite{LM} is the universal  formal group law over the Lazard ring $\Laz$. There are some properties of $^{\Laz}\calP$ which have no direct analogue in the Yangian or quantum loop algebra cases.
 
Recall that $\Laz_\bbQ$ is dual to the ring of symmetric polynomials $\Lambda_\bbQ$, where the duality pairing is given by the integral of the Chern class on manifolds. Let $\bbH$ be the Heisenberg algebra, which is a quantum double of $\Lambda_\bbQ$. As vector spaces, we have the isomorphism $\bbH \cong \Laz_\bbQ\otimes\Lambda_\bbQ$. Furthermore, there is an algebra embedding $\Laz_\bbQ\inj \bbH$. The duality between $\Laz_\bbQ$ and $\Lambda_\bbQ$ induces a $\bbQ$-linear map $^{\Laz}\calP^{\sph}\otimes_\bbQ\Lambda_\bbQ\to Y_\hbar^+(\fg_Q)$.  However, this map does not respect algebra structures on any of the three algebras $^{\Laz}\calP^{\sph}, \Lambda_\bbQ$, or $Y_\hbar^+(\fg_Q)$. 
Nevertheless, there is a non-standard quantum double of $^{\Laz}\calP$, without direct analogue for the previously known affine quantum groups. Namely, base changing the $^{\Laz}\calP^{\sph}$ to $\bbH$ via the embedding $\Laz_\bbQ\inj \bbH$, and extending by scalars the multiplication formulas in \S~\ref{subsec:HallMulti} (while $\bbH$ commutes with the equivariant parameters), we get an algebra $^\bbH\calP^{\sph}$. This algebra  has $\bbH$ embedded, and is isomorphic as $\Laz_\bbQ$-modules to $^{\Laz}\calP^{\sph}\otimes_\bbQ \null^{\Laz}\calP^{\sph}$.
\end{example}

Another example  of interest is the case when $A$ is the Morava $K$-theory.
\footnote{The possibility that Lusztig's reformulated conjecture is related to our construction in Morava $K$-theory is kindly suggested to us by David Ben-Zvi and Ivan Mirkovi\'c.}
Assume $Q$ is of finite type. 
In \cite{Lusz}, in a  reformulation of his conjecture from 1979 on modular representation of algebraic groups, for each prime $p$, Lusztig wrote down a family of character formulas, parameterized by an integer $n$. For each  integral dominant weight $\lambda$, Lusztig's formula is denoted by $E^n_\lambda$. It has the following properties:
\begin{enumerate}
\item $E^0_\lambda$ coincides with the Weyl character formula;
\item $E^1_\lambda$ is  the character of irreducible representations of $U_q(\fg)$ at a $p$-th root of unity; 
\item for each fixed $\lambda$, $E^n_\lambda$  stabilizes for large enough $n$;  in this range $E^n_\lambda$ is conjectured to be the character of irreducible representations 
the simply-connected algebraic group $G$ with root system $Q$ over $\overline{\bbF_p}$.
\end{enumerate}
When $p$ is larger than a number depending only on $Q$, Lusztig proved property (3), based on earlier work of Anderson-Jantzen-Soergel \cite{AJS} and Steinberg \cite{Stein}. 
However, the formulas for general $n\neq 0$ which are not in the stabilizing range open the question, proposed in \cite{Lusz}: find a family of quantum groups, parameterized by a prime $p$ and an integer $n\in\bbN$, whose character of irreducible module of highest weight $\lambda$ is given by $E_\lambda^n$.

Recall that Morava $K$-theory, denoted by $K_pn$, is an oriented cohomology theory associated to each prime $p$ and non-negative integer $n$.
The constructions in \S~\ref{subsec:HallMulti}
 and \ref{subsec:PreproRepn} give an algebra $^{K_pn}\calP$ as well as its action on quiver varieties. 
It should be possible to obtain a purely algebraic description of this algebra along the lines of shuffle formula in \S~\ref{sec:formalShuf}, although at the present the formula stated in \S~\ref{sec:formalShuf} requires the coefficient ring of the formal group law to be a $\bbQ$-algebra.
Let $U_p(n)^+$ be the subalgebra generated by the constant loops. 
That is, the subalgebra of $^{K_pn}\calP$ generated by $1_{ke_i} \in K_pn_{\GL_{ke_i}\times\Gm}(\mu^{-1}_{ke_i}(0))$ for all $k\in \bbN$ and $i\in I$.
We expect the comultiplication on the shuffle algebra constructed in \cite{YZ2} has a geometric description, hence is well-defined on  $U_p(n)^+$, the Drinfeld double of which is denoted by $U_p(n):=D(\null^{K_pn}U_p(n)^+)$.

Note that at present the construction of finite quantum group of Lusztig \cite{L91} using perverse sheaves has no analogue for Morava $K$-theory, hence the only available construction is taking the finite part inside the quantum affine algebra.

Now we provide some evidence that this family of algebras is related to the family proposed by Lusztig.  
\begin{enumerate}
\item For $n=0$, the Morava $K$-theory is the Chow group with rational coefficients. Theorem~\ref{thm:Yangian} below shows that $^{K_pn}\calP\cong Y_\hbar^+(\fg)$. In \cite{YZ2} we further prove that $D(\null^{K_pn}\calP)\cong Y_\hbar(\fg)$. It is known that the subalgebra of $Y_\hbar(\fg)$ generated by constant loops is the universal enveloping algebra $U(\fg_\bbQ)$, the characters of absolutely irreducible modules of which are given by $E^0_\lambda$.
\item In Proposition~\ref{prop:stable}, we prove that the characters of the absolutely irreducible representations of $U_p(n)$ with highest weight $\lambda$  stabilize, if $p^n$ is larger than a constant depending on $\lambda$. 
\item In the limit case $n=\infty$, $K_pn$ is given by the Eilenberg-MacLane spectrum $H\bbZ/p$, the corresponding cohomology theory is the Chow group with $\bbZ$-coefficients.   
Similar construction as in Theorem~\ref{thm:Yangian} using Chow group with $\bbZ$-coefficients in lieu of $\bbQ$-coefficients gives a $\bbZ$-form of $U(\fg_\bbQ)$. An easy calculation shows that this $\bbZ$-form is the Kostant $\bbZ$-form, which by definition is generated by the divided powers of the Chevalley generators. That is, let $1_{e_i}\in \CH_{\GL_{e_i}\times\Gm}(\mu_{e_i}^{-1}(0),\bbZ)$ be denoted by $E_i$ for $i\in I$, then $E_i^{* k}=k!1_{ke_i}$ where $1_{ke_i}\in \CH_{\GL_{ke_i}\times\Gm}(\mu^{-1}_{ke_i}(0),\bbZ)$. In other words, $1_{ke_i}$ is the divided power operator $E^{(k)}_i$ in the Kostant $\bbZ$-form.
For large enough $p$, it is reasonable to expect that $U_p(\infty)$ is isomorphic to
the reduction mod-$p$ of the Kostant $\bbZ$-form (referred to as the hyperalgebra in literature), whose characters are known to be the same as those of the algebraic group \cite{Hum}, which in turn are given by $E_\lambda^\infty$ according to  \cite{Lusz}. 
\item For $n=1$, the Morava $K$-theory has an alternative description in terms of the usual $K$-theory. 
According to Grojnowski's work in progress \cite{Gr2}, $^K\calP$ is expected to be the positive part of the quantum loop algebra, the subalgebra of which generated by constant loops is a $\bbZ$-form of the finite quantum group $U_q(\fg_Q)$.  
Similar to (3) above, for suitable choice of 1-dimensional subgroup of $GL_v\times\Gm$, this integral form has $q$-divided powers. Specializing $q$ of this $\bbZ$-form at a $p$th root of unity, the irreducible modules would have the same character as given by the formula $E^1_\lambda$. However, 
to obtain the first Morava $K$-theory from the usual $K$-theory, instead of specializing  $q$ at a $p$th root of unity, one needs to take the invariant of the $K$-theory spectrum under the action of the group of $(p-1)$-th roots of unity $\mu_{p-1}$  via Adams operations. The precise relation between these two manipulations is unknown to us at present.
\end{enumerate}

We expect the irreducible representation of $U_p(n)$ over $\overline{\bbF}_p$ with highest weight $\lambda$ factors through the convolution algebra of fixed points $K_pn(Z(w)^{G})$, where $Z(w)$ is the Hecke correspondence in quiver variety for $w\in \bbN^I$ depending only on $\lambda$, and $G\subseteq \GL_w\times\Gm$. 
\begin{prop}\label{prop:stable}\footnote{The proof presented here was suggested to us by Marc Levine.}
Let $X$ be a smooth quasi-projective variety with $Z\subseteq X\times X$ a convolution subvariety. If $p^n-1>4\dim X+1$, then the convolution algebra $K_pn(Z)$ is isomorphic to the convolution algebra $\CH(Z,\bbZ/p)\otimes_{\bbZ/p}\bbZ/p[v^\pm]$.
\end{prop}
To prove this proposition, we will use the slice spectral sequence \cite{VoevSSS}  which we now recall. Let $\calK_pn$ be the motivic $\bbP^1$-spectrum representing $K_pn$. That is, there is a bi-graded oriented cohomology theory $\bigoplus_{s,m\in\bbZ}K_pn^{s,m}(A)$ for any $\bbP^1$-spectrum $A$ with  $K_pn^{s,m}(A)=H^{s-m}(A,\calK_pn(m))$, where $(m)$ is smashing with $\Gm^{\wedge m}$. Restricting to the case when $A$ is a smooth variety and $s=2m$ part gives the oriented cohomology  $K_pn^*(A)=\bigoplus_{m\in\bbZ}K_pn^{2m,m}(A)$.  The oriented Borel-Moore homology theory $K_pn(A)$ for not necessarily smooth variety $A$ is obtained as in \cite[\S~4]{Le}. 
Recall that the coefficient ring of $K_pn$ is $\bbZ/p[v^\pm] $ with $v$ in degree $p^n-1$. 
Hence, the slices of  $\calK_pn$ are: 
\begin{equation} 
\label{slice}
s_q(\calK_pn)= 
\left\{
     \begin{array}{lr}
     \Sigma_{\bbP^1}^{l(p^n-1)}H\bbZ/p\otimes (v^l\cdot\bbZ/p ),
        & \hbox{ if $q=l(p^n-1)$ for some $l\in\bbZ$ }; \\
       \hbox{zero otherwise}.
     \end{array}
   \right.
   \end{equation}
   Below we denote $s_q(\calK_pn)$ simply by $s_q$. The slice spectral sequence for any $\bbP^1$-spectrum $A$ and weight $N\in\bbZ$ has 
   \[
   E_2^{h,q}= H^{h-q}(A,s_{-q}(N-q))\Rightarrow K_pn^{h+q,N}(A),\] with the $r$-th differential \begin{equation}\label{eqn:SSS_diff}
d_r^{h,q,N}:E_r^{h,q}\to E_r^{h+r,q-r+1}
\end{equation} for $r\geq 2$.
The source is a subquotient of $H^{h-q}(A,s_{-q}(N-q))$ and the target is a subquotient of $H^{h-q+2r-1}(A,s_{-q+r-1}(N-q+r-1))$.
Recall the following result on  higher Chow groups which can be found in \cite[Vanishing Theorem 3.6 and 19.3]{MVW}. 
\begin{theorem}
 \label{thm: higher Chow}
The higher Chow group $\CH^{m,i}(A,\bbZ/p)\cong H^{2m-i}(A,H\bbZ/p(m))$ vanishes, when $m>\dim A+i$ or $i<0$. 
\end{theorem}
Assuming for the moment $A$ is smooth, we are primarily interested in $K_pn^{s,m}(A)$ in the case when $s=2m,2m-1$. 
\begin{lemma}\label{lem:SSS_deg_sm}
Assume $A$ is a smooth quasi-projective variety with $\dim A< p^n-2$. Then,
\begin{enumerate}
\item  $d_r^{2m-q,q,m}=0$, and $d_r^{2m-q-1,q,m}=0$; In particular, 
\begin{equation*}
K_pn^{2m,m}(A)\cong v^l \CH^{m-l(p^n-1)}(A,\bbZ/p), \end{equation*}
where $l\in\bbZ$ is the unique integer so that $0\leq m-l(p^n-1)\leq \dim A<p^n-1$ if there is, and zero otherwise.
\item 
Similarly, $d_r^{2m-1-q,q,m}=0$, and $d_r^{2m-q-2,q,m}=0$; In particular, \[K_pn^{2m-1,m}(A)\cong v^l \CH^{m-l(p^n-1),1}(A,\bbZ/p), \]
where $l\in\bbZ$ is the unique integer so that $0\leq m-l(p^n-1)\leq \dim A+1<p^n-1$ if there is, and zero otherwise.
\end{enumerate}
\end{lemma}
\begin{proof}
In \eqref{eqn:SSS_diff}, taking $N=m$ and $h+q=2m$, 
the target of the differential $d_r^{h,q,N}$ is a subquotient of $H^{h-q+2r-1}(A,s_{-q+r-1}(N-q+r-1))$, which is in homological degree $2(m-q+r-1)+1$, which is one more than twice the weight $m-q+r-1$, hence vanishes by Theorem \ref{thm: higher Chow}.  Therefore, $d_r^{2m-q,q,m}=0$, i.e., there is no differential from the part converging to $K_pn^{2m,m}(Z)$. 

Now in \eqref{eqn:SSS_diff}, taking $N=m$ and $h+q+1=2m$, 
the slices \eqref{slice} are non-trivial only if $q-r+1=l(p^n-1)$ for some $l$ and $q=l'(p^n-1)$ for some $l'$, in which case \[H^{h-q+2r-1}(A,s_{-q+r-1}(N-q+r-1))=H^{2(m-l(p^n-1))}(A,H\bbZ/p(m-l(p^n-1)))\cong \CH^{m-l(p^n-1)}(A,\bbZ/p)\] and $H^{h-q}(A,s_{-q}(N-q))\cong \CH^{m-l'(p^n-1),1}(A,\bbZ/p)$, a higher Chow group. By Theorem \ref{thm: higher Chow}, these two Chow groups vanish unless  $0\leq m-l(p^n-1)\leq\dim A<p^n-1$ and $0\leq m-l'(p^n-1)\leq\dim A+1<p^n-1$,  which forces $l=l'$. This contradicts with the fact that $r\geq2$. Therefore, $d_r^{2m-q-1,q,m}=0$, i.e., there is no differential to the part converging to $K_pn^{2m,m}(A)$.
Now we  have shown (1).

The same argument with vanishing of the first and second higher Chow groups yields the vanishing of differentials to or from the part converging to $K_pn^{2m-1,m}(A)$.
\end{proof}

The following is a corollary to Lemma~\ref{lem:SSS_deg_sm}.
\begin{lemma}\label{lem:SSS_deg_sing}
Assume $A$ is a quasi-projective variety, not necessarily smooth, embedded into a smooth variety $M$. If $\dim M< p^n-2$, then  Lemma~\ref{lem:SSS_deg_sm}(1) holds for $A$.
\end{lemma}
\begin{proof}
As $A$ naturally embeds into a smooth variety $M$ with open complement $U$, we have the localization exact sequence
\[\cdots\to H^{2m-1}(U,H\bbZ/p(m))\to H^{2m}(M/U,H\bbZ/p(m))\to H^{2m}(M,H\bbZ/p(m))\to \cdots.\]
Lemma~\ref{lem:SSS_deg_sm} applied to $M$ and $U$ implies the vanishing of all possible differentials in the slice spectral sequence converging to $K_pn^{2m,m}(M/U)$, which, by definition, is the Borel-Moore homology $K_pn$ applied to $A$ (\cite[\S~4]{Le}). This proves Lemma~\ref{lem:SSS_deg_sm}(1) for $A$ in this case.
\end{proof}

\begin{proof}[Proof of Proposition~\ref{prop:stable}]
Note that $Z\subseteq X\times X$ and $\dim (X\times X)<p^n-2$. Hence Lemma~\ref{lem:SSS_deg_sing} implies $K_pn(Z)\cong \CH(Z,\bbZ/p)[v^\pm]$ as abelian groups. 
Let $W=p_{12}^{-1}Z\cap p_{23}^{-1}Z\subseteq X^3$, then $Z\times Z\subseteq X^4$ and $W\subset X^3$ both satisfy conditions in Lemma~\ref{lem:SSS_deg_sing}, hence $K_pn(Z\times Z)\cong \CH(Z\times Z,\bbZ/p)[v^\pm]$ and $K_pn(W)\cong \CH(W,\bbZ/p)[v^\pm]$, which is compatible with restriction with support.

It suffices to show that the push-forward $K_pn(W)\to K_pn(Z)$ is compatible with the push-forward $\CH(W,\bbZ/p)[v^\pm]\to \CH(Z,\bbZ/p)[v^\pm]$ under the isomorphisms above. This can be proven through the degree formula  \cite[Theorem~4.4.7]{LM}.  Any homogeneous element $\gamma\in K_pn(W)$ is a Chow cycle shifted by a power of $v$, according to Lemma~\ref{lem:SSS_deg_sing}. Without loss of generality, we may assume $\gamma$ is a Chow cycle of degree-$m$. Hence $\gamma$ lies in $\Omega(W)\otimes_{\Laz}\bbZ/p[v]\subseteq \Omega(W)\otimes_{\Laz}\bbZ/p[v^\pm]\cong K_pn(W)$. Choosing any pre-image  $\widetilde\gamma$ in $\Omega(W)$, \cite[Theorem~4.4.7]{LM} yields that $\widetilde\gamma=\sum_{i=0}^m a_i \gamma_i$ where $a_i\in\Laz_i$ and $\gamma_i$ are cobordism cycles on $W$ which are birational to their images in $W$. In particular, $\dim \gamma_i\leq \dim W<4\dim X<p^n-1$. In particular, the image of $a_i$ in $\bbZ/p[v]$ is zero unless $i=0$, and $a_0\gamma_0$ is the Chow cycle corresponding to $\gamma$ in the following diagram
\[\xymatrix @R=1em{\Omega(W)\ar[r]\ar[dr]&K_pn(W)\ar[d]\\
& \CH(W,\bbZ/p)}\] with the vertical map being an isomorphism in degree-$m$ piece by Lemma~\ref{lem:SSS_deg_sing}. As this diagram is compatible with pushing-forward to $Z$ in the three homology theories involved, we are done. 
\end{proof}

As the characters of the irreducible modules of $U_p(0)$ are given by $E^0_\lambda$ and those of $U_p(\infty)$ are given by $E^\infty_\lambda$, it is natural to expect that those of $U_p(n)$ would be be given by $E^n_\lambda$. However, the cases $n=0$ and $n=\infty$ studied in the present paper does not provide an indication of the right specialization of the loop and quantization parameters in order to get $E^n_\lambda$. 
Further investigations on  the irreducible representations of this family of algebras will be carried out in future publications.

\section{Yangians and shuffle algebras}\label{sec:Yangian}
From this point on we have several miscellaneous sections. 
For any quiver $Q$ without edge-loops, and $A$ is the Chow group with $\bbQ$-coefficients, in this section, we show there is a map from the Yangian  to the (twisted) shuffle algebra.  This map will be further studied in \S~\ref{sec:Yangian_action}.

Throughout this section, we assume the quiver $Q$ has no edge-loops. We assume the $T$-action has the same weights as in Remark \ref{rmk:weights}.

\subsection{The Yangian}
\label{symmetricYangian}
Let $\fg_Q$ be the symmetric Kac-Moody Lie algebra associated to the quiver $Q$.
The  Cartan matrix of $\fg_Q$ is
$C+\overline{C}=(c_{kl})_{k, l\in I}$, which is a symmetric matrix.
Recall that the Yangian of $\fg_Q$, denoted by  $Y_\hbar(\fg_Q)$, is an associative algebra over $\bbQ[\hbar]$, generated by the variables
\[
x_{k, r}^{\pm}, h_{k, r}, (k\in I, r\in \N),
\]
subject to certain relations. 
Let $Y_\hbar^+(\fg_Q)$ be the algebra generated by the elements $x_{k, r}^{+}$, for $k\in I, r\in \N$. 
Define the generating series $x_k^+(u)\in Y_{\hbar}^+(\fg_Q)[\![u^{-1}]\!]$ by
$x_k^+(u)=\hbar \sum_{r\geq 0} x_{k, r}^+ u^{-r-1}. $
The following is a complete set of relations defining $Y_\hbar^+(\fg_Q)$:
\begin{align}
&(u-v-\frac{\hbar c_{kl}}{2}) x_k^+(u)x_l^+(v)
=(u-v+\frac{\hbar c_{kl}}{2}) x_l^+(v)x_k^+(u)
 \tag{Y1}\label{calY1} \\ &\phantom{ABCDEFGABCDEFG}+\hbar\Big(
[x_{k, 0}^+, x_{l}^+(v)]-[x_{k}^+(u), x_{l, 0}^+]
\Big), \text{for any $k, l\in I$}. \notag\\
&
\sum_{\sigma\in\fS_{1-c_{kl}}}[x_k^+(u_{\sigma(1)}),[x_k^+(u_{\sigma(2)}),[\cdots,[x_k^+(u_{\sigma (1-c_{kl})}),x_l^+(v)]\cdots]]]=0, \text{for $k \neq l\in I$}.
\tag{Y2}\label{calY2}
\end{align}

\subsection{The map from Yangian to the shuffle algebra}\label{subsec:YtoSh}
In this subsection, we assume $(R,F)$ is the additive formal group law. We prove the following.
\begin{theorem}\label{thm:Yangian to sh}
Let $(R,F)$ be the additive formal group law. Let $Q$ be any quiver without edge loops.
The assignment
\[Y^+_\hbar(\fg_Q)\ni  x_{k, r}^+ \mapsto (\lambda^{(k)})^{r}\in \calS\calH_{e_k}=R[t_1,t_2][\lambda^{(k)}]\]
 extends to a well-defined algebra homomorphism $Y^+_{\hbar}(\fg_Q)\to\widetilde{\calS\calH}|_{t_1=t_2=\hbar}$.
\end{theorem}
Here recall that $\widetilde{\calS\calH}$ is the shuffle algebra twisted by a sign coming from the Euler-Ringel form (see \S~\ref{subsec:twisting CoHA}).
In order to prove Theorem~\ref{thm:Yangian to sh}, we need to verify  the  relations \eqref{calY1} and \eqref{calY2}  in the algebra $\widetilde{\calS\calH}$. It will take the rest of \S~\ref{subsec:YtoSh}.

For simplicity, we write the multiplication in $\widetilde{\calS\calH}$ as $*$. 
\subsubsection{The quadratic relation \eqref{calY1}}
We now check the relation \eqref{calY1} in the shuffle algebra.
We have
\[
x_k^+(u)\mapsto \hbar \sum_{r\geq 0} (\lambda^{(k)})^r u^{-r-1}=\frac{\hbar}{u-\lambda^{(k)}}, 
\]
where the equality is understood in the usual sense.
To check the  quadratic relation \eqref{calY1}, it suffices to show
\begin{align}
(u-v-\frac{\hbar c_{kl}}{2}) 
&\frac{\hbar}{u-\lambda^{(k)}}
*\frac{\hbar}{v-\lambda^{(l)}}
-(u-v+\frac{\hbar c_{kl}}{2}) 
\frac{\hbar}{v-\lambda^{(l)}}*
\frac{\hbar}{u-\lambda^{(k)}} \label{equ:quad}\\
=&\hbar
\Big(1^{(k)}*\frac{\hbar}{v-\lambda^{(l)}}
-\frac{\hbar}{v-\lambda^{(l)}}* 1^{(k)}
-\frac{\hbar}{u-\lambda^{(k)}}* 1^{(l)}
+ 1^{(l)}*\frac{\hbar}{u-\lambda^{(k)}}
\Big).\notag
\end{align}

We first consider the case when $k\neq l$. We first spell out the formula of the multiplication
$\widetilde{\calS\calH}_{e_k} \otimes \widetilde{\calS\calH}_{e_l}\to \widetilde{\calS\calH}_{e_k+e_l}$ as a map 
$R[\hbar][\lambda^{(k)}]\otimes  R[\hbar][\lambda^{(l)}] \to R[\hbar][\lambda^{(k)}, \lambda^{(l)}]$. Plugging-in $v_1=e_k$, and $v_2=e_l$ to \eqref{equ:fac1},  we have $\fac_1=1$. For simplicity, we write $a=-c_{kl}$, the number of arrows from the vertex $k$ to the vertex $l$. Let $S$ be the set of integers $\{a, a-2, a-4, \dots, -a+4, -a+2\}$. By the running assumption of this section, the 2-dimensional  torus $T$ acts on $T^*\Rep(Q,v)$ as in Remark \ref{rmk:weights}; in particular,  the set $S$ is the set of weights of $T$-action on arrows from the vertex $k$ to the vertex $l$. Plugging the weights into  \eqref{equ:fac2}, we have $\fac_2=\prod_{m\in S}(\lambda^{(l)}-\lambda^{(k)}+m\frac{\hbar}{2})$. Therefore, by the shuffle formula \eqref{shuffle formula}, the Hall multiplication, taken into account of the sign twist of \S~\ref{subsec:twisting CoHA}, is given by
\[
(\lambda^{(k)})^p *(\lambda^{(l)})^q=
-(\lambda^{(k)})^p (\lambda^{(l)})^q \prod_{m\in S}(\lambda^{(l)}-\lambda^{(k)}+m\frac{\hbar}{2}).
\] 
Similarly, the multiplication
$\widetilde{\calS\calH}_{e_l} \otimes \widetilde{\calS\calH}_{e_k}\to \widetilde{\calS\calH}_{e_k+e_l}$
is given by
\begin{align*}
(\lambda^{(l)})^p *(\lambda^{(k)})^q
=&(-1)^{a+1} (\lambda^{(l)})^p (\lambda^{(k)})^q \prod_{m\in S}(\lambda^{(k)}-\lambda^{(l)}+m\frac{\hbar}{2})
=-(\lambda^{(l)})^p (\lambda^{(k)})^q \prod_{m\in S}(\lambda^{(l)}-\lambda^{(k)}-m\frac{\hbar}{2}).
\end{align*}
Plugging  into equation \eqref{equ:quad}, \eqref{calY1} becomes the following identity
\begin{align}
(u-v-\frac{\hbar c_{kl}}{2}) &
\frac{\hbar}{u-\lambda^{(k)}}
\frac{\hbar}{v-\lambda^{(l)}}
\prod_{m\in S}(\lambda^{(l)}-\lambda^{(k)}+m\frac{\hbar}{2}) \notag\\
-(u-v+\frac{\hbar c_{kl}}{2}) &
\frac{\hbar}{v-\lambda^{(l)}}
\frac{\hbar}{u-\lambda^{(k)}}\prod_{m\in S}(\lambda^{(l)}-\lambda^{(k)}-m\frac{\hbar}{2})  \notag\\
=&\hbar
\Big(\frac{\hbar}{v-\lambda^{(l)}}
\prod_{m\in S}(\lambda^{(l)}-\lambda^{(k)}+m\frac{\hbar}{2})
-\frac{\hbar}{v-\lambda^{(l)}}
\prod_{m\in S}(\lambda^{(l)}-\lambda^{(k)}-m\frac{\hbar}{2}) \label{equ*}\\
&-\frac{\hbar}{u-\lambda^{(k)}}\prod_{m\in S}(\lambda^{(l)}-\lambda^{(k)}+m\frac{\hbar}{2})
+\frac{\hbar}{u-\lambda^{(k)}}
\prod_{m\in S}(\lambda^{(l)}-\lambda^{(k)}-m\frac{\hbar}{2})
\Big).  \notag
\end{align}
Canceling the common factor 
\[
\prod_{m\in S'}(\lambda^{(l)}-\lambda^{(k)}+m\frac{\hbar}{2}), 
\text{
where $S'=\{a-2, a-4, \dots, -a+2\}$},
\] the equality \eqref{equ*} becomes
\begin{align*}
&(u-v+\frac{\hbar a}{2}) 
\frac{\lambda^{(l)}-\lambda^{(k)}+a\frac{\hbar}{2}}{(u-\lambda^{(k)})(v-\lambda^{(l)})}-(u-v-\frac{\hbar a}{2}) 
\frac{\lambda^{(l)}-\lambda^{(k)}-a\frac{\hbar}{2}}{(v-\lambda^{(l)})(u-\lambda^{(k)})}
=
a\hbar\Big( \frac{1}{v-\lambda^{(l)}}
-\frac{1}{u-\lambda^{(k)}}
\Big). 
\end{align*}
Both sides of the above identity are equal to $a\hbar\frac{u-v+\lambda^{(l)}-\lambda^{(k)}}{(u-\lambda^{(k)})(v-\lambda^{(l)})}$. 
This shows  the relation \eqref{calY1} for the case when $k\neq l$. 

We now check the relation \eqref{calY1} when $k=l$.
A similar calculation using the shuffle formula \eqref{shuffle formula} shows that equation \eqref{equ:quad} becomes the following identity in $\widetilde{\calS\calH}_{2e_k}
=R[\hbar][\lambda_1, \lambda_2]$
\begin{align*}
&(u-v-\hbar) \sum_{\sigma\in\fS_2}\sigma \Big(
\frac{\hbar}{u-\lambda_1}
\frac{\hbar}{v-\lambda_2}
\frac{\lambda_1-\lambda_2+ \hbar}{\lambda_2-\lambda_1}\Big)
-(u-v+\hbar) \sum_{\sigma\in\fS_2}\sigma
\Big(\frac{\hbar}{v-\lambda_1}
\frac{\hbar}{u-\lambda_2}\frac{\lambda_1-\lambda_2+ \hbar}{\lambda_2-\lambda_1}\Big)\\
&=\hbar\sum_{\sigma\in\fS_2}\sigma
\Bigg(\Big(\frac{\hbar}{v-\lambda_2}
-\frac{\hbar}{v-\lambda_1}
-\frac{\hbar}{u-\lambda_1}
+\frac{\hbar}{u-\lambda_2}\Big)
\frac{\lambda_1-\lambda_2+ \hbar}{\lambda_2-\lambda_1}
\Bigg).
\end{align*} 
It is straightforward to show that both sides of the above identity can be simplified to 
\[\frac{2\hbar^3}{\lambda_2-\lambda_1} \Big(\frac{1}{v-\lambda_2}
-\frac{1}{v-\lambda_1}
-\frac{1}{u-\lambda_1}
+\frac{1}{u-\lambda_2}\Big).\]
This completes the proof of relation \eqref{calY1} for $k=l$.

\subsubsection{The Serre relation \eqref{calY2}}
An argument similar to \cite[\S~10.4]{Nak99} shows that to check the relation \eqref{calY2}, it suffices to check:
\[\sum_{p=0}^{1-c_{kl}} (-1)^p
{1-c_{kl} \choose p} x_{k, 0}^{*p} 
* x_{l, 0} * x_{k, 0}^{*(1-c_{kl}-p)}=0, 
\tag{Y2$'$}\label{calY2'}\]
where $x^{*n}=x*x* \cdots * x$, the shuffle product of $n$-copies of $x$.
We use the shuffle formula \eqref{shuffle formula} to check the Serre relation \eqref{calY2'}.

For any $i,j$, let $\lambda_{i,j}=\lambda_i-\lambda_j$. By the shuffle formula \eqref{shuffle formula}, we have the recurrence relation
\begin{align*}
&(x_{k, 0})^{* n+1}
=(-1)^{n+1}\sum_{\sigma \in\Sh_{(n,1)}} \sigma\Big( (x_{k, 0})^{* n}  \prod_{i=1}^n \frac{\lambda^{(k)}_{i, n+1}+\hbar}{\lambda^{(k)}_{n+1, i}}\Big).
\end{align*}
Therefore, inductively, we get a formula of $(x_{k, 0})^{* n}$:
\begin{align}
(x_{k, 0})^{* n}
=&(-1)^{\frac{n(n+1)}{2}-1} \sum_{\sigma\in \fS_n}
\sigma\left(
 \frac{\lambda^{(k)}_{12}+\hbar}{\lambda^{(k)}_{21}}\cdot
 \frac{\lambda^{(k)}_{13}+\hbar}{\lambda^{(k)}_{31}}
 \cdots
 \frac{\lambda^{(k)}_{n-1, n}+\hbar}{\lambda^{(k)}_{n, n-1}}\right). \label{1*}
\end{align}

Note that $k\neq l$. By the shuffle formula \eqref{shuffle formula}, the multiplication
$\widetilde{\calS\calH}_{n e_k} \otimes \widetilde{\calS\calH}_{e_l}\to\widetilde{\calS\calH}_{n e_k+e_l}
$ is given by
\begin{align}
&(x_{k, 0})^{* n}*x_{l, 0}
=-(x_{k, 0})^{* n} \prod_{i=1}^n \prod_{m\in S}
(\lambda^{(l)}-\lambda^{(k)}_i+m\frac{\hbar}{2}), \label{2*}
\end{align} where $S=\{a, a-2, a-4, \dots, -a+4, -a+2\}$.

For the multiplication $\widetilde{\calS\calH}_{p e_k+ e_l} \otimes \widetilde{\calS\calH}_{q e_k}\to \widetilde{\calS\calH}_{e_l+(p+q) e_k}$, considered as a map \[
R[t_1, t_2 ][\lambda^{(k)}_1, \cdots,\lambda^{(k)}_p, \lambda^{(l)}_{1} ] \otimes R[t_1, t_2 ][\lambda^{(k)}_{p+1}, \cdots,\lambda^{(k)}_{p+q}]  \to 
R[t_1, t_2 ][\lambda^{(k)}_1, \cdots,\lambda^{(k)}_{p+q}, \lambda^{(l)}_{1}],\] we have
\begin{align}
((x_{k, 0})^{* p}*x_{l, 0})*(x_{k, 0})&^{* q}
=(-1)^{pq+qa+1}\sum_{\sigma \in \Sh_{(p, q)}}\sigma \Big((x_{k, 0})^{* p}*x_{l, 0}\cdot  (x_{k, 0})^{* q} \cdot \notag \\
&\cdot  \prod_{s=1}^p \prod_{t=p+1}^{p+q}
\frac{\lambda^{(k)}_s-\lambda^{(k)}_t+\hbar }{\lambda^{(k)}_t- \lambda^{(k)}_s} 
\cdot \prod_{t=p+1}^{p+q} \prod_{m\in S}
(\lambda^{(k)}_t -\lambda^{(l)}+m\frac{\hbar}{2})\Big).
 \label{3*}
\end{align}

Plugging the formulas of \eqref{1*} \eqref{2*} into  \eqref{3*} with $q=a+1-p$, we get
\begin{align*}
%------Serre relation: line 1------
x_{k, 0}^{*p} 
* x_{l, 0} * x_{k, 0}^{*(a+1-p)}&
=(-1)^{\frac{(a+1)(a+2)}{2}} \sum_{\pi\in \Sh_{(p, a+1-p)}}\pi \Bigg( 
%---- copy the computation to here:
\Big( \sum_{\sigma\in \fS_p} \sigma\cdot \prod_{1\leq i<j \leq p}
\frac{\lambda^{(k)}_{i, j}+\hbar}{\lambda^{(k)}_{j, i}}\Big)\\
&\cdot \Big( \sum_{\sigma\in \fS_{a+1-p}} \sigma \cdot \prod_{\{p+1\leq i<j \leq a+1\}}
\frac{\lambda^{(k)}_{i, j}+\hbar}{\lambda^{(k)}_{j, i}}\Big)
 %---line 2---
\Big(\prod_{s=1}^p \prod_{t=p+1}^{a+1}
\frac{\lambda^{(k)}_{s, t}+\hbar }{\lambda^{(k)}_{t, s}}\Big)\\
&\cdot 
\prod_{m\in S}\Big( \prod_{i=1}^p(\lambda^{(k)}_i-\lambda^{(l)}-m\frac{\hbar}{2}) 
\prod_{t=p+1}^{a+1} 
(\lambda^{(k)}_t -\lambda^{(l)}+m\frac{\hbar}{2})\Big)\Bigg).
\end{align*}
Re-arranging the above summation, we have:
\begin{align*}
\sum_{p=0}^{a+1} (-1)^p 
{a+1 \choose p} 
x_{k, 0}^{*p} 
* x_{l, 0} * & x_{k, 0}^{*(a+1-p)}
=\sum_{p=0}^{a+1} (-1)^p
{a+1 \choose p} \sum_{\sigma\in \fS_{a+1}}
\Big( \prod_{1\leq i<j \leq a+1}
\frac{\lambda^{(k)}_{\sigma(i), \sigma(j)}+\hbar}{\lambda^{(k)}_{\sigma(j), \sigma(i)}}\cdot 
\\
%---- copy the computation to here:
&
\cdot \prod_{m\in S}\Big( \prod_{i=1}^p(\lambda^{(k)}_{\sigma(i)}-\lambda^{(l)}-m\frac{\hbar}{2}) 
\prod_{t=p+1}^{a+1} 
(\lambda^{(k)}_{\sigma(t)} -\lambda^{(l)}+m\frac{\hbar}{2})\Big)\Big).
\end{align*}

Note that the factor
\[
\prod_{m\in S'}\Big( \prod_{i=1}^p(\lambda^{(k)}_{\sigma(i)}-\lambda^{(l)}-m\frac{\hbar}{2}) 
\prod_{t=p+1}^{a+1} 
(\lambda^{(k)}_{\sigma(t)} -\lambda^{(l)}+m\frac{\hbar}{2})\Big)
=\prod_{m\in S'}\prod_{i=1}^{a+1}
(\lambda^{(k)}_{i}-\lambda^{(l)}-m\frac{\hbar}{2}),
\]  is independent of $\sigma\in \mathfrak{S}_{a+1}$, hence a common factor.  Here again $S'=\{a-2, a-4, \dots, -a+2\}$.
Let $\lambda'^{(k)}_i=\lambda^{(k)}_i-\lambda^{(l)}$. 
After canceling the above common factor, to show the Serre relation \eqref{calY2'}, it suffices to show
\begin{align}
&\sum_{p=0}^{a+1} (-1)^p
{a+1 \choose p} \sum_{\sigma\in \fS_{a+1}}
 \sigma\left( 
%---- copy the computation to here:
 \prod_{s=1}^p(\lambda^{(k)}_{s}-\frac{a\hbar}{2}) 
\prod_{t=p+1}^{a+1} 
(\lambda^{(k)}_{t}+\frac{a\hbar}{2})\prod_{1\leq i<j \leq a+1}
\frac{\lambda^{(k)}_{i, j}+
\hbar}{\lambda^{(k)}_{j, i}} \right)=0.\label{simply:serre}
\end{align}
The identity \eqref{simply:serre} is proved in Appendix, Corollary~\ref{appendix cor}. As a consequence, the map  $Y^+_{\hbar}(\fg)\to\widetilde{\calS\calH}|_{t_1=t_2=\hbar}$ respects the Serre relation \eqref{calY2}.

In Theorem~\ref{thm:sym_poly_id}, we have a more general identity of symmetric functions, which deforms the identify \eqref{simply:serre}. It might be interesting on its own rights. Therefore, we prove it in Appendix~\ref{app:sym_poly} and deduce \eqref{simply:serre} from it.

\section{Comparison with Yangian actions on quiver varieties}\label{sec:Yangian_action}

In this section, we show that the (twisted) spherical subalgebra $\calP^{\sph, \ext}$ of preprojective CoHA defined in \S~\ref{preproj CoHA} specializes to the Borel of the Yangian  when $A$ is the intersection theory.
Again  we assume $Q$ is a quiver without edge-loops, and the $T$-action has the same weights as in Remark \ref{rmk:weights}.

%\subsection{Comparison of the preprojective CoHA with Varagnolo Yangian action}
Now we take the oriented Borel-Moore homology theory to be the intersection theory $\CH$.
Recall that Varagnolo in \cite{Va00} constructed representations of the Yangians using quiver variety. It is proved that, for each $w\in\bbN^I$, there is an algebra homomorphism $a^Y:Y_\hbar(\fg)\to \End(\CH_{G_w\times\Gm}(\fM(w)))$.  
The action of generator $x_{k, r}^{+}$ is given by
\[
x_{k, r}^{+}\mapsto \sum_{v_2}(-1)^{(e_k, \overline C v_2)} \Delta_{*}^{+}(c_1(\calL_k))^r \in \CH_{G_w\times T} (Z(w))\to \End(\calM(w)), 
\]
where
\[
\Delta^+: C_{k}^+(v_2, w) \inj Z(v_2-e_k, v_2, w)
\] is the natural embedding of the irreducible component. 

Observe that according to the projection formula and \eqref{eq:action a_ek}, we have $a^Y(x_{k, r}^{+})=\widetilde{a}\big((z^{(k)})^l\big)\in \End(\calM(w))$. 

Assume the quiver $Q$ is $ADE$ type. For $w\in\bbN^I$, we call the action map
$Y_\hbar(\fg) \to \End(\calM(w))$ $a_w^Y$ to emphasize the dependence on $w$.
\begin{lemma}\label{lem:cap}
\footnote{We thank Sachin Gautam for explaining to us the proof of \cite[Proposition A.8]{GTL}.}
With notations as above, we have $
\bigcap_{w} \ker(a_w^Y)=0.$
\end{lemma}
\begin{proof}
In \cite{Nak12}, Nakajima proved that 
for any $w_1,w_2\in\bbN^I$, the kernel of the map  $Y_\hbar(\fg)\to \End(\calM(w_1)\otimes\calM(w_2))$ is contained in the kernel of  $a^Y_{w'}$ for some $w'\in\bbN^I$. Therefore, the ideal $\bigcap_{w} \ker(a_w^Y)$ is a $\bbQ[\hbar]$-flat Hopf ideal in $Y_{\hbar}(\fg)$. 
By \cite[Proposition A.8]{GTL}, if $Q$ is of finite Dynkin type, there is no non-trivial such ideal in $Y_{\hbar}(\fg)$. Therefore, $\bigcap_{w} \ker(a_w^Y)=0$. 
\end{proof}

The following is a direct consequence of Lemma \ref{lem:cap}.
\begin{corollary}\label{cor:yang_compatible}
Assume $Q$ is a quiver of $ADE$ type. 
The assignment $(z^{(k)})^l\mapsto x_{k, l}^+$, $t_1,t_2\mapsto\hbar/2$ extends to a well-defined surjective algebra homomorphism
$\Upsilon: \null^{\CH}\widetilde{\calP}^{\sph}\surj Y^+_\hbar(\fg)$. 
Moreover, the following diagram commutes
\[\xymatrix@R=1.5em{
^{\CH}{\widetilde{\calP}^{\sph}}
\ar[r]^(.3){\widetilde{a}}\ar[d]_{\Upsilon}&\End(\CH_{G_w\times\Gm}(\fM(w)))\\
Y^+_\hbar(\fg)\ar@{^{(}->}[r]&Y_\hbar(\fg)\ar[u]_{a^Y}.
}\]
\end{corollary}

For an $ADE$ type quiver $Q$,  summarizing Lemma~\ref{lem:twistedCoHA}, Theorem~\ref{thm:Yangian to sh}, and Corollary~\ref{cor:yang_compatible}, we have a commutative diagram of algebras
\[
\xymatrix@R=1.5em{
^{\CH}\widetilde{\calP}^{\sph}\ar@{->>}[d]^{\Upsilon} \ar[r] &\widetilde{\calS\calH}|_{t_1=t_2=\hbar/2}\\
Y^+_\hbar(\fg)\ar[ur]&
}
\]

For any $x\in ^{\CH}{\widetilde{\calP}^{\sph}}|_{t_1=t_2=\hbar/2}$ such that $\Upsilon(x)=0$, then $x$ lies in the kernel of the map $^{\CH}{\widetilde{\calP}^{\sph}}|_{t_1=t_2=\hbar/2} \to \widetilde{\calS\calH}^{\sph}|_{t_1=t_2=\hbar/2}$. 
We know this map is an isomorphism after localization in the sense of Remark~\ref{rmk:torsion}. Therefore, $x$ is a torsion element in $^{\CH}{\widetilde{\calP}^{\sph}}|_{t_1=t_2=\hbar/2}$.

Define $^{\CH}\underline{\widetilde{\calP}^{\sph}}$ to be the quotient of $^{\CH}\widetilde{\calP}^{\sph}|_{t_1=t_2=\hbar/2}$ by the torsion part in the same sense as in Remark~\ref{rmk:torsion}. 
\begin{theorem}\label{thm:Yangian}
Assume $Q$ is a quiver of $ADE$ type. 
We have the following isomorphism
\[
\Upsilon^{-1}: Y^+_\hbar(\fg)\cong \null^{\CH}\underline{\widetilde{\calP}^{\sph}},
\]
such that the diagram
\[
\xymatrix@R=1.5em{Y^+_\hbar(\fg)\ar[d]_{\Upsilon^{-1}}\ar@{^{(}->}[r]&Y_\hbar(\fg)\ar[d]_{a^Y}\\
^{\CH}\underline{\widetilde{\calP}^{\sph}}
\ar[r]^(.3){\widetilde{a}}&\End(\CH_{G_w\times\Gm}(\fM(w)))
} \]
commutes. Here the action $a^Y$ is given by Varagnolo in \cite{Va00}. 
\end{theorem}

\appendix
\section{A symmetric polynomial identity}\label{app:sym_poly}
Let $\mathbb{S}(n, b, \hbar)\in\bbQ[b,\hbar](\lambda_1,\dots,\lambda_n)^{\fS_n}$ be 
\[
\mathbb{S}(n, b, \hbar):=
\sum_{\sigma\in \fS_{n}}
\sigma\cdot
\Bigg(
\sum_{p=0}^{n}(-1)^p{n \choose p}
\prod_{i=1}^p(\lambda_i-b\hbar)
\prod_{j=p+1}^{n}(\lambda_j+b\hbar)
\prod_{1\leq i<j \leq n}\frac{\lambda_{ij}+\hbar}{\lambda_{ji}}\Bigg).
\]
Here $\lambda_{i,j}=\lambda_i-\lambda_j$ for any $i,j$.

\begin{theorem}\label{thm:sym_poly_id}
\begin{enumerate}
\item The element $\mathbb{S}(n, b, \hbar)$ lies in $\Q[b, \hbar]$. In other words, 
  $\mathbb{S}(n, b, \hbar)$ does not depend on
  the variables $\{ \lambda_i, i=1, \dots, n\}$. 
  \item We have $\mathbb{S}(1, b, \hbar)=2\hbar b$, and the recursive formula of $\mathbb{S}(n, b, \hbar)$
  \[
  \mathbb{S}(n, b, \hbar)=
\mathbb{S}(n-1, b, \hbar)\cdot\big((-1)^{n+1} 2\hbar n(b -\frac{n-1}{2})\big).
  \]
  \end{enumerate}
\end{theorem}
As a direct consequence, we get the following.
\begin{corollary}\label{appendix cor}
When $b=\frac{n-1}{2}$, we have $\mathbb{S}(n, \frac{n-1}{2}, \hbar)=0$.
In particular, the identity \eqref{simply:serre} holds.
\end{corollary}
\begin{proof}[Proof of Theorem \ref{thm:sym_poly_id}]
Using the equality: 
\[
 {n \choose k}={ n-1 \choose k}+ {n-1 \choose k-1},\,\  \text{for $1\leq k\leq n-1.$ }
 \]
 We have
 \begin{align}
 \mathbb{S}(n, b, \hbar)
 =&\sum_{\sigma\in \fS_{n}}
\sigma\cdot \Bigg(\sum_{p=0}^{n-1} (-1)^p {n-1 \choose p}\prod_{i=1}^{p}
 (\lambda_i-b\hbar)
\prod_{j=p+1}^{n-1}(\lambda_j+b\hbar)(\lambda_n+b\hbar)
\prod_{1\leq i<j \leq n}\frac{\lambda_{ij}+\hbar}{\lambda_{ji}}\Bigg)
\notag \\
&- \sum_{\sigma\in \fS_{n}}
\sigma\cdot \Bigg((\lambda_1-b\hbar)\sum_{p=0}^{n-1} (-1)^p {n-1 \choose p}\prod_{i=2}^{p+1}
 (\lambda_i-b\hbar)
\prod_{j=p+2}^{n}(\lambda_j+b\hbar)
\prod_{1\leq i<j \leq n}\frac{\lambda_{ij}+\hbar}{\lambda_{ji}}\Bigg) \notag \\
=&\sum_{\sigma\in \Sh_{([1, n-1], n)}}
\sigma\cdot \Big( \mathbb{S}(n-1, b, \hbar)
(\lambda_n+b\hbar)
\prod_{1\leq i\leq  n-1}\frac{\lambda_{in}+\hbar}{\lambda_{ni}}\Big)
\notag \\
&-\sum_{\sigma\in \Sh_{(1, [2, n])}}
\sigma\cdot \Big( \mathbb{S}(n-1, b, \hbar)
(\lambda_1-b\hbar)
\prod_{2\leq j\leq  n}\frac{\lambda_{1j}+\hbar}{\lambda_{j1}}\Big)
\notag \\
=&\mathbb{S}(n-1, b, \hbar)
\left(
\sum_{j=1}^n
\Big(
(\lambda_j+b\hbar)
\prod_{\{i: 1\leq i\leq  n, i\neq j\}}\frac{\lambda_{ij}+\hbar}{\lambda_{ji}}\Big)
-\sum_{j=1}^n\Big(
(\lambda_j-b\hbar)
\prod_{\{i: 1\leq i\leq  n, i\neq j\}}\frac{\lambda_{ji}+\hbar}{\lambda_{ij}}\Big)\right).\label{line for S}
 \end{align}
Here the last equality follows from the induction hypothesis that 
 $\mathbb{S}(n-1, b, \hbar)$ only depends on $n-1, b$ and $\hbar$.
 
It remains to compute the right hand side of \eqref{line for S}.  
Define $
 F(t, b, \hbar):=\frac{1}{\hbar}(t+b\hbar) \prod_{i=1}^n
 \frac{\lambda_i-t+\hbar}{t-\lambda_i}$. 
 The function $ F(t, b, \hbar)$ has only simple poles at $t=\lambda_j$, for $j=1, \dots, n$, and
 \[
 \sum_{j=1}^n \Res_{t=\lambda_j} F(t, b, \hbar)
 =\sum_{j=1}^n
\Big(
(\lambda_j+b\hbar)
\prod_{\{i: 1\leq i\leq  n, i\neq j\}}\frac{\lambda_{ij}+\hbar}{\lambda_{ji}}\Big).
 \]
By the residue theorem, $\sum_{j=1}^n \Res_{t=\lambda_j} F(t, b, \hbar)$ is equal to
\begin{align*}
-\Res_{t=\infty} F(t, b, \hbar)
=&-\Res_{t=\infty}\frac{1}{\hbar}(t+b\hbar)\prod_{i=1}^n(-1+\frac{\hbar}{t-\lambda_i})
\Omit{=&-\Res_{t=\infty}
\frac{1}{\hbar}(t+b\hbar)
\prod_{i=1}^n
(-1+ \hbar\sum_{m\geq 1} \frac{\lambda_i^{m-1}}{t^m})\\}
=-(-1)^n \Big(nb\hbar +
\sum_{i=1}^n \lambda_i -\hbar {n\choose 2}\Big).
\end{align*}
Therefore
\begin{equation}\label{equ:sum1}
\sum_{j=1}^n
\Big(
(\lambda_j+b\hbar)
\prod_{\{i: 1\leq i\leq  n, i\neq j\}}\frac{\lambda_{ij}+\hbar}{\lambda_{ji}}\Big)
=-(-1)^n \Big(nb\hbar +
\sum_{i=1}^n \lambda_i -\hbar {n\choose 2}\Big).
\end{equation}
Similarly, we also have
\Omit{set:
 \[
 G(t, b, \hbar):=
 \frac{-1}{\hbar}(t+b\hbar) 
 \prod_{i=1}^n
 \frac{t-\lambda_i+\hbar}{\lambda_i-t}.
 \]
 The function $ G(t, b, \hbar)$ has only simple poles at 
 $t=\lambda_j$, for $j=1, \dots, n$. 
 And the sum of the residues at those simple poles is:
 \[
 \sum_{j=1}^n \Res_{t=\lambda_j} G(t, b, \hbar)
 =\Bigg(
(\lambda_j-b\hbar)
\prod_{\{i: 1\leq i\leq  n, i\neq j\}}
\frac{\lambda_{ji}+\hbar}{\lambda_{ij}}\Bigg)
 \]
By the residue theorem, it is the same as
\begin{align*}
&-\Res_{t=\infty} G(t, b, \hbar)\\
=&-\Res_{t=\infty}
\frac{-1}{\hbar}
(t-b\hbar)\prod_{i=1}^n
(-1-\frac{\hbar}{t-\lambda_i})\\
=&-\Res_{t=\infty}
\frac{-1}{\hbar}(t-b\hbar)
\prod_{i=1}^n
(-1- \hbar\sum_{m\geq 1} \frac{\lambda_i^{m-1}}{t^m})\\
=&-(-1)^n \Big(-nb\hbar +\sum_{i=1}^n \lambda_i +\hbar {n\choose 2}\Big).
\end{align*}
Therefore,}
\begin{equation}\label{equ:sum2}
\sum_{j=1}^n\Big(
(\lambda_j-b\hbar)
\prod_{\{i: 1\leq i\leq  n, i\neq j\}}
\frac{\lambda_{ji}+\hbar}{\lambda_{ij}}\Big)
=-(-1)^n \Big(-nb\hbar +\sum_{i=1}^n \lambda_i +\hbar {n\choose 2}\Big).
\end{equation}
Plugging the equalities \eqref{equ:sum1} and \eqref{equ:sum2} into \eqref{line for S}, we get:
\begin{align*}
\mathbb{S}(n, b, \hbar)
=&
\mathbb{S}(n-1, b, \hbar)
\Big(
-(-1)^n \Big(nb\hbar +\sum_{i=1}^n \lambda_i -\hbar {n\choose 2}\Big)+
(-1)^n \Big(-nb\hbar +\sum_{i=1}^n \lambda_i +\hbar {n\choose 2}\Big) \Big)\\
\Omit{=&\mathbb{S}(n-1, b, \hbar)(-1)^{n+1} 2\hbar (nb - {n\choose 2})\\}
=&\mathbb{S}(n-1, b, \hbar)(-1)^{n+1} 2\hbar n(b -\frac{n-1}{2}).
 \end{align*}
This completes the proof.   
\end{proof}

\newcommand{\arxiv}[1]
{\texttt{\href{http://arxiv.org/abs/#1}{arXiv:#1}}}
\newcommand{\doi}[1]
{\texttt{\href{http://dx.doi.org/#1}{doi:#1}}}
\renewcommand{\MR}[1]
{\href{http://www.ams.org/mathscinet-getitem?mr=#1}{MR#1}}

\end{document}